\documentclass[12pt]{article}
\usepackage{amsmath}
\usepackage{latexsym}
\usepackage{amssymb}
\newtheorem{thm}{Theorem}[section]
\newtheorem{la}[thm]{Lemma}
\newtheorem{Defn}[thm]{Definition}
\newtheorem{Remark}[thm]{Remark}
\newtheorem{Note}[thm]{Note}
\newtheorem{prop}[thm]{Proposition}

\newtheorem{cor}[thm]{Corollary}
\newtheorem{Example}[thm]{Example}
\newtheorem{Examples}[thm]{Examples}
\newtheorem{Problems}[thm]{Problems}

\newtheorem{Problem}[thm]{Problem}
\newtheorem{Convention}[thm]{Convention}
\newtheorem{Number}[thm]{\!\!}
\newenvironment{defn}{\begin{Defn}\rm}{\end{Defn}}

\newenvironment{example}{\begin{Example}\rm}{\end{Example}}

\newenvironment{rem}{\begin{Remark}\rm}{\end{Remark}}
\newenvironment{numba}{\begin{Number}\rm}{\end{Number}}
\newenvironment{proof}{{\noindent\bf Proof.}}%
                  {\nopagebreak\hspace*{\fill}$\Box$\medskip\medskip\par}   
\newcommand{\Punkt}{\nopagebreak\hspace*{\fill}$\Box$}
\newcommand{\wb}{\overline}
\newcommand{\ve}{\varepsilon}

\newcommand{\wt}{\widetilde}

\newcommand{\impl}{\Rightarrow}

\newcommand{\mto}{\mapsto}

\newcommand{\isom}{\cong}
\DeclareMathOperator{\Ad}{Ad}
\newcommand{\N}{{\mathbb N}}
\newcommand{\R}{{\mathbb R}}

\newcommand{\K}{{\mathbb K}}

\newcommand{\C}{{\mathbb C}}

\newcommand{\cO}{{\cal O}}
\newcommand{\cV}{{\cal V}}
\newcommand{\cW}{{\cal W}}
\newcommand{\cg}{{\mathfrak g}}
\newcommand{\ch}{{\mathfrak h}}

\DeclareMathOperator{\Diff}{Diff}
\newcommand{\dl}{{\displaystyle \lim_{\longrightarrow}}}
\newcommand{\pl}{{\displaystyle \lim_{\longleftarrow}}}

\DeclareMathOperator{\Aut}{Aut}

\newcommand{\sub}{\subseteq}

\DeclareMathOperator{\im}{im}

\DeclareMathOperator{\pr}{pr}

\DeclareMathOperator{\id}{id}

\newcommand{\cS}{{\cal S}}

\DeclareMathOperator{\evol}{evol}
\DeclareMathOperator{\Evol}{Evol}

\DeclareMathOperator{\ev}{ev}

\DeclareMathOperator{\Germ}{Germ}
\DeclareMathOperator{\Gau}{Gau}

\begin{document}
$\;$\\[-27mm]
\begin{center}
{\Large\bf Regularity properties\vspace{2mm} of
infinite-dimensional Lie groups, and semiregularity}\\[7mm]
{\bf Helge Gl\"{o}ckner}\vspace{4mm}
\end{center}
\begin{abstract}
\noindent
Let $G$ be a Lie group modelled on a locally convex space,
with Lie algebra $\cg$, and $k\in \N_0\cup\{\infty\}$.
We say that $G$ is \emph{$C^k$-semiregular}
if each $\gamma\in C^k([0,1],\cg)$ admits a left evolution
$\Evol(\gamma)\in C^{k+1}([0,1],G)$.
If, moreover, the map $\evol\colon C^k([0,1],\cg)\to G$,
$\gamma\mto\Evol(\gamma)(1)$ is smooth,
then $G$ is called \emph{$C^k$-regular}.
For~$G$ a $C^k$-semiregular Lie group
and $m\in \N_0\cup\{\infty\}$,
we show that $\evol\colon C^k([0,1],\cg)\to G$
is~$C^m$ if and only~if
\[
\Evol\colon C^k([0,1],\cg)\to
C^{k+1}([0,1],G)
\]
is~$C^m$.
If $\evol$ is continuous at~$0$, then
$\evol$ is continuous.
If $G$ is a $C^0$-semiregular Lie group, then
continuity of $\evol$ implies its smoothness
(so that $G$ will be $C^0$-regular),
if smooth
homomorphisms from~$G$ to $C^0$-regular
Lie groups separate points on~$G$ and $\cg$ is (e.g.)\ sequentially complete.
Further criteria for regularity properties
are provided, and used to prove regularity
for several important classes of Lie groups.
In particular, we show that $\Diff(M)$ is $C^1$-regular for each
paracompact finite-dimensional smooth manifold~$M$
(which need not be $\sigma$-compact).
We also provide tools which enable
the proof of $C^1$-regularity for the
Lie group $\Diff^\omega(M)$
of analytic diffeomorphisms
of a real analytic compact manifold~$M$.\vspace{4mm}
\end{abstract}
\noindent
{\bf\large Introduction and statement of the results}\\[4mm]
Let $G$ be a Lie group modelled on a locally convex space
(with neutral\linebreak
element~$e$ and Lie algebra $\cg:=L(G):=T_e(G)$),
as in \cite{RES}, \cite{GaN}, and \cite{SUR}
(cf.\ \cite{Mil} for the case of sequentially complete
modelling spaces).
Given $g\in G$,
we use the tangent map of the left translation
$\lambda_g\colon G\to G$, $x\mto gx$
to define
\[
g.v:=(T_x\lambda_g)(v)\in T_{gx}(G)\quad \mbox{for $x\in G$,
$v\in T_x(G)$.}
\]
Let $k\in \N_0\cup\{\infty\}$. We say that a Lie group $G$ is
\emph{$C^k$-semiregular} if the initial value problem
\[
\eta'(t)\, =\, \eta(t).\gamma(t),\quad \eta(0)=e
\]
has a (necessarily unique) solution $\Evol(\gamma):=\eta\in C^{k+1}([0,1],G)$
for each $\gamma\in C^k([0,1],\cg)$.
If~$G$ is $C^k$-semiregular and the map
\[
\evol\colon C^k([0,1],\cg)\to G,\quad \gamma\mto \eta(1)=\Evol(\gamma)(1)
\]
is smooth, then the Lie group~$G$ is called \emph{$C^k$-regular}
(compare \cite{SUR}; cf.\ also \cite{GDL},
where this property is referred to as
\emph{strong} $C^k$-regularity).
Let
\[
C_*^{k+1}([0,1],G):=\{\gamma\in C^{k+1}([0,1],G)\colon \gamma(0)=e\},
\]
endowed with
the Lie group structure recalled in~\ref{defmastar}.
For $G$ a $C^k$-semiregular Lie group and $m\in \N_0\cup\{\infty\}$,
we show that $\evol\colon C^k([0,1],\cg)\to G$
is $C^m$ if and only if
$\Evol\colon C^k([0,1],\cg)\to C^{k+1}([0,1],G)$ is $C^m$
(see Lemma~\ref{evolEvol}).
This implies:\\[4mm]
{\bf Theorem A}
\emph{If $G$ is a $C^k$-regular Lie group, then the map}
\[
\Evol\colon C^k([0,1],\cg)\to C^{k+1}_*([0,1],G)
\]
\emph{is a $C^\infty$-diffeomorphism.
In particular,
$\Evol\colon C^k([0,1],\cg)\to C^{k+1}([0,1],G)$
is a smooth mapping.}\\[4mm]
A Lie group is called \emph{regular} if it is
$C^\infty$-regular (see \cite{SUR};
cf.\ \cite{Mil} for the sequentially complete case,
and \cite{KMR} for corresponding notions
in the convenient setting of analysis).\footnote{Compare
also \cite{OMO} for an earlier
concept of regularity in Fr\'{e}chet-Lie groups,
which is stronger than $C^0$-regularity.}
The special case $k=\infty$ of Theorem~A
(i.e., the special case of regular Lie groups)
is known; see \cite[Lemma~A.5\,(1)]{NaW}.
We recall that the modelling space of a regular Lie group
is necessarily Mackey complete
(see \cite[Remark~II.5.3\,(b)]{SUR}).
We refer to \cite[2.14]{KaM}
for the definition of Mackey completeness.
Let $r\in \N\cup\{\infty\}$.
It is known that a locally convex space $E$
is Mackey complete if and only if the weak integral
$\int_0^1\gamma(t)\,dt$ exists in~$E$
for each $C^r$-curve $\gamma\colon [0,1]\to E$
(cf.\ \cite[2.14]{KaM}).
The additive group of a locally convex space~$E$ is
regular if and only if~$E$ is Mackey
complete~\cite[Proposition~II.5.6]{SUR}.\\[3mm]
Let us say that a Lie group is \emph{$1$-connected}
if it is connected and simply connected.
Consider the following statements:\vspace{1mm}
\begin{itemize}
\item[(a)]
If $H$ is a $1$-connected Lie group,
$G$ a Lie group
and $\phi\colon L(H)\to L(G)$ a continuous Lie algebra
homomorphism, then there is a smooth group homomorphism
$\Phi\colon H\to G$ with $L(\Phi)=\phi$
(where $L(\Phi):=T_1(\phi)$).
\item[(b)]
If $G$ and $H$ are $1$-connected Lie groups
such that $L(G)\cong L(H)$ as a topological Lie algebra,
then $G\cong H$ as a Lie group.
\item[(c)]
If $G$ is a
$1$-connected abelian Lie group,
then $G\isom E/\Gamma$ for the additive group $(E,+)$
of some locally convex space~$E$
and some discrete additive subgroup $\Gamma\sub E$.
\item[(d)]
For each $v\in L(G)$, there is a smooth homomorphism
$\gamma_v\colon (\R,+)\to G$ such that $\gamma_v'(0)=v$.\vspace{2mm}
\end{itemize}
It is well-known that (a) holds if $G$ is regular,
whence (b) holds if $G$ and $H$ are regular (see
Theorem III.1.5 and Corollary III.1.6
in \cite{SUR};
cf.\ \cite{Mil}
for the sequentially complete case).
If $G$ is regular, the preceding implies (c)
(as in \cite[Remark~3.13]{NAB};
cf.\ \cite{MaT} for an analogous result in the convenient
setting).
Finally, (d) is a well-known consequence of regularity of $G$
(because $\gamma_v=\Evol(t\mto v)$;
see \cite[Remark~II.5.3\,(a)]{SUR} or \cite{Mil}).
As our second result,
we show that all of these familiar facts become
false for suitable non-regular Lie groups modelled
on non-Mackey complete spaces.\\[4mm]
{\bf Theorem B.} \emph{All of the preceding statements} (a), (b), (c) \emph{and} (d)
\emph{are false for suitable examples of $1$-connected abelian
Lie groups~$G$ and~$H$
modelled on non-Mackey complete spaces,
and the counterexample for} (d) \emph{can be chosen such that
$L(G)\not=\{0\}$ and $\gamma_v$ exists for no non-zero
vector $v\in L(G)$.
Moreover, there is a connected infinite-dimensional Lie group~$K_1$
such that the only group homomorphism $(\R,+)\to K_1$
is the trivial map $t\mto e$.}\\[4mm]
The negative answer to (b)
answers a problem by Neeb
\cite[Problem II.3]{SUR}. Whether Mackey-completeness
or sequential completeness of the modelling spaces
suffices for
a positive answer, without the assumption of regularity
(as originally asked by Milnor) remains unknown.\\[2.4mm]
Let us say that a locally convex space $E$ is \emph{integral complete}
if the weak integral $\int_0^1\gamma(t)\,dt$ exists in~$E$,
for each continuous path $\gamma\colon [0,1]\to E$.
It is known that this property is equivalent to the metric convex compactness
property, meaning that the closed convex hull
of every metrizable compact subset $K\sub E$ is compact \cite{Wei}.
Varying Remark~II.5.3\,(b) and
Proposition~II.5.6
from \cite{SUR}, we also show:\\[4mm]
{\bf Theorem~C.}
\begin{itemize}
\item[(a)]
\emph{Each $C^0$-regular Lie group
has an integral complete modelling space.}
\item[(b)]
\emph{The additive group of a locally convex space~$E$ is a $C^0$-regular
Lie group if and only if~$E$ is integral complete.}
\item[(c)]
\emph{The additive group of a locally convex space~$E$ is a $C^1$-regular
Lie group if and only if~$E$ is Mackey complete.}
\item[(d)]
\emph{There is a Lie group which is $C^1$-regular
but not $C^0$-regular.}\vspace{2mm}
\end{itemize}
Usually, the $C^k$-regularity of a given
Lie group~$G$ is established in two steps:
In the first step, one shows that
each $\gamma\in C^k([0,1],\cg)$
has an evolution, i.e., that $G$ is $C^k$-semiregular.
In the second step, one shows that $\evol$
(or $\Evol$) is smooth.
The following three results are intended to
simplify (or enable) the second step.\\[4mm]
{\bf Theorem D.}
\emph{Let $G$ be a Lie group which is $C^k$-semiregular
for some $k\in \N_0\cup\{\infty\}$.
If $\evol\colon C^k([0,1],\cg)\to G$ is continuous
at~$0$} (\emph{or at any other point}), \emph{then $\evol$
is continuous.}\footnote{Hence also $\Evol\colon
C^k([0,1],\cg)\to C^{k+1}([0,1],G)$
is continuous, by Lemma~\ref{evolEvol}.}\\[4mm]
{\bf Theorem E.}
\emph{Let $G$ be a Lie group which is $C^k$-semiregular for some
$k\in \N_0\cup\{\infty\}$.
If $\evol\colon C^k([0,1],\cg)\to G$ is $C^1$, then $\evol$ is smooth.}\\[4mm]
{\bf Theorem F.}
\emph{Let $G$ be a $C^0$-semiregular
Lie group whose Lie algebra $L(G)$ is integral complete. If $\evol\colon C^0([0,1],\cg)\to G$
is continuous and
there exists a point-separating family $(\alpha_j)_{j\in J}$
of smooth homomorphisms
$\alpha_j\colon G\to H_j$ to
$C^0$-regular Lie groups~$H_j$,
then~$G$ is
$C^0$-regular.}\\[4mm]
For example, the final condition
is satisfied if there exists an injective smooth homomorphism
\[
\alpha\colon G\to H
\]
from $G$ to a $C^0$-regular Lie group~$H$.\\[2.3mm]
Let $m\in \N_0\cup\{\infty\}$.
We say that a subgroup~$H$ of a Lie group~$G$ (modelled
on a locally convex space) is
\emph{$C^m$-initial} in~$G$
if $H$ admits a Lie group structure
with the following properties:
\begin{itemize}
\item[(a)]
The inclusion map $\iota\colon H\to G$
is a smooth homomorphism;
\item[(b)]
$L(\iota)$ is injective;
\item[(c)]
For each map $f\colon X\to H$ from a $C^m$-manifold~$X$
(with or without boundary)
to~$H$, the map $f$ is $C^m$ as a map to~$H$
if and only if~$f$ is $C^m$ as a map to~$G$.
\end{itemize}
For example, an ordinary Lie subgroup
$H\subseteq G$ is $C^m$-initial for each
$m\in \N_0\cup\{\infty\}$,
if $H$ is
modelled on a closed vector subspace
of the modelling space of~$G$.\\[2.3mm]
Also the following tool (announced in slightly weaker
form in~\cite{OWR})
is useful.\\[4mm]
{\bf Theorem G.}
\emph{Let $G$ be a Lie group
and $H\subseteq G$ be a subgroup such that}
\[
H=\{x\in G\colon (\forall j\in J)\;\alpha_j(x)=\beta_j(x)\}
\]
\emph{for families $(\alpha_j)_{j\in J}$
and $(\beta_j)_{j\in J}$ of smooth homomorphisms
$\alpha_j,\beta_j\colon G\to G_j$ to some Lie groups~$G_j$.
Let $k\in \N_0\cup\{\infty\}$.Then the following holds}:
\begin{itemize}
\item[(a)]
\emph{If~$G$ is $C^k$-semiregular and $H$ is $C^m$-initial in~$G$
for some $m\in \N\cup\{\infty\}$
such that $m\leq k+1$, then also~$H$ is $C^k$-semiregular.}
\item[(b)]
\emph{If $G$ is $C^k$-regular and
$H$ is $C^m$-initial in~$G$
for some $m\in \N\cup\{\infty\}$
such that $m\leq k+1$,
then also~$H$ is $C^k$-regular.}\vspace{2mm}
\end{itemize}
Theorem~G (and a variant for local
groups) can be used to compare
regularity of a real analytic (local) Lie group~$G$
and regularity of a complexification~$G_\C$.
Moreover, it naturally applies to
projective limits of Lie groups
and many other situations (see Section~\ref{cxpl}).\\[4mm]
{\bf Known examples of regular Lie groups.}
We recall that every finite-dimensional Lie group
(and every Banach-Lie group) is $C^0$-regular~\cite{SUR}.
Thus $C^\ell(K,H)$ is $C^0$-regular
for each $\ell\in \N_0$, Banach-Lie group~$H$ and compact smooth manifold~$K$.
The group of compactly supported
diffeomorphisms of a finite-dimensional
smooth manifold~$M$ (cf.\ \cite{Mic}) is
$C^0$-regular (as a special case of results
in~\cite{Alx};
cf.\ \cite{OMO} if $M$ is compact).
Many more examples are known.
For instance,
countable direct limits of finite-dimensional
Lie groups are $C^1$-regular~\cite{FUN}.
Criteria for the $C^0$-regularity and $C^1$-regularity
of ascending unions of Banach-Lie groups
can be found in~\cite{Da2}.
In the case of Fr\'{e}chet-Lie groups,
the notion of $C^\infty$-regularity
used here is equivalent to regularity in the
sense of convenient differential calculus,
as discussed in~\cite{KaM}, \cite{KMR}, \cite{MaT}
and applied there to many Lie groups of that
setting.\\[4mm]
{\bf Application: New examples of regular Lie groups.}
The techniques provided in this article
can be used
to establish regularity properties for the following
examples of Lie groups (partly announced in~\cite{OWR}):
\begin{itemize}
\item[(a)]
$C^\ell(K,H)$ is $C^k$-regular,
for all $k,\ell\in \N_0\cup\{\infty\}$,
each compact smooth manifold~$K$ and each
$C^k$-regular Lie group~$H$ (see Section~\ref{gauge}).
\item[(b)]
In~\cite{MEA},
the weak direct product
\[
\bigoplus_{j\in J}G_j:=\{(x_j)_{j\in J}\in\prod_{j\in J}
G_j\colon \mbox{$x_j=e$
for almost all $j\in J$}\}
\]
was made a Lie group modelled on the locally convex direct
sum\linebreak
$\bigoplus_{j\in J}L(G_j)$.
If $J$ is countable,
then the weak direct product is $C^k$-regular if each $G_j$ is $C^k$-regular,
and $k\in \N_0$ is finite (Proposition~\ref{weakreg}).
If $J$ is uncountable and each $G_j$ is $C^k$-regular with $k\in \N_0$,
then $\bigoplus_{j\in J}G_j$ is $C^{k+1}$-regular
(Corollary~\ref{alsodoch}).
\item[(c)]
Using (b) and Theorem~G, we see that
the test function group
$C^\ell_c(M,H)$ (see \cite{QUO}, cf.\ \cite{Alb} and \cite{Mic})
is $C^k$-regular,
for all $\ell\in \N_0\cup\{\infty\}$,
$k\in\N_0$,
each $\sigma$-compact finite-dimensional
smooth manifold~$M$ and each
$C^k$-regular Lie group~$H$
(if $M$ is merely paracompact, then $C^\ell_c(M,H)$ is $C^{k+1}$-regular).
%
%
Likewise, gauge groups $\Gau_c(P)$
are $C^k$-regular for suitable
principal bundles $P\to M$
with structure group~$H$,
and so is the Lie group
$\Aut_c(P)$
of compactly supported symmetries as
constructed in~\cite{Sch}\footnote{See already \cite{Woc}
for the case of compact~$M$.}
(see Section~\ref{gauge}),
which is an extension
\[
\{1\}\to \Gau_c(P)\to \Aut_c(P)\to \Diff_c(M)_P\to\{1\}
\]
for a suitable open subgroup $\Diff_c(M)_P\sub \Diff_c(M)$.
\item[(d)]
For $H$ a regular Lie group,
Neeb and Wagemann~\cite{NaW}
constructed a regular Lie group structure on $C^\infty(\R,H)$.
Let $k\in \N_0\cup\{\infty\}$.
Using a projective limit argument (based on Theorem~G),
in Example~\ref{Wage} we give short proof for the fact that
$C^\infty(\R,H)$ is $C^k$-regular if~$H$ is $C^k$-regular
(as first shown in \cite[Corollary 144]{Alz} as part of a larger theory).
\item[(e)]
If $H$ is a finite-dimensional Lie group with compact Lie algebra
and $\cW$ a suitable set of functions $f\colon \R\to [0,\infty[$
(e.g., $\cW=\{f_a\colon a>0$ with $f_a(t):=e^{-a|t|}$),
then a Lie group $C^\infty_\cW(\R,H)$ can be constructed
whose Lie algebra is the weighted function space
$C^\infty_\cW(\R,L(H))$ (see~\cite{Pie}).\footnote{Note that,
in contrast to
classical constructions
of mapping groups, $C^\infty(\R,L(H))$
and $C^\infty_\cW(\R,L(H))$ contain
unbounded functions.}
Using Theorem~D and applying (d) and Theorem~F
to the inclusion\footnote{Or to the
family $(\alpha_t)_{t\in \R}$ of point evaluations
$\alpha_t\colon C^\infty_\cW(\R,H)\to H$, $\alpha_t(\gamma):=\gamma(t)$.}
\[
\alpha\colon C^\infty_\cW(\R,H)\to C^\infty(\R,H),
\]
one deduces that~$C^\infty_\cW(\R,H)$
is $C^0$-regular
(the author, work in progress).
\item[(f)]
If $M$ is a real analytic manifold
modeled on a metrizable locally convex space,
$M$ is regular as a topological space,
$K\sub M$ a compact non-empty subset
and~$H$ a real Banach-Lie group,
then the group $\Germ(K,H)$ of germs of $H$-valued
real analytic maps around~$K$
can be made a real analytic Lie group
(see \cite{DGS}; the case that~$X$
is a metrizable locally convex space was
already treated in~\cite{GEM}).
Using complexifications (as mentioned after Theorem~G)
and techniques from~\cite{Da2},
one finds that $\Germ(K,H)$
is $C^0$-regular and has a real analytic
evolution map~$\Evol$ (see \cite{DGS};
the case that~$M$ is a Banach space
was already treated in~\cite{Da2}).
In particular, the Lie group
\[
C^\omega(K,H)
\]
of real analytic $H$-valued maps on a
compact real
analytic manifold~$K$ is $C^0$-regular with real analytic
evolution, for
each Banach-Lie group~$H$.
Also, $\Germ([{-n},n],H)$ (with $M:=\R$)
is $C^0$-regular with real analytic evolution
(as another special case).
\item[(g)]
Varying ideas from~\cite{NaW},
the group $C^\omega(\R,H)$
of~$H$-valued real analytic maps
on~$\R$ can be made a Lie group~\cite{DGS},
for each Banach-Lie group~$H$.
Then
\[
C^\omega(\R,H)=\pl\,\Germ([{-n},n],H)\,.
\] 
Combining (f) with
a projective limit argument based on Theorem~G,
it is shown in \cite{DGS} that
$C^\omega(\R,H)$ is $C^0$-regular.
\item[(h)]
The Lie group $\Diff^\omega(M)$
of real analytic diffeomorphisms
of a compact real analytic manifold $M$
is $C^1$-regular
(Corollary~\ref{anaDiffreg}).
Indeed, we have a general result
ensuring $C^1$-regularity
for suitable Lie groups modelled on Silva-spaces
(Proposition~\ref{C1regSilva});
Dahmen and Schmeding~\cite{DaS}
were able to verify the hypotheses
of the proposition.
\end{itemize}
Weakened topologies (like the $L^1$-topology on $C([0,1], \cg)$)
play a role in some of our arguments,
and paved the way for the study of measurable regularity
properties (like $L^1$-regularity based
on $L^1$-curves $\gamma\colon [0,1]\to\cg$)
in subsequent work~\cite{MER}.
\section{Preliminaries and notation}
We write $\N:=\{1,2,\ldots\}$ and $\N_0:=\N\cup\{0\}$.
All locally convex spaces are assumed Hausdorff.
The terms 
``Lie group''
and ``manifold''
refer to Lie groups and manifolds
modelled on locally convex spaces,
as in \cite{RES}, \cite{GaN} and \cite{SUR}.
The notion of $C^k$-map between locally convex
spaces is that of Keller's $C^k_c$-theory
(see \cite{RES}, \cite{Mic}, \cite{Mil}, \cite{SUR}
for expositions in varying generality),
as extended in \cite{GaN} to the
case of maps on subsets of locally convex spaces which have dense interior and are locally convex
(i.e., each point has a relatively open, convex neighbourhood):
\begin{numba}
Let $E$ and $F$ be locally convex spaces,
$r\in \N_0\cup\{\infty\}$
and $U\sub E$ be a subset.
If $U$ is open, then a map $f\colon U \to F$
is called $C^r$ if the iterated directional derivatives
\[
d^if(x,v_1,\ldots, v_i):=(D_{v_i}\cdots D_{v_1}f)(x)
\]
exist for all $i\in \N_0$ with $i\leq r$
and $x\in U$, $v_1\ldots, v_i\in E$,
and $d^if\colon U\times E^i\to F$ is continuous.
If $U\sub E$ is a locally convex subset with dense interior $U^0$,
we say that
$f\colon U\to F$ is $C^r$ if $f|_{U^0}$ is $C^r$
and the maps $d^i(f|_{U^0})$ admit a continuous extension
$d^if\colon U \times E^i\to F$
for each $i$ as before. It is known that $f$ is $C^{r+1}$
if and only if $f$ is $C^1$ and $df$ is $C^r$.
\end{numba}
\begin{numba}
In \cite{GaN},
one finds the concept of a $C^r$-manifold $M$
\emph{with rough boundary}, modelled on a locally convex space~$E$.
In contrast to an ordinary manifold,
the charts $\phi\colon U\to V$ of such~$M$
are $C^r$-diffeomorphisms
from an open subset $U$ of~$M$
to a locally convex subset $V\sub E$
with dense interior.
Manifolds with $C^r$-boundary or
corners are obtained as
special cases.
\end{numba}
We shall also need $C^{r,s}$-maps (as developed in~\cite{Alz}
and \cite{AS}):
\begin{numba}
Let $E_1$, $E_2$ and $F$ be locally convex spaces,
$r,s\in \N_0\cup\{\infty\}$
and $U\sub E_1$, $V\sub E_2$ be subsets.
If $U$ and $V$ are open, then a map $f\colon U\times V\to F$
is called $C^{r,s}$ if the iterated directional derivatives
\[
d^{i,j}f(x,y,v_1,\ldots,v_i,w_1,\ldots,w_j):=
(D_{(v_i,0)}\cdots D_{(v_1,0)}D_{(0,w_j)}\cdots D_{(0,w_1)}f)(x,y)
\]
exist for all $i,j\in \N_0$ with $i\leq r$, $j\leq s$
and $(x,y)\in U\times V$, $v_1\ldots, v_i\in E_1$,
$w_1,\ldots, w_j\in E_2$, and the map
\begin{equation}\label{soonext}
d^{i,j}f\colon U\times V\times E_1^i\times E_2^j\to F
\end{equation}
so obtained is continuous.
If $U$ and $V$ are merely locally convex subsets with dense
interior, we say that $f$ is $C^{r,s}$ if $f|_{U^0\times V^0}$
is $C^{r,s}$ and each of the maps $d^{i,j}(f|_{U^0\times V^0})$
has a continuous extension to a mapping as in (\ref{soonext}).
\end{numba}
\begin{numba}\label{baprels} We recall from \cite{AS}:
\begin{itemize}
\item[(a)]
A mapping to a (finite or infinite)
product of locally convex spaces is $C^{r,s}$
if and only if all of its components are $C^{r,s}$.
\item[(b)]
If $g$ is a $C^{r+s}$-map
and $f$ a $C^{r,s}$-map taking its values in
the domain of~$g$, then also $g\circ f$ is $C^{r,s}$.
\item[(c)]
If $f$ is a $C^{r,s}$-map, $g_1$ a $C^r$-map and $g_2$ a $C^s$-map
such that $g_1\times g_2$ takes its values in the domain of $f$,
then $f\circ(g_1\times g_2)$ is $C^{r,s}$.
\item[(d)]
If $E_1$, $E_2$ and $F$ are locally convex spaces, $U\sub E_1$ a locally convex subset with dense
interior and $f\colon U \times E_2\to F$ a $C^{r,0}$-map
such that $f(x,.)\colon E_2\to F$ is linear for each $x\in U$,
then $f$ is $C^{r,\infty}$.
\item[(e)]
If $E_1$, $E_2$ and $F$ are locally convex spaces, $U\sub E_1$ and $V\sub E_2$ locally
convex subsets with dense interior and $f\colon U\times V\to F$
a $C^r$-map,
then $f$ is $C^{r,0}$ and $C^{0,r}$.
\item[(f)]
If $E_1$, $E_2$ and $F$ are locally convex spaces, $U\sub E_1$ and $V\sub E_2$ locally
convex subsets with dense interior and
$f\colon U\times V\to F$ is $C^{r,s}$,
then $f^\vee(x):=f(x,.)\colon V\to F$ is a $C^s$-map
for each $x\in U$, and the map
$f^\vee\colon U\to C^s(V,F)$ is $C^r$.
\end{itemize}
The following is obvious:
\begin{itemize}
\item[(g)]
If $E_1$, $E_2$ and $F$ are locally convex spaces, $U\sub E_1$ and $V\sub E_2$ locally
convex subsets with dense interior and
$g\colon E_2\supseteq V\to F$ is $C^s$, then $f\colon U\times V\to F$,
$f(x,y):=g(y)$ is $C^{\infty,s}$.
\end{itemize}
\end{numba}
\begin{numba}\label{defcrsman}
If $M_1$, $M_2$ and $N$ are smooth manifolds
modelled on locally convex spaces (possibly with rough boundary)
and $r,s\in \N_0\cup\{\infty\}$,
then a map $f\colon M_1\times M_2\to N$
is called $C^{r,s}$ if it is continuous
and
\begin{equation}\label{ngd1}
\phi\circ f\circ (\phi_1\times \phi_2)^{-1}\:\:\mbox{is $C^{r,s}$}
\end{equation}
for all charts $\phi_1\colon U_1\to V_1$,
$\phi_2\colon U_2\to V_2$
and
$\phi\colon U\to V$ of $M_1$, $M_2$ and $N$, respectively,
such that
\begin{equation}\label{ngd2}
f(U_1\times U_2)\sub U.
\end{equation}
\end{numba}
\begin{numba}\label{nwintr}
In view of \ref{baprels} (b) and (c),
$f$ is $C^{r,s}$ if and only if for all $(p_1,p_2)\in M_1\times M_2$,
there exist $\phi_1$, $\phi_2$ and $\phi$
with (\ref{ngd1}) and (\ref{ngd2})
such that $p_1\in U_1$ and $p_2\in U_2$.
\end{numba}
\begin{numba}\label{preexpo}
\emph{If $f\colon M_1\times M_2\to N$} (\emph{as in} \ref{defcrsman})
\emph{is $C^{r,s}$, then $f^\vee(p_1):=f(p_1,.)\colon M_2\to N$
is $C^s$ for each $p_1\in M_1$.}\\[2.4mm]
To see this, let $p_2\in M_2$ and pick charts $\phi_1$, $\phi_2$
and $\phi$ with $p_1\in U_1$ and $p_2\in U_2$ as in \ref{nwintr}.
Thus $g:=\phi\circ f\circ (\phi_1\times \phi_2)^{-1}\colon V_1\times V_2\to V$
is a $C^{r,s}$-map.
Then $g^\vee(\phi_1(p_1))\colon V_2\to V$ is $C^s$,
by \ref{baprels}\,(f).
As a consequence, also $f^\vee(p_1)|_{U_2}=\phi^{-1}\circ g^\vee(\phi_1(p_1))\circ \phi_2$ is~$C^s$.
Since $p_2$ was arbitrary, it follows that $f^\vee(p_1)$ is~$C^s$.
\end{numba}
\begin{numba}\label{defmagp}
If $G$ is a Lie group
and $r\in \N_0\cup\{\infty\}$, then also $C^r([0,1],G)$
is a Lie group, as is well known.
If $\phi\colon U\to V\sub E$ is a chart for $G$
such that $e\in U$ and $\phi(e)=0$,
then $C^r([0,1],U)$ is open in $C^r([0,1],G)$, the set
$C^r([0,1],V)$ is open in $C^r([0,1],E)$,
and the map $\phi_* \colon C^r([0,1],U)\to C^r([0,1],V)$, $\gamma\mto \phi\circ\gamma$ is a $C^\infty$-diffeomorphism
(see, e.g.,  \cite{GaN}; cf.\ \cite{QUO}).
\end{numba}
\begin{numba}\label{defmastar}
Note that $C^r_*([0,1],E):=
\{\gamma\in C^r_*([0,1],E)\colon \gamma(0)=0\}$ is a closed
vector subspace of $C^r_*([0,1],E)$ in the preceding situation.
Since $\phi_*$ takes
$C^r([0,1],U)\cap C^r_*([0,1],G)$ onto $C^r([0,1],V)\cap C^r_*([0,1],E)$,
we see that\linebreak
$C^r_*([0,1],G)$ is a Lie subgroup
of $C^r([0,1],G)$ modelled
on the closed vector subspace
$C^r_*([0,1],E)$.
This implies that a map to $C^r_*([0,1],G)$ 
is smooth if and only if
it is smooth as a map to $C^r([0,1],G)$
(cf.\ \cite[Lemma~10.1]{BGN}).
\end{numba}
\begin{numba}\label{mapsugp}
If $H\sub G$ is a Lie subgroup
modelled on a closed vector subspace,
then so is $C^r([0,1],H)\sub C^r([0,1],G)$
(entailing that a map to $C^r([0,1],H)$ is smooth if and only if it is smooth as a map
to $C^r([0,1],G)$).\\[2.4mm]
In fact,
the chart $\phi$ in \ref{defmagp}
can be chosen such that $\phi(H\cap U)=F\cap V$
for some closed vector subspace $F\sub E$.
Since $\phi_*$ takes
$C^r([0,1],U)\cap C^r([0,1],H)$ onto $C^r([0,1],V)\cap C^r([0,1],F)$,
we see that $C^r([0,1],H)$ is a Lie subgroup
modelled
on the closed vector subspace
$C^r_*([0,1],F)$.
\end{numba}
\begin{numba}\label{embeddings}
Let $G$ be a Lie group and $\mu\colon G\times G\to G$
be the group multiplication.
Then $TG$ is a Lie group with respect to the multiplication
\[
T\mu\colon T(G\times G)\to TG,
\]
identifying the left hand side with
$TG\times TG$. The zero section
$\theta\colon G\to TG$, $g\mto 0 \in T_g(G)$ is a smooth group
homomorphism, and actually an isomorphism onto a Lie subgroup
modelled on a closed vector subspace.
Hence also
$\theta_*\colon C^r([0,1],G)\to C^r([0,1],TG)$, $\gamma\mto \theta\circ \gamma$
is an isomorphism onto such a Lie subgroup, by~\ref{mapsugp}.
Likewise, the inclusion map $\alpha \colon L(G)\to TG$
is an isomorphism onto a Lie subgroup modelled on a closed vector subspace,
and hence so is
$\alpha_*\colon C^r([0,1],L(G))\to C^r([0,1],TG)$,
$\gamma\mto \alpha\circ \gamma$.
This will be useful later.
\end{numba}
\begin{numba}
If $M$ is a manifold, $E$ a locally convex space
and $f\colon M\to E$
a $C^1$-map, we identify $TE$ with $E\times E$ as usual
and let $df\colon TM\to E$ be the second component
of $Tf\colon TM\to E\times E$. Thus $Tf(v)=(f(p),df(v))$ if $v\in T_pM$.
\end{numba}
%
%
%
%
%
%
%
%
%
%
%
%
%
%
%
\begin{numba}
Recall that the left logarithmic derivative
$\delta^\ell(\eta)\colon [0,1]\to \cg$
of a $C^1$-curve $\eta\colon [0,1]\to G$
in a Lie group~$G$ is defined via
\[
\delta^\ell(\eta)(t):=\eta(t)^{-1}.\eta'(t)
\quad\mbox{for $\,t\in [0,1]$.}
\]
Thus $\eta=\Evol(\gamma)$ if and only if
$\delta^\ell(\eta)=\gamma$ and $\eta(0)=e$.
\end{numba}
\begin{numba}
If $\eta\in C^1([0,1],G)$
with $\delta^\ell(\eta)=\gamma\in C^k([0,1],G)$,
then $\eta\in C^{k+1}([0,1],G)$
by a simple induction on~$k$,
using that $\eta'=\eta.\gamma=T\mu\circ (\theta\circ \eta,\alpha\circ
\gamma)$ (with $\mu,\theta,\alpha$
as in \ref{embeddings})
is $C^k$
if~$\eta$ is~$C^k$.
\end{numba}
The following three facts are well known and easy to check.
\begin{numba}\label{simplsimp}
If $G$ is a Lie group
and $\gamma\in C^0([0,1],\cg)$
has an evolution $\Evol(\gamma)\in C^1_*([0,1],G)$, then
also $\gamma_s\colon [0,1]\to \cg$, $\gamma_s(t):=\gamma(st)$
has an evolution for
for each $s\in [0,1]$, given by
\[
\Evol(\gamma_s)(t)=\Evol(\gamma)(st)\quad\mbox{for all $t\in [0,1]$.}
\]
In particular,
\[
\Evol(\gamma)(s)=\evol(s\gamma_s).
\]
This follows from $\frac{d}{dt}\Evol(\gamma)(st)
=s\Evol(\gamma)(st).\gamma(st)$.
\end{numba}
\begin{numba}\label{evolandhoms}
If $\alpha\colon G\to H$ is a smooth homomorphism between
Lie groups and $\eta\in C^1([0,1],G)$, then
\[
\delta^\ell(\alpha\circ\eta)=L(\alpha)\circ \delta^\ell(\eta).
\]
Hence, if $\gamma\in C^0([0,1],L(G))$ such that
$\Evol(\gamma)\in C^1_*([0,1],G)$ exists,
then the evolution of $L(\alpha)\circ\gamma\in C^0([0,1],L(H))$
exists and is given by
\[
\Evol(L(\alpha)\circ \gamma)=\alpha\circ \Evol(\gamma)\in C^1([0,1],H).
\]
\end{numba}
\begin{numba}\label{Adalpha}
If $\alpha\colon G\to H$ is a smooth
homomorphism between Lie groups,
then
\[
\alpha(\Ad_G(g).x)=\Ad_H(\alpha(g)).L(\alpha)(x)
\]
for all $g\in G$ and $x\in L(G)$.
\end{numba}
\begin{numba}
Given a Lie group $G$ and $g\in G$,
we consider the right translation $\rho_g\colon G\to G$, $x\mto xg$
and define
\[
v.g:=T\rho_g(v)\quad \mbox{for $v\in TG$.}
\]
The right logarithmic derivative
of a $C^1$-curve $\gamma\colon [0,1]\to G$
is the continuous curve $\delta^r(\gamma)\colon [0,1]\to\cg$,
$\delta^r(\gamma)(t):=\gamma'(t).\gamma(t)^{-1}$.
It determined $\gamma$ up to multiplication
with a group element on the right.
\end{numba}
We shall use the same notation for logarithmic derivatives and
evolution maps for curves defined on non-degenerate intervals
$J\sub \R$ (in place of $[0,1]$) such that $0\in J$.
\section{Auxiliary results}
We first establish differentiability properties
for some relevant maps.
\begin{la}\label{aresmooth}
Let $G$ be a Lie group and
$k\in \N_0\cup\{\infty\}$.
Then the map
\[
D\colon C^{k+1}([0,1],G)\to C^k([0,1], TG),\quad
\gamma\mto \gamma'
\]
is a smooth group homomorphism,
and also the following map is smooth:
\[
\delta^\ell \colon C^{k+1}([0,1],G)\to C^k([0,1], L(G)),\quad
\gamma\mto \delta^\ell(\gamma).
\]
\end{la}
\begin{proof}
Let $E$ be the modelling space for $G$.
The map
\[
D_E\colon C^{k+1}([0,1],E)\to C^k([0,1],E),\quad \gamma\mto \gamma'
\]
is continuous linear (see \cite{AS};
cf.\ also \cite[Lemma A.1\,(d)]{ZOO}).
If $\mu\colon G\times G\to G$ is the group multiplication in $G$,
then the map $(\gamma,\eta)\mto
\gamma\eta=\mu\circ (\gamma,\eta)$ is the multiplication
in $C^{k+1}([0,1],G)$.
Now $(\gamma\eta)'(t)=T\mu(\gamma'(t),\eta'(t))$
for $\gamma,\eta\in
C^{k+1}([0,1],G)$ and $t\in [0,1]$.
Hence $(\gamma\eta)'=T\mu\circ (\gamma',\eta')$,
which is the product of $\gamma'$ and $\eta'$ in
$C^k([0,1],TG)$.
Thus $D$ is a homomorphism of groups,
and hence it will be smooth if we can show its smoothness
on the open identity neighbourhood
$\Omega:=C^{k+1}([0,1],U)$
for some chart $\phi\colon U\to V\sub E$
around~$e$, with $\phi(e)=0$.
Then
$\phi_* \colon C^{k+1}([0,1],U)\to C^{k+1}([0,1],V)$
is a chart for $C^{k+1}([0,1],G)$ and
$(T\phi)_* \colon C^k([0,1],TU)\to C^k([0,1],V\times E)$,
$\gamma\mto (T\phi)\circ\gamma$ a chart for
$C^k([0,1], TG)$.
Since
\[
(T\phi)_*\circ D|_\Omega = D_E\circ \phi_*
\]
(as $T\phi(\gamma'(t))=(\phi\circ\gamma)'(t)$),
we see that $D|_\Omega=
((T\phi)_*)^{-1} \circ D_E\circ \phi_*$ is smooth. Hence $D$ is a smooth homomorphism.\\[2.4mm]
Let $\Lambda\colon C^{k+1}([0,1],G)\to C^k([0,1],G)$
be the inclusion map (which is a smooth homomorphism),
$\sigma\colon G\to G$, $g\mto g^{-1}$
be the inversion map
and $\sigma_*\colon
C^k([0,1],G)\to C^k([0,1],G)$, $\gamma\mto \sigma\circ \gamma=\gamma^{-1}$
be the smooth inversion map in $C^k([0,1],G)$.
Let $\theta\colon G\to TG$,
$\alpha \colon L(G)\to TG$
as well as $\theta_*\colon C^k([0,1],G)\to C^k([0,1],TG)$
and $\alpha_*\colon C^k([0,1],L(G))\to C^k([0,1],TG)$
be as in~\ref{embeddings}.
Finally, let
\[
(T\mu)_*\colon
C^k([0,1],TG)\times C^k([0,1],TG)\to C^k([0,1],TG),\quad
(\gamma,\eta)\mto T\mu\circ (\gamma,\eta)
\]
be the smooth group multiplication of $C^k([0,1],TG)$.
Then
\[
\delta^\ell(\gamma)=(T\mu)_*( \theta_*(\sigma_*(\gamma)),D(\gamma))
\]
for $\gamma\in C^{k+1}([0,1],G)$ and thus
\[
\alpha_*\circ \delta^\ell=(T\mu)_*\circ (\theta_*\circ \sigma_*\circ \Lambda, D),
\]
with $D$ as in Lemma~\ref{ini1}. Thus
$\alpha_*\circ \delta^\ell$ is smooth, i.e., $\delta^\ell$ is smooth as a map to
$C^k([0,1],TG)$.
Then $\delta^\ell$ is also smooth
as a map to $C^k([0,1],L(G))$ (see \ref{embeddings}),
which completes the proof.
\end{proof}
It will be useful later that the $C^r$-manifold structure
on $C^{k+1}([0,1],G)$ is initial
with respect to suitable maps.
\begin{la}\label{ini1}
Let $G$ be a Lie group,
$k\in \N_0$,
$r\in \N_0\cup\{\infty\}$,
\[
D\colon C^{k+1}([0,1],G)\to C^k([0,1], TG),\quad
\gamma\mto \gamma'
\]
and
$\iota\colon C^{k+1}([0,1],G)\to C([0,1],G)$, $\gamma\mto\gamma$
be the inclusion map.
If
$M$ is a $C^r$-manifold and
$f\colon M\to C^{k+1}([0,1],G)$
a map such that both $\iota\circ f$ and $D\circ f$ are $C^r$,
then $f$ is $C^r$.
\end{la}
\begin{proof}
Step 1. Let $E$ be the modelling space of $G$.
Then the map
\[
(\iota_E, D_E)\colon C^{k+1}([0,1],E)\to C([0,1],E)\times C^k([0,1],E),\quad \gamma\mto(\gamma,\gamma')
\]
is linear and a topological embedding with closed
image (see \cite{AS}; cf.\ also
\cite[Lemma~A.1\,(d)]{ZOO}).
Therefore, a map $h\colon M\to C^{k+1}([0,1],E)$
is $C^r$ if and only if both $\iota_E\circ h$ and $D_E\circ h$
are $C^r$ (cf.\ \cite[Lemmas 10.1 and 10.2]{BGN}).\\[2.4mm]
Step 2.
Let a chart $\phi\colon U\to V\sub E$ for $G$
and the pushforward\linebreak
$\phi_* \colon C^{k+1}([0,1],U)\to C^{k+1}([0,1],V)$
be as in \ref{defmagp}.\\[2mm]
Each $p\in M$ has an open neighbourhood $P\sub M$ such that $f(q)f(p)^{-1}\in U$
for all $q\in P$. As it suffices to show that $f|_P$ is $C^r$, we may assume $P=M$.\linebreak
Consider $\eta:=f(p)^{-1}\in C^{k+1}([0,1],G)$
and
the right translation mappings $\rho_\eta \colon C^{k+1}([0,1],G)\to C^{k+1}([0,1],G)$
and
$R_\eta\colon C([0,1],G)\to C([0,1],G)$
given by  $\gamma\mto \gamma \eta$ (which are smooth).
Then
\[
g:=\rho_\eta \circ f \colon M\to C^{k+1}([0,1],G)
\]
is $C^r$ as a function to $C([0,1],G)$
(as $\iota\circ g=R_\eta \circ \iota\circ f$).
Also, $D\circ g=\rho_{\eta'}\circ D \circ f$
is $C^r$,
where we use the smooth right translation map
\[
\rho_{\eta'}\colon
C^k([0,1],TG)\to C^k([0,1],TG),\quad \gamma\mto \gamma \eta'\, .
\]
Now $h:= \phi_*\circ g\colon M\to
C^{k+1}([0,1],E)$ is
$C^r$ as a map to $C(M,E)$
(as $\iota_E\circ h=\wt{\phi}_*\circ \iota\circ g$
with the $C^\infty$-diffeomorphism
$\wt{\phi}_*\colon
C([0,1],U)\to C([0,1],V)$, $\gamma\mto \phi\circ\gamma$).
Furthermore, $D_E\circ h= (d\phi)_*\circ D\circ g$
is $C^r$,
using that
$(d\phi)_*\colon
C^k([0,1],TU)\to C^k([0,1],E)$ is smooth
as it is the second component of the chart
\[
(T\phi)_*\colon
C^k([0,1],TU)\to C^k([0,1],TV)\isom
C^k([0,1],V)\times C^k([0,1],E).
\]
By Step~1, $h$ is $C^r$.
As a consequence, $g$ is~$C^r$ and
hence also~$f$.
\end{proof}
\begin{la}\label{ini2}
Let $G$ be a Lie group,
$k\in \N_0$, $r\in \N_0\cup\{\infty\}$,
\[
\delta^\ell \colon C^{k+1}([0,1],G)\to C^k([0,1], L(G)),\quad
\gamma\mto \delta^\ell(\gamma)
\]
and $\iota\colon C^{k+1}([0,1],G)\to C([0,1],G)$
be the inclusion map.
If $M$ is a $C^r$-manifold and
$f\colon M\to C^{k+1}([0,1],G)$
a map such that both $\iota\circ f$ and $\delta^\ell \circ f$ are $C^r$,
then $f$ is $C^r$.
\end{la}
\begin{proof}
We show by induction on $j=0,\ldots, k$ that $f$ is $C^r$ as a map to
$C^{j+1}([0,1],G)$.
The same notation, $\delta^\ell$,
will be used for the corresponding map
$C^{j+1}([0,1],G)\to C^j([0,1], L(G))$.
Let $\mu$, $\theta$, $\alpha$ and the pushforwards
$\theta_*\colon C^j([0,1],G)\to C^j([0,1],TG)$
and $\alpha_*\colon C^j([0,1],L(G))\to C^j([0,1],TG)$
be as in~\ref{embeddings}.
If $x\in M$, we have
$D(f(x))(t)=f(x)(t).\delta^\ell(f(x))(t)$
and hence $D(f(x))
=T\mu\circ (\theta\circ (f(x)),\alpha\circ (\delta^\ell(f(x))))$,
i.e.,
\begin{equation}\label{twoti}
D\circ f=(T\mu)_*\circ (\theta_*\circ f,\alpha_*\circ \delta^\ell\circ f)
\end{equation}
with the smooth group multiplication
$(T\mu)_*\colon C^j([0,1],TG)\times C^j([0,1],TG)\to C^j([0,1],TG)$.
If $j=0$, then $\theta_*\circ f=\theta_*\circ \iota\circ f$ is $C^r$,
whence also $D\circ f$ is $C^r$ (by (\ref{twoti}))
and hence also $f$ (by Lemma~\ref{ini1}).\\[2.4mm]
If $j>0$, let us make the inductive hypothesis that $f$ is $C^r$ as a map to $C^j([0,1],G)$.
Then (\ref{twoti}) shows that $D\circ f$ is $C^r$ as a map to $C^j([0,1],TG)$.
Hence $f$ is $C^r$ as a map to $C^{j+1}([0,1],G)$,
by Lemma~\ref{ini1}.
\end{proof}
\begin{la}\label{pivot}
Let $G$ be a Lie group,
$k \in \N_0\cup\{\infty\}$
and
\begin{equation}\label{thisvers}
\delta^\ell\colon C^{k+1}_*([0,1],G)\to
C^k_*([0,1],L(G))
\end{equation}
be left logarithmic differentiation.
Let ${\bf 1}\in C^{k+1}_*([0,1],G)=:P$
be the constant function $t\mto e$.
Then, after identifying $T_{\bf 1}P$ with
$C^{k+1}_*([0,1],L(G))$ in a suitable way,
the tangent map of $\delta^\ell$ at ${\bf 1}$
becomes the map
\begin{equation}\label{roughly}
C^{k+1}_*([0,1],L(G))\to C^k([0,1],L(G)),\quad \gamma\mto \gamma'\,.
\end{equation}
\end{la}
\begin{proof}
Let $U\sub G$ be an open symmetric identity neighbourhood
on which a chart $\phi\colon U\to V\sub L(G)$ is defined, with $\phi(e)=0$
and $d\phi|_{L(G)}=\id_{L(G)}$.
Let $\mu\colon G\times G\to G$ be the multiplication of $G$
and $\sigma\colon G\to G$ be inversion.
Then $D_U:=\{(g,h)\in U\times U\colon gh\in U\}$
is an open subset of $U\times U$,
and hence $D_V:=(\phi\times \phi)(D_U)$
is an open subset of $V\times V$. The maps
\[
\tau:=\phi\circ \sigma\circ \phi^{-1}\colon V\to V\quad\mbox{and}
\]
\[
\nu:=\phi\circ \mu\circ (\phi^{-1}\times \phi^{-1})\colon D_V\to V
\]
are smooth and $(V,D_V,\nu,0)$ 
is a local group as in \cite[Definition~II.1.10]{SUR},
with inverses given by $x^{-1}:=\tau(x)$.
Consider the restriction
\[
\delta^\ell_U\colon
C^{k+1}_*([0,1],U)\to C^k([0,1],L(G))
\]
of $\delta^\ell$ from (\ref{thisvers})
and the map
$\delta^\ell_V\colon
C^{k+1}_*([0,1],V)\to C^k([0,1],L(G))$
defined via
\[
\delta^\ell_V(\gamma)(s):=
d\nu(\gamma(s)^{-1},\gamma(s);0,\gamma'(s))
\]
for $\gamma\in C^{k+1}_*([0,1],V)$ and $s\in [0,1]$.
Let $\phi_*\colon C^{k+1}_*([0,1],U)\to C^{k+1}([0,1],V)$,
$\gamma\mto\phi\circ\gamma$
be the pushforward.
An elementary calculation shows that
\begin{equation}\label{commutat}
\delta^\ell_V\circ \phi_*=\delta^\ell_U.
\end{equation}
Since $\phi_*$ is a $C^\infty$-diffeomorphism between open subsets of
$C^{k+1}_*([0,1],G)$ and $C^{k+1}([0,1],L(G))$, respectively,
we deduce from the smoothness of $\delta^\ell_U$
that also $\delta^\ell_V$ is smooth. Applying the Chain Rule to (\ref{commutat})
now gives
\begin{equation}\label{commu2}
d(\delta^\ell_V)\circ T_{{\bf 1}}(\phi_*)=d(\delta^\ell_U)|_{T_{{\bf 1}}P}.
\end{equation}
We claim that
\begin{equation}\label{almost}
\frac{d}{dt}\Big|_{t=0}\delta^\ell_V(t\gamma)=\gamma'
\end{equation}
for each $\gamma\in C^{k+1}_*([0,1],L(G))$,
i.e., $d\delta^\ell_V(0,\gamma)=\gamma'$.
Since also the map\linebreak
$T_{{\bf 1}}(\phi_*)
\colon T_{\bf 1}P\to \{0\}\times  C^{k+1}_*([0,1],L(G))\cong C^{k+1}_*([0,1],L(G))$
is an isomorphism, we deduce from (\ref{commu2})
and (\ref{almost})
that $d(\delta^\ell)|_{T_{\bf 1}P}$ is the map
in (\ref{roughly}),
up to composition with the isomorphism
$T_{{\bf 1}}(\phi_*)$.\\[2.4mm]
To verify (\ref{almost}), we only need to show that
\[
\left(\frac{d}{dt}\Big|_{t=0}(\delta^\ell_U(t\gamma)\right)(s)=\gamma'(s)
\]
for each $s\in [0,1]$. The point evaluation at $s$ being
continuous linear, the left hand side coincides
with
\begin{eqnarray*}
\frac{d}{dt}\Big|_{t=0}(\delta^\ell_U(t\gamma)(s)) &=&
\frac{d}{dt}\Big|_{t=0}   d\nu((t\gamma(s))^{-1},t\gamma(s); 0, t\gamma'(s))\\
&=&\underbrace{d^{(2)}\nu(0,0;0,0;-\gamma(s),\gamma(s))}_{=0}+
d\nu(0,0;0,\gamma'(s))\\
&=&\gamma'(s),
\end{eqnarray*}
as required.
\end{proof}
\begin{la}\label{explawgp}
Let $G$ be a Lie group, $M$ be a manifold
and $f\colon M\times [0,1]\to G$ be a $C^{r,s}$-map.
Then $f^\vee(p):=f(p,.)\colon [0,1]\to G$ is $C^s$ for each
$p\in M$, and the map $f^\vee\colon M\to C^s([0,1],G)$
is $C^r$.
\end{la}
\begin{proof}
The first assertion is a special case of \ref{preexpo}.
To establish the second, choose a chart $\phi\colon U\to V$ of $G$ around~$e$.
For $p\in M$, consider the auxiliary function
\[
g\colon M\times [0,1]\to G,\quad h(q,t):=f(q,t)f(p,t)^{-1}.
\]
Since $g$ is continuous and $g(\{p\}\times [0,1])=\{e\}$,
$g^{-1}(U)$ is an open superset of the compact set
$\{p\}\times [0,1]$. By the Wallace Lemma \cite[5.12]{Kel},
there is an open neighbourhood $P\sub M$
such that $P\times [0,1]\sub g^{-1}(U)$ and thus
$g(P\times [0,1])\sub U$. After shrinking~$P$, we may assume that there exists
a chart $\psi\colon P\to Q$ for~$M$.
Without loss of generality $M=P$.
Let $\mu\colon G\times G\to G$ be the group multiplication and
$\sigma\colon G\to G$ be inversion.
Since $f(p,.)\colon [0,1]\to G$ is $C^s$ by \ref{preexpo},
using \ref{baprels} (g) we deduce that the map
\[
\zeta\colon M\times [0,1]\to G,\quad (q,t)\mto f(p,t)^{-1}=\sigma(f(p,t))
\]
is $C^{r,s}$. Hence $g=\mu\circ (f,\zeta)$
is $C^{r,s}$, as a consequence of \ref{baprels}\,(c).
Thus $h:=\phi\circ g\circ (\psi^{-1}\times \id_{[0,1]})$
is $C^{r,s}$. Now \ref{baprels}\,(f) shows that
$h^\vee\colon Q\to C^s([0,1],V)$\linebreak
is $C^r$.
Hence also $g^\vee=(\phi_*)^{-1}\circ h^\vee\circ \psi$
is $C^r$, using that the map\linebreak
$\phi_*\colon C^s([0,1],U)\to C^s([0,1],V)$, $\gamma\mto\phi\circ\gamma$
is a chart for $C^s([0,1],G)$ and hence a $C^\infty$-diffeomorphism. 
\end{proof}
\begin{la}\label{themapS}
Let $E$ be a locally convex space and $k\in \N_0\cup\{\infty\}$.
Then
$S\colon C^k([0,1],E)\times [0,1]\to C^k([0,1],E)$,
\[
S(\gamma,s)(t):=s\gamma(st)\quad
\mbox{for $\,\gamma\in C^k([0,1],E)$, $\,s,t\in [0,1]$}
\]
is a $C^{\infty,0}$-map.
\end{la}
\begin{proof}
Consider the map
\[
h\colon C^k([0,1],E)\times [0,1]\to C^k([0,1],E),\quad
(\gamma,s)\mto \gamma_s
\]
(with $\gamma_s(t):=\gamma(st)$).
Then
\begin{equation}\label{conseq}
S(\gamma,s)=\beta(s,h(\gamma,s)),
\end{equation}
using the scalar multiplication
$\beta\colon \R\times C^k([0,1],E)\to C^k([0,1],E)$, $\beta(r,\gamma):=r\gamma$.
Now
$\beta$ is continuous bilinear and hence smooth.
Moreover, the map $\pi\colon C^k([0,1],E)\times [0,1]\to \R$,
$(\gamma,s)\mto s$ is $C^\infty$
and hence $C^{\infty,0}$.
We claim that also $h$ is $C^{\infty,0}$.
If this is true, then $(\pi,h)$ is $C^{\infty,0}$
by \ref{baprels}\,(a).
Since $S=\beta\circ (\pi,h)$ by (\ref{conseq}),
\ref{baprels}\,(b) shows that $S$ is $C^{\infty,0}$.\\[2.5mm]
Because $h$ is linear in its first argument,
\ref{baprels}\,(c) applies.
Thus, the claim will hold if we can show that $h$ is $C^{0,0}$
(i.e., $C^0$).
To this end, consider the map
\[
h^\wedge \colon (C^k([0,1],E)\times [0,1])\times [0,1]\to E,\quad
h^\wedge(\gamma,s,t):=\gamma(st).
\]
Thus $h^\wedge(\gamma,s,t)=\ve(\gamma,st)$,
where $\ve\colon C^k([0,1],E)\times [0,1]\to E$, $\ve(\gamma,s):=\gamma(s)$
is $C^k$ (see \cite{GaN}, cf.\ \cite[Proposition~11.1]{ZOO})
and hence $C^{0,k}$.
Therefore $h=(h^\wedge)^\vee$ is continuous
(by \ref{baprels}\,(e)), as required.
\end{proof}
\section{Proof of Theorem A}
\begin{la}\label{evolEvol}
Let $k,m\in \N_0\cup\{\infty\}$
and $G$ be a $C^m$-semiregular Lie group,
with Lie algebra $\cg:=L(G)$.
Then $\evol \colon C^k([0,1],\cg)\to G$
is~$C^m$ if and only if $\Evol\colon C^k([0,1],\cg)\to
C^{k+1}([0,1],G)$ is~$C^m$.
\end{la}
\begin{proof}
If $\Evol$ is $C^m$, then also
$\evol=\ev_1\circ \Evol$ is $C^m$,
using that the evaluation map
\[
\ev_1\colon C^{k+1}([0,1],G)\to G, \quad \gamma\mapsto \gamma(1)
\]
is a smooth homomorphism.\\[2.3mm]
Conversely, assume that $\evol$ is $C^m$.

(i) By Lemma~\ref{explawgp}, $\Evol\colon C^k([0,1],\cg)\to C^0([0,1],G)$
will be~$C^m$ if we can show that the map
\[
\Evol^\wedge\colon C^k([0,1],\cg)\times [0,1]\to G,\quad
\Evol^\wedge(\gamma,s):=\Evol(\gamma)(s)
\]
is $C^{m,0}$ (using that $\Evol=(\Evol^\wedge)^\vee$).
However, using the $C^{\infty,0}$-map
\[
S\colon C^k([0,1],L(G))\times [0,1]\to C^k([0,1],L(G)),\quad
(\gamma,s)\mto s\gamma_s
\]
from Lemma~\ref{themapS},
\ref{simplsimp} enables us to write
\[
\Evol^\wedge(\gamma,s)=\Evol(\gamma)(s)=\evol(S(\gamma,s)).
\]
Thus $\Evol^\wedge=\evol\circ S$,
which is $C^{m,0}$ by \ref{baprels}\,(b)
(being a composition of a $C^\infty$-map and
a $C^{m,0}$-map).

(ii) Now consider the inclusion map
$\iota\colon C^{k+1}([0,1],G)\to C^0([0,1],G)$
and the map $\delta^\ell\colon
C^{k+1}([0,1],G)\to C^k([0,1],\cg)$.
Then $\iota\circ \Evol$ is $C^m$ (by (i))
and $\delta^\ell\circ \Evol$ is the identity map
$C^k([0,1],\cg)\to C^k([0,1],\cg)$
and hence~$C^m$ as well.
Therefore $\Evol\colon C^k([0,1],\cg)\to C^{k+1}([0,1],G)$
is~$C^m$, by Lemma~\ref{ini2}.
\end{proof}
{\bf Proof of Theorem A.}
By \ref{defmastar}, $\Evol$ is also smooth when regarded as a map
\begin{equation}\label{onedir}
\Evol \colon C^k([0,1],\cg)\to C^{k+1}_*([0,1],G).
\end{equation}
Also the map
$\delta^\ell \colon C^{k+1}([0,1],G)\to C^k(0,1],\cg)$
is smooth (by Lemma~\ref{aresmooth}),
and so is its restriction
\begin{equation}\label{secdir}
\delta^\ell \colon C^{k+1}_*([0,1],G)\to C^k([0,1],\cg).
\end{equation}
Because the mappings in (\ref{onedir})
and (\ref{secdir}) are mutually inverse,
we conclude that they are $C^\infty$-diffeomorphisms.
The proof is complete.\Punkt
\section{Proof of Theorem B}
\noindent
Consider the algebra $P = \R[X]$ of
polynomial functions on $[0, 1]$
and its multiplicative subset
$S := \{s\in P\colon s([0, 1]) \sub \R^\times\}$.
Then
\[
A := PS^{-1}= \{p/s\colon p \in P, s \in S\}
\]
is a unital subalgebra of the Banach algebra
$C[0, 1]$ of continuous real-valued
functions on $[0,1]$ (and we endow $A$ with the induced topology).
As shown in \cite[\S6]{ALG},
$A$ has an open group~$A^\times$
of invertible elements
and~$A^\times$ is a smooth (and even analytic)
Lie group.
Note that $H:=\{a\in A\colon a(0)=0\}$
is a closed non-unital subalgebra of~$A$.
Now consider the subgroup
\[
K:=\{a\in A^\times\colon a(0)=1\}
\]
of~$A^\times$.
The map $\phi\colon A^\times \to A^\times -\{1\}$,
$a\mto a-1$ is a global chart for the manifold~$A^\times$,
which takes~$K$ onto $(A^\times -1)\cap H$.
Hence~$K$ is a submanifold of~$A^\times$ modelled on~$H$
and hence a Lie (sub)group.
We let~$G$ be the $1$-connected
covering Lie group of the connected component~$K_1$
of~$1$ in~$K$,
and $q\colon G\to K_1$ be the universal covering
homomorphism. We identify $L(G)$ with $L(K_1)$ by means of $L(q)$,
and $L(K_1)$ with~$H$ by means of the isomorphism $d\phi|_{L(K_1)}$.
We now verify that $G$ and $H$ provide counterexamples
for all of the statements (a)--(d) discussed in Theorem~B.\\[2.4mm]
First, we show that $K_1$ has the property described at the end of Theorem~B.
To this end, let $\gamma\colon (\R,+)\to K_1$
be a homomorphism of groups.
Since~$K_1$ is a subgroup of $A^\times$,
\cite[Proposition 6.1]{ALG} shows that $\im(\gamma)\sub \R^\times 1$.
Then $\im(\gamma)\sub( \R^\times 1)\cap K=\{1\}$,
whence $\gamma=1$ indeed.\\[2.4mm]
(d) Let $\gamma\colon (\R,+)\to G$
be a homomorphism of groups.
Then $q\circ \gamma=1$, as just observed.
Hence $\im(\gamma)$ is a subgroup of
$\ker(q)$ and hence discrete.
If we assume that $\gamma$ is continuous, then
$\im(\gamma)$ is connected.
Being also discrete, it must be $\{e\}$.
Hence $\gamma=e$.\\[2.4mm]
(c) If $G$ was isomorphic to $E/\Gamma$ as a Lie group
with a locally convex space~$E$ and discrete subgroup $\Gamma\sub E$,
then $E$ (like $G$) would be infinite-dimensional.
Hence a non-zero vector $v\in E$ exists, and provides
a non-constant smooth homomorphism $\R\to E$, $t\mto tv$,
corrresponding to a non-constant smooth homomorphism
$\R\to G$. But we have just seen that $G$ does not admit non-trivial continuous
one-parameter groups, contradiction.\\[2.4mm]
(b) The Lie algebras of both $(H,+)$ and $G$ can be identified with the locally convex space
$H$ (with the zero Lie bracket). Both $G$ and $H$ are $1$-connected,
but they cannot be isomorphic because $H$ has many non-constant
smooth one-parameter groups (of the form $t\mto tv$ with non-zero $v\in H$)
while $H$ has none.\\[2.4mm]
(a) The Lie algebra homomorphism $\id_H\colon H\to H$
cannot integrate to a smooth homomorphism $\Phi\colon (H,+)\to G$,
because for non-zero $v\in H$ the map $\gamma_v\colon \R\to G$, $t\mto \Phi(tv)$ then would be a smooth
homomorphism with $\gamma_v'(0)=L(\Phi)(v)=v\not=0$
and thus $\gamma_v\not=e$, contradicting the above.
This completes the proof of Theorem~B.\,\Punkt
\begin{rem}
The author does not know whether~$K_1$ in Theorem~B
can be chosen connected and simply connected.
In particular, he was not able to show that~$K_1$ can be replaced
by the Lie group $G$ just discussed.
\end{rem}
\section{Proof of Theorem C}\vspace{-1.2mm}
\hspace*{4mm}(a) Let $G$ be a $C^k$-regular Lie group with $k=0$
and $P:=C^{k+1}_*([0,1],G)$.
Since $\delta^\ell\colon C^{k+1}_*([0,1],G)
\to C^k([0,1],L(G))$
is a diffeomorphism
with inverse $\Evol\colon C^k([0,1],L(G))\to C^{k+1}_*([0,1],G)$,
the differential
\[
d(\delta^\ell)|_{T_{\bf 1}P}
\colon T_{\bf 1}P\to  C^k([0,1],L(G))
\]
(which coincides with
$d(\delta^\ell_U)|_{T_{\bf 1}P}$)
is an isomorphism.
Since also $T_{\bf 1}(\phi_*)$
is an isomorphism, we deduce from (\ref{commu2})
and (\ref{almost}) (applied with $k=0$)
that the map
\[
D\colon C^{k+1}_*([0,1],L(G))\to C^k([0,1],L(G)),\quad \gamma\mto\gamma'
\]
is an isomorphism and hence surjective. Recalling that $k=0$, we see that
$L(G)$ is integral complete.\vspace{2.4mm}

(b) If $(E,+)$ is $C^0$-regular, then $E$ is integral complete, by (a).
Conversely, suppose that $E$ is integral complete.
Then the weak integral $\eta(t):=\int_0^t \gamma(s)\,ds$
exists in~$E$ for
each $\gamma\in C([0,1],E)$ and $t\in [0,1]$,
and $\eta\colon [0,1]\to E$ so obtained is a $C^1$-map
such that $\eta(0)=0$ and $\delta^\ell\eta=\eta'=\gamma$
(by the Fundamental Theorem of Calculus, \cite{GaN}) and thus $\eta=\Evol(\gamma)$.
Now $\evol\colon C([0,1],E)\to E$, $\gamma\mto\int_0^1\gamma(s)\,ds$
is linear, and its continuity (and hence smoothness)
follows from the estimate $p(\int_0^1\gamma(s)\,ds)\leq \sup\{p(\gamma(s))\colon s\in [0,1|\}$,
valid for each continuous seminorm~$p$ on~$E$
and $\gamma\in C([0,1],E)$.\vspace{2.4mm}

(c) If $(E,+)$ is $C^1$-regular, it is regular and thus $E$ is Mackey complete, by
\cite[Remark~II.5.3\,(b)]{SUR}.
Conversely, suppose that $E$ is Mackey complete.
Then the weak integral $\eta(t):=\int_0^t \gamma(s)\,ds$
exists in~$E$ for
each $\gamma\in C^1([0,1],E)$ and $t\in [0,1]$,
and we conclude as in the proof of (b) that $(E,+)$ is $C^1$-regular.\vspace{2.4mm}

(d) According to \cite[p.\,267]{Voi},  the Example 4.6.110 on p.\,244 in \cite{Wil}
furnishes a locally convex space~$E$
which is Mackey complete but does not have the metric convex compactness
property. Thus $E$ fails to be integral complete,
and hence $(E,+)$ is a Lie group which is $C^1$-regular (by (c)) but not $C^0$-regular
(by (b)).\Punkt
\begin{rem}\label{betterC}
The proof of Theorem C\,(a)
remains intact if we assume instead that $G$ is $C^k$-semiregular for some $k\in \N_0\cup\{\infty\}$ and
$\evol\colon C^k([0,1],\cg)\to G$ (and hence also $\Evol\colon C^k([0,1],\cg)\to C^{k+1}_*([0,1],G)$) is $C^1$.
If we identify $L(C^{k+1}_*([0,1],G))$ with $C^{k+1}([0,1],\cg)$ using
$d(\phi_*)|_{T_1P}$, the conclusion is that $(d\delta^\ell)|_{T_1P}$ is the bijective map $\gamma\mto\gamma'$
whose inverse yields a map
\begin{equation}\label{TofEvol}
T_0\Evol\colon C^k([0,1],\cg)\to C^{k+1}([0,1],\cg)
\end{equation}
taking each $\gamma\in C^k([0,1],\cg)$ to a primitive.
Now consider $\evol\colon C^k([0,1],\cg)\to G$.
Evaluating at $t=1$,
we deduce from (\ref{TofEvol}) that $T_0\evol$ is given by
\begin{equation}\label{Tofevol}
I_\cg:=T_0\evol\colon C^k([0,1],\cg)\to\cg,\quad
\gamma\mto\int_0^1\gamma(s)\, ds;
\end{equation}
notably, the indicated weak integral exists.
This will be useful later, in the proof of Theorem~E (i.e., Proposition~\ref{getsmooth}).
\end{rem}
\section{Proofs for Theorems D and  E}\label{secDE}
Let $G$ be a Lie group with Lie algebra~$\cg$.
If $\gamma,\eta\in C^1([0,1],G)$, then we have
\[
\delta^\ell(\gamma\eta)(t)=\Ad(\eta(t))^{-1}\delta^\ell(\gamma)(t)+\eta(t)\quad \mbox{for all $t\in [0,1]$}
\]
(or $\delta^\ell(\gamma\eta)=\Ad(\eta)^{-1}.\delta^\ell(\gamma)+\eta$, in short)
and
\[
\delta^\ell(\gamma\eta^{-1})=\Ad(\eta)(\delta^\ell(\gamma)-\delta^\ell(\eta))
\]
(with $(\gamma\eta)(t):=\gamma(t)\eta(t)$ and $\gamma^{-1}(t):=\gamma(t)^{-1}$).
If $k\in \N_0\cup\{\infty\}$ and $G$ is $C^k$-semiregular
with evolution $\Evol\colon C^k([0,1],\cg)\to C^{k+1}([0,1],G)$,
then $\Evol\colon C^k([0,1],\cg)\to C^{k+1}_*([0,1],G)$ is a bijection
whose inverse is the map $\delta^\ell\colon C^{k+1}_*([0,1],G)\to C^k([0,1],\cg)$.
We define, for $\gamma, \eta\in C^k([0,1],\cg)$,
\begin{equation}\label{deltaprod}
\gamma*\eta:=\Ad(\Evol(\eta))^{-1}\gamma+\eta.
\end{equation}
Then
$C^k([0,1],\cg)$ becomes a group with multiplication~$*$
and the mapping\linebreak
$\Evol\colon C^k([0,1],\cg) \to C^{k+1}_*([0,1],G)$
becomes an isomorphism of groups (by the preceding calculations).
Moreover,
\begin{equation}\label{easiermap}
\gamma*\eta^{-1}=\Ad(\Evol(\eta))(\gamma-\eta).
\end{equation}
Using the identity map $\id$ from the group
$C^k([0,1],\cg)$ to the locally convex space $C^k([0,1],\cg)$
as a global chart, we consider $C^k([0,1],\cg)$
as a smooth manifold.
Note that $\rho_\eta(\gamma):=
\gamma*\eta :=\Ad(\Evol(\eta))^{-1}.\gamma+\eta$ is an affine-linear function of~$\gamma$ for fixed~$\eta$.
Moreover, $\rho_\eta$ is continuous (as we presently verify).
Thus $\rho_\eta$ is smooth, with differential
\begin{equation}\label{drho}
d\rho_\eta(\gamma,\gamma_1)=
\Ad(\Evol(\eta))^{-1}\gamma_1=\gamma_1 *\eta-\eta.
\end{equation}
Likewise,
\begin{equation}
d\rho_{\eta^{-1}}(\gamma,\gamma_1)=
\Ad(\Evol(\eta))\gamma_1=(\gamma_1+\eta)*\eta^{-1}.
\end{equation}
Since $\rho_{\eta^{-1}}$ and $\rho_\eta$ are mutually inverse mappings,
we see that the right translation $\rho_\eta\colon C^k([0,1],\cg)\to C^k([0,1],\cg)$
is a $C^\infty$-diffeomorphism for each $\eta\in C^k([0,1],\cg)$.
In particular, each $\rho_\eta$ is a homeomorphism.\\[2.3mm]
Since $\Evol$ is a homomorphism of groups
and so is the point evaluation $\ev_1\colon C^{k+1}([0,1],G)\to G$, $\gamma\mapsto\gamma(1)$,
we observe that also
$\evol\colon C^k([0,1],\cg)\to G$ is a homomorphism
with respect to~$*$.\\[2.3mm]
To establish the continuity of $\rho_\eta$ asserted above,
consider the map
\[
h\colon [0,1]\times \cg\to\cg,\quad h(t,y):=\Ad(\Evol(\eta)(t))y+\eta(t)
\]
which is $C^{k+1,0}$ by \ref{baprels}~(c) and~(g).
Then $\rho_\eta$ coincides with the map
\[
h_*\colon C^k([0,1],\cg)\to C^k([0,1],\cg),\quad \gamma\mto h_*(\gamma)
\]
with $h_*(\gamma)(t):=h(t,\gamma(t))$, which is continuous by
\cite[Proposition 3.10]{QUO} (which holds just as well for non-open domains).\\[2.3mm]
{\bf Proof of Theorem~D.}
Let $\gamma,\eta\in C^k([0,1],\cg)$ and assume that the map\linebreak
$\evol\colon C^k([0,1],\cg)\to G$ is continuous at~$\eta$.
For each $g\in G$, the right translation $\rho^G_g\colon G\to G$, $x\mapsto xg$ is a homeomorphism.
Since $\evol$ is a homomorphism, we have $\evol=\rho^G_{\evol(\eta^{-1}*\gamma)}\circ
\evol\circ \rho_{\gamma^{-1}*\eta}$.
As $\rho_{\gamma^{-1}*\eta}$ takes $\gamma$ to $\eta$ and is continuous,
we see that $\evol$ is continuous at~$\gamma$
if it is continuous at~$\eta$.\vspace{2mm}\Punkt

\noindent
{\bf Proof of Theorem E.} The following observation will be used:
\begin{la}\label{ckmfd}
Let $k\in \N$. If $f\colon M\to N$ is a $C^1$-map between $C^{k+1}$-manifolds
such that $Tf\colon TM\to TN$ is $C^k$, then $f$ is $C^{k+1}$.
\end{la}
\begin{proof}
After composing with charts and inverses of charts,
we may assume that $M$ and $N$ are open subsets
of locally convex spaces. By the hypotheses, $f$ is $C^1$ and $df$ (being the second component of $Tf$)
is $C^k$. So $f$ is $C^{k+1}$.
\end{proof}
Now Theorem E is subsumed by the following proposition.
\begin{prop}\label{getsmooth}
Let $k\in \N_0\cup\{\infty\}$ and $G$ be a $C^k$-semiregular Lie group.
\begin{itemize}
\item[\rm(a)]
If the map $\evol\colon C^k([0,1],\cg)\to G$ is continuous, then
$(C^k([0,1],\cg),*)$ is a topological group.
\item[\rm(b)]
If $\evol\colon C^k([0,1],\cg)\to G$ is $C^1$,
then $\evol$ is $C^\infty$ and $(C^k([0,1],\cg),*)$
is a Lie group.
\end{itemize}
\end{prop}
\begin{proof}
The map
\[
h\colon G\times \cg\times \cg\to\cg,\quad
(x,y,z)\mto \Ad(x).y+z
\]
is smooth, entailing that the map
\[
C^k([0,1],h)\colon C^k([0,1],G)\times C^k([0,1],\cg)\times C^k([0,1],\cg)\to C^k([0,1],\cg),
\]
$(\theta,\gamma,\eta)\mto h\circ (\theta,\gamma,\eta)$ is smooth (see \cite[Lemma 3.4 and Proposition 3.10]{QUO},
which hold just as well for manifolds with boundary).
For $\eta\in C^k([0,1],\cg)$, we have
\begin{equation}\label{givestar}
\gamma*\eta^{-1}=\Ad(\Evol(\eta)).\gamma+\eta=C^k([0,1],h)(\Evol(\eta),\gamma,\eta).
\end{equation}
(a) If $\evol\colon C^k([0,1],\cg)\to G$ is continuous, then
the evolution map\linebreak
$\Evol\colon C^k([0,1],\cg)\to C^{k+1}([0,1],G)$ is continuous (by Lemma~\ref{evolEvol}).
Combining this with (\ref{givestar}), we see that the map
\begin{equation}\label{topgpmap}
C^k([0,1],\cg)\times C^ k([0,1],\cg)\to C^k([0,1],\cg),\quad (\gamma,\eta)\mto \gamma*\eta
\end{equation}
is continuous. Thus $(C^k([0,1],\cg),*)$ is a topological group.

(b) We first note that if $\evol$ is $C^m$ for some $m\in \N_0\cup\{\infty\}$,
then also $\Evol\colon C^k([0,1],\cg)\to C^{k+1}([0,1],G)$ is $C^m$ (by Lemma~\ref{evolEvol}).
Hence (\ref{givestar}) shows that the map in (\ref{topgpmap})
is $C^m$ (and hence also the mappings $\eta\mto\eta^{-1}$ and $(\gamma,\eta)\mto\gamma*\eta)$.

By the preceding, it suffices to show that $\evol\colon C^k([0,1],\cg)\to G$ is $C^m$
for all $m\in \N$; then $(C^k([0,1],\cg),*)$ will be a Lie group.
We proceed by induction on $m$. For  $m=1$, we have that $\evol$ is $C^m$ (i.e., $C^1$)
by hypothesis.

If we already know that $\evol$ is $C^m$, then also the map (\ref{topgpmap})
is $C^m$ (as well as $\eta\mto\eta^{-1}$ and $(\gamma,\eta)\mto\gamma*\eta$),
by our preparatory consideration. Notably,
right translations in both $(C^k([0,1],\cg),*)$ and $G$ are diffeomorphisms
and
\[
\evol=\rho_{\evol(\eta)}\circ \evol\circ \rho_{\eta^{-1}},
\]
entailing that
\[
T_\eta\evol=T\rho_{\evol(\eta)}\circ T_0\evol\circ T_\eta\rho_{\eta^{-1}}
\]
is the map taking $(\eta,\gamma)\in T_\eta C^k([0,1],\cg)=\{\eta\}\times C^k([0,1],\cg)$
to
\[
I_\cg((\gamma+\eta)*\eta^{-1}).\evol(\eta).
\]
As $\evol$ is $C^k$ by the inductive hypothesis, the preceding formula shows
that $T\evol$ is a $C^k$-map in $(\eta,\gamma)$.
Hence $\evol$ is $C^{k+1}$, by Lemma~\ref{ckmfd}.
\end{proof}
\section{Proof of Theorem F}\label{secF}
We shall use the following tool:
\begin{la}\label{getC1tool} Let $M$ and $N$ be $C^1$-manifolds modeled on locally convex spaces
and $f\colon M\to N$ be a map.
Then $f$ is $C^1$ if and only if there exists a continuous map $\omega \colon TM\to TN$ with
the following properties:
\begin{itemize}
\item[\rm(a)]
$\omega(T_xM)\sub T_{f(x)}N$ for each $x\in M$;
\item[\rm(b)]
If $\ve>0$ and $\gamma\colon \,]{-\ve},\ve[\,\to M$ is a $C^1$-map,
then $f\circ \gamma$ is $C^1$
with $(f\circ\gamma)'(0)=\omega(\gamma'(0))$.
\end{itemize}
In this case, $Tf=\omega$.
If $M$ is an open subset of a locally convex space~$X$,
it suffices to take paths of the form $\gamma(s)=x+sy$ in {\rm(b)},
for $x\in M$ and $y\in X$.
\end{la}
\begin{proof}
If $f$ is $C^1$, we can take $\omega:=Tf$.
Conversely, assume that~$\omega$ exists.
Using the $0$-section $\theta\colon M\to TM$, $x\mapsto 0_x$
and the bundle projection $\pi_N\colon TN\to N$,
we have $f=\pi_N\circ \omega\circ \theta$ by (a), and thus~$f$ is continuous.
If $\psi\colon P\to Q$ is a chart for~$N$ and $\phi\colon U\to V$ is a chart for $M$
(which we may choose as the identity map
of an open subset $U\sub M$ if $M$ is open in~$X$)
such that $f(U)\sub P$, then $d(\psi\circ f\circ\phi^{-1})(x,y)=\frac{d}{dt}|_{t=0}
\psi((f\circ \phi^{-1})(x+ty))$ exists by~(b) and the Chain Rule, and is given by
\[
d(\psi\circ f\circ\phi^{-1})(x,y)=d\psi(\omega(T\phi^{-1}(x,y)))
\]
and thus continuous in $(x,y)$.
Hence $\psi\circ f\circ \phi^{-1}$ is $C^1$ and hence $f$ is $C^1$.
\end{proof}
\begin{la}\label{alsorsemi}
If $G$ is a $C^0$-semiregular Lie group,
then for each $\gamma\in C^0([0,1],\cg)$
there exists $\Evol^r(\gamma):=\eta\in C^1_*([0,1],G)$
such that $\delta^r(\eta)=\gamma$.
\end{la}
\begin{proof}
The right logarithmic derivative of $\Evol(-\gamma)^{-1}$
is $\gamma$, whence $\Evol^r(\gamma)=\Evol(-\gamma)^{-1}$.
\end{proof}
\begin{la}\label{ifthn}
If $G$ is $C^0$-regular and $\gamma,\eta\in C^0([0,1],\cg)$,
then
\[
\theta\colon \R\to G,\quad s\mto \evol(\eta+s\gamma)
\]
is a smooth curve in $G$ with right logarithmic derivative
\begin{equation}\label{horrib}
\delta^r(\theta)(s)=I_\cg(\Ad(\Evol(\eta+s\gamma)).\gamma)
=I_\cg((\eta+s\gamma+\gamma)*(\eta+s\gamma)^	{-1})
\quad\mbox{for $s\in\R$,}
\end{equation}
where $I_\cg\colon C^0([0,1],\cg)\to \cg$, $\gamma\mto \int_0^1\gamma(t)\,dt$.
\end{la}
\begin{proof}
Since $\evol\colon (C^0([0,1],\cg),*)\to G$ is a smooth homomorphism between Lie groups,
we have
\begin{eqnarray*}
(\delta^r\theta)(s)&=&L(\evol)\delta^r(s\mto \eta+s\gamma)\\
&=&I_\cg d\rho_{(\eta+s\gamma)^{-1}}(\eta+s\gamma,\gamma)\\
&=&I_\cg(\Ad(\Evol(\eta+s\gamma)).\gamma),
\end{eqnarray*}
as asserted.
\end{proof}
{\bf Proof of Theorem F.}
Since $\evol$ is continuous by assumption,
$(C^0([0,1],\cg),*)$ is a topological group by Proposition~\ref{getsmooth}\,(a).
Thus
$\omega\colon TC^0([0,1],\cg)\to TG$,
\begin{eqnarray*}
\omega(\eta,\gamma)&:=&I_\cg(d\rho_{\eta^{-1}}(\eta,\gamma)).\evol(\eta)
=I_\cg((\gamma+\eta)*\eta^{-1}).\evol(\eta)\\
&=&I_{\cg}(\Ad(\Evol(\eta)).\gamma).\evol(\eta)
\end{eqnarray*}
is continuous.
Moreover, $\omega$ takes $T_\eta C^0([0,1],\cg)=\{\eta\}\times C^0([0,1],\cg)$
to $T_{\evol(\eta)}G$.
Hence, by Lemma~\ref{getC1tool}, $\evol$ will be $C^1$ if we can show that,
for all $\gamma,\eta\in C^0([0,1],\cg)$, the curve
\[
\xi\colon \R\to G,\quad \xi(s):=\evol(\eta+s\gamma)
\]
is $C^1$ and satisfies
\begin{equation}\label{willengh}
\xi'(0)=\omega(\eta,\gamma).
\end{equation}
Stimulated by (\ref{horrib}), we use Lemma~\ref{alsorsemi} to find a $C^1$-curve $\theta\colon \R\to G$, $s\mto \theta(s)$
such that $\theta(0)=\evol(\eta)$ and
\[
\delta^r(\theta)(s)=I_\cg(\Ad(\Evol(\eta+s\gamma)).\gamma)
=I_\cg((\eta+s\gamma+\gamma)*(\eta+s\gamma)^{-1}).
\]
For each $j\in J$, we have that
\begin{eqnarray*}
\delta^r(\alpha_j\circ \theta)(s)&=&L(\alpha_j)\delta^r(\theta)(s)\\
&=&L(\alpha_j)I_\cg(\Ad(\Evol(\eta+s\gamma)).\gamma)\\
&=&I_{\ch_j}(L(\alpha_j)\circ \Ad(\Evol(\eta+s\gamma)).\gamma)\\
&=&I_{\ch_j}(\Ad(\Evol(L(\alpha_j)(\eta)+sL(\alpha_j)(\gamma)).L(\alpha_j)(\gamma))
\end{eqnarray*}
This coincides with
\begin{eqnarray*}
\delta^r(\alpha_j\circ \xi)&=& \delta^r(s\mto \alpha_j(\evol(\eta+s\gamma)))\\
&=& \delta^r(s\mto \evol(L(\alpha_j)\circ (\eta+s\gamma)))\\
&=& \delta^r(s\mto \evol(L(\alpha_j)(\eta)+sL(\alpha_j)(\gamma)))\\
&=& I_{\ch_j}(\Ad(\Evol(L(\alpha_j)(\eta)+sL(\alpha_j)(\gamma)).L(\alpha_j)(\gamma))
\end{eqnarray*}
(where we used (\ref{horrib}) for the last equality).
Hence $\alpha_j\circ \theta=\alpha_j\circ \xi$ for all $j\in J$,
entailing (as the $\alpha_j$ separate points)
that $\xi=\theta$ is indeed $C^1$,
with $\xi'(0)=\theta'(0)=(\delta^r\theta)(0).\theta(0)=I_\cg(\Ad(\Evol(\eta)).\gamma).\evol(\eta)
=\omega(\eta,\gamma)$.\Punkt
\begin{rem}\label{variantThmF}
Let $G$ be a $C^0$-semiregular Lie group
whose Lie algebra $\cg:=L(G)$
is an integral complete locally convex space.
Assume that the smooth homomorphisms
from $G$ to $C^0$-regular Lie groups separate
points on $G$ and
\[
\evol\colon C^1([0,1],\cg)\to G
\]
is continuous. Then $G$ is $C^1$-regular.
In fact, we can simply replace
$C^0([0,1],\cg)$ with $C^1([0,1],\cg)$ in the preceding
proof (in particular, we now consider
$\omega\colon TC^1([0,1],\cg)\to TG$).
\end{rem}
\noindent
The first lines of the proof of Theorem~F also show the following.
\begin{prop}\label{moresitu}
Let $k\in \N_0\cup\{\infty\}$ and $G$ be a $C^k$-semiregular Lie group,
with Lie algebra $\cg$. If $k=0$, assume that $\cg$ is integral complete;
if $k\geq 1$, assume that $\cg$ is Mackey-complete.
Assume that
\begin{itemize}
\item[\rm(a)]
$\evol_G\colon C^k([0,1],\cg)\to G$ is continuous;
\item[\rm(b)] The map
$\R\to G$, $t\mto \evol_G(\gamma+t\eta)$
is $C^1$ for all $\gamma,\eta\in C^k([0,1],\cg)$; and
\item[\rm(c)]
$\frac{d}{dt}\Big|_{t=0}\evol_G(\gamma+t\eta)
=
I_{\cg}(\Ad(\Evol_G(\eta))\gamma).\evol_G(\eta)$,
where the dot means multiplication in $TG$ and
$I_{\cg}\colon C^k([0,1],\cg)\to\cg$, $\zeta\mto\int_0^1\zeta(t)\,dt$.
\end{itemize}
Then $G$ is $C^k$-regular.\,\Punkt
\end{prop}
\begin{cor}\label{moredl}
Let $k\in \N_0\cup\{\infty\}$, $(J,\leq)$ be a directed set, $((G_j)_{j\in J}, (\alpha_{i,j})_{i\geq j})$
be a direct system of $C^k$-semiregular Lie groups $G_j$ and smooth homomorphisms
$\alpha_{i,j}\colon G_j\to G_i$.
Let $G$ be a Lie groups and $\alpha_j\colon G_j\to G$ be smooth homomorphisms
for $j\in J$ such that $\alpha_i\circ\alpha_{i,j}=\alpha_j$
for all $i,j\in J$ such that $i\geq j$. Let $\cg:=L(G)$ and $\cg_j:=L(G_j)$.
Assume that each $\gamma\in C^k([0,1],\cg)$
is of the form $L(\alpha_j)\circ\zeta$ for some $j\in J$ and $\zeta\in C^k([0,1],\cg_j)$.
Then $G$ is $C^k$-semiregular. If $\evol\colon C^k([0,1],\cg)\to G$
is continuous and each $G_j$ is $C^k$-regular, then $G$ is $C^k$-regular.
\end{cor}
\begin{proof}
If $\gamma\in C^k([0,1],\cg)$ and $\zeta\in C^k([0,1],\cg_j)$ such that
$\gamma=L(\alpha_j)\circ \zeta$, then $\alpha_j\circ  \Evol_{G_j}(\zeta)$
is the left evolution of $\gamma$.
Hence $G$ is $C^k$-semiregular.
Now assume that $\evol_G$ is continuous.
If $\gamma,\eta\in C^k([0,1],\cg)$,
then there are $i,j\in J$ and $\zeta\in C^k([0,1],\cg_i)$,
$\xi\in C^k([0,1],\cg_j)$ such that $\gamma=L(\alpha_i)\circ\zeta$
and $L(\alpha_j)\circ\xi$. Since $J$ is directed,
there is $\ell\in J$ such that $\ell\geq i,j$.
Let $\bar{\zeta}:=L(\alpha_{\ell,i})\circ\zeta$
and $\bar{\xi}:=L(\alpha_{\ell,j})\circ\xi$.
Then $\bar{\zeta}$, $\bar{\xi}\in C^k([0,1],\cg_\ell)$
and $L(\alpha_\ell)\circ \bar{\zeta}=L(\alpha_\ell)\circ L(\alpha_{\ell,i})\circ\zeta
=L(\alpha_\ell\circ\alpha_{\ell,i})\circ\zeta=L(\alpha_i)\circ\zeta=\gamma$;
likewise, $L(\alpha_\ell)\circ\bar{\xi}=\xi$.
Note that
\[
\theta\colon \R\to G_\ell,\quad \theta(t):=\evol_{G_\ell}(\bar{\zeta}+t\bar{\xi})
\]
is a $C^1$-curve in $G_\ell$ such that $\theta'(t)=I_{\cg_\ell}(\Ad(\Evol_{G_\ell}(\bar{\xi})\bar{\zeta})).\evol_{G_\ell}(\bar{\xi})$
Hence $\alpha_\ell\circ\theta\colon \R\to G$, $t\mto\evol_G(\zeta+t\xi)$
is a $C^1$-curve and $\frac{d}{dt}\Big|_{t=0}\evol_G(\zeta+t\xi)=
T(\alpha_\ell)(\theta'(t))=
I_{\cg}(\Ad(\Evol_{Gl}(\xi)\zeta)).\evol_{G}(\xi)$.
Now apply Proposition~\ref{moresitu}.
\end{proof}
The following side result is established using a similar (but easier) argument.
It is related to the open question of whether
every continuous homomorphism between (possibly infinite-dimensional) Lie groups
is smooth.
\begin{prop}
Let $\alpha\colon G\to H$ be a continuous homomorphism
between Lie groups modelled on locally convex spaces. Assume that
\begin{itemize}
\item[\rm(a)]
$H$ is $C^0$-semiregular;
\item[\rm(b)]
There exists a continuous linear map $\beta\colon L(G)\to L(H)$
and a family $(\alpha_j)_{j\in J}$ of smooth homomorphisms
$\alpha_j\colon H\to H_j$ to Lie groups $H_j$
such that $\alpha_j\circ\alpha\colon G\to H_j$
is smooth and $L(\alpha_j\circ \alpha)=L(\alpha_j)\circ\beta$, for all $j\in J$.
\end{itemize}
Then $\alpha$ is smooth.
\end{prop}
\begin{proof}
Let $\pi_G \colon TG \to G$ be the bundle projection.
The map
\[
\omega\colon TG\to TH,\quad v\mto \alpha(\pi_G(v)).\beta(\pi_G(v)^{-1}.v)
\]
is continuous and takes $T_gG$ to $T_{\alpha(g)}H$.
If $\gamma\colon \,]{-\ve},\ve[\,\to G$ is $C^1$,
then $\eta:=\beta\circ \delta^\ell(\gamma)$
is a continuous curve in $L(H)$. By $C^0$-semiregularity,
there exists a unique $C^1$-map $\theta\colon\,]{-\ve},\ve[\,\to H$
such that $\delta^\ell(\theta)=\eta$ and $\theta(0)=e$.
Then
\begin{eqnarray*}
\delta^\ell(\alpha_j\circ\theta)&=& L(\alpha_j)\circ\delta^\ell(\theta)=L(\alpha_j)\circ\eta\\
&=&
L(\alpha_j)\circ\beta\circ \delta^\ell(\gamma)=L(\alpha_j\circ \alpha)\circ\delta^\ell(\gamma)\\
&=&\delta^\ell((\alpha_j\circ\alpha)\circ\gamma)
\end{eqnarray*}
(using hypothesis (b) for the penultimate equality)
and thus $\alpha_j\circ \theta=\alpha_j\circ \alpha\circ\gamma$.
As the $\alpha_j$ separate points, we deduce that $\alpha\circ\gamma=\theta$, which is $C^1$.
Moreover,
$(\alpha\circ\gamma)'(s)=\theta'(s)=\theta(s).\delta^\ell(\theta)(s)=\alpha(\gamma(t)).\beta(\delta^\ell(\gamma)(s))
=\omega(\gamma'(s))$. Thus $\alpha$ is $C^1$, by Lemma~\ref{getC1tool}.
Like any $C^1$-homomorphism between Lie groups, $\alpha$ is smooth
\cite{HOL}.
\end{proof}
Let $M$ be a $C^k$-manifold modelled
on a locally convex space~$E$. Recall
that a subset $N$ of~$M$
is called a \emph{submanifold}
modelled on a closed vector subspace
$F$ of~$E$
if, for each $x\in N$,
there is a chart $\phi\colon U_\phi\to V_\phi\sub E$
of~$M$ around~$x$ such that $\phi(U_\phi\cap N)=V_\phi\cap F$.
A subgroup $H$ of a Lie group~$G$ is called
a \emph{Lie subgroup} if it is a submanifold of~$G$
in the sense just recalled.
The author is grateful to Alexander Schmeding (Trondheim)
for his suggestion to include the following lemma.
\begin{la}\label{Schmed}
Let $G$ be a $C^k$-regular Lie group with $k\in \N_0\cup\{\infty\}$
and $H\sub G$ be a Lie subgroup.
If $H$ is $C^k$-semiregular,
then $H$ is $C^k$-regular.
\end{la}
\begin{proof}
Let $\lambda\colon H\to G$ be the inclusion map,
$\ch:=L(H)$ and $\cg:=L(G)$.
For each $\gamma\in C^k([0,1],\ch)$,
we have $\delta^\ell(\lambda\circ \Evol_H(\gamma))=L(\lambda)\circ
\delta^\ell(\Evol_H(\gamma))=\lambda\circ \gamma$ and hence
$\lambda\circ \Evol_H(\gamma) = \Evol_G(L(\lambda)\circ \gamma)$.
Thus $\lambda(\evol_H(\gamma))=\evol_G(L(\lambda)\circ \gamma)$
and hence
\[
\evol_H=(\evol_G\circ C^k([0,1],\lambda))|^H.
\]
Since $C^k([0,1],\lambda)\colon C^k([0,1],\ch)\to C^k([0,1],\cg)$, $\gamma\mto
\lambda\circ \gamma$ is continuous linear and hence smooth, we deduce
that
\[
\evol_G\circ C^k([0,1],\lambda)\colon C^k([0,1]\,ch)\to G.
\]
is smooth.
As this map has image inside~$H$ and~$H$ is a submanifold of~$G$, we deduce
that $\evol_H=(\evol_G\circ C^k([0,1],\lambda))|^H$ is smooth
(see \cite{BGN}).
\end{proof}
\begin{la}\label{autoimpr}
Let $\eta\colon [0,1]\to G$ be a $C^1$-curve in a Lie group $G$
such that $\delta^\ell(\eta)\colon [0,1]\to L(G)$
is $C^k$ for some $k\in \N_0$.
Then $\eta$ is $C^{k+1}$.
\end{la}
\begin{proof}
Since $\gamma:=\delta^\ell(\eta)$ is $C^k$,
an easy induction based on the formula
\[
\eta'(t)=\eta(t).\gamma(t)\quad\mbox{for all $t\in [0,1]$}
\]
shows that $\eta'$ is $C^j$ for all $j\in \{0,1,\ldots, k\}$.
Hence $\eta$ is $C^{k+1}$.
\end{proof}
Also the following situation arises frequently.
\begin{la}\label{newfreq}
Let $f\colon G\to H$ be a smooth homomorphism
between Lie groups modelled on locally convex spaces
$E$ and $F$, respectively,
with Lie algebras~$\cg$ and~$\ch$.
Assume that $L(f)\colon \cg\to \ch$ is injective
and $\cg$ is integral complete.
Moreover, assume that
there exists a chart
$\phi\colon U_\phi\to V_\phi\sub E$
of~$G$ around~$1$, a chart $\psi\colon U_\psi\to V_\psi\sub F$ of $H$
around~$1$ and a continuous linear map $\alpha\colon E\to F$
such that $f(U_\phi)\sub U_\psi$ and
\begin{equation}\label{loclinbuin}
\psi\circ f\circ \phi^{-1}=\alpha|_{V_\phi}.
\end{equation}
Let $k\in \N_0\cup\{\infty\}$.
If $\gamma\in C^k([0,1],\cg)$
such that
\[
\Evol_H(L(f)\circ \gamma)=f\circ \eta
\]
for a continuous map $\eta\colon [0,1]\to H$
with $\eta(0)=1$,
then $\eta\in C^{k+1}([0,1],H)$ and $\eta=\Evol_G(\gamma)$.
\end{la}
\begin{proof}
We identify $E$ and $F$ with $\cg$ and $\ch$, respectively,
and may assume that the charts have been chosen with $\phi(1)=0$, $\psi(1)=0$,
$d\phi|_\cg=\id_\cg$ and $d\psi|_\ch=\id_\ch$.
By Lemma~\ref{autoimpr},
it suffices to show that $\eta$ is $C^1$ and $\delta^\ell(\eta)=\gamma$.
There is a partition $0=t_0<t_1<\cdots, t_n=1$ of $[0,1]$
such that $\eta(t_{j-1})^{-1}\eta(t)\in U_\phi$
for all $j\in \{1,\ldots, n\}$
and $t\in [t_{j-1}, t_j]$.
Define
\[
\eta_j\colon [0,1]\to G, \quad \eta_j(t):=\eta(t_{j-1})^{-1}\eta(t_{j-1}+t(t_j-t_{j-1})).
\]
It then suffices to show that
\[
\delta^\ell(\eta_j)(t)=(t_j-t_{j-1})\gamma(t_{j-1}+t(t_j-t_{j-1}))
\]
for all $j\in \{1,\ldots, n\}$ and $t\in [0,1]$. After replacing $\eta$ with $\eta_j$,
we may therefore assume that $\eta([0,1])\sub U_\phi$.
The group multiplication of $G$ induces on $U_\phi$ a structure of local Lie group;
then there is a unique local Lie group structure on~$V_\phi$
making $\phi$ an isomorphism of local Lie groups.
Likewise, $H$ and $\psi$ induce local Lie group structures on $U_\psi$ and $V_\psi$.
Then
\[
\alpha|_{V_\phi}=\psi\circ f|_{U_\phi}\circ \phi^{-1}\colon U_\phi\to V_\phi
\]
is both a homomorphism between local Lie groups,
and the restriction of a continuous linear map.
Given $x\in V_\phi$, let $\lambda_x$ be left translation with~$x$
on its natural domain $D$ in the local Lie group~$V_\phi$, and $\lambda^G_{\alpha(x)}$ be
left translation on its natural domain in~$V_\psi$.
Because $\alpha$ is a homomorphism, we have
\begin{equation}\label{thereonch}
\alpha\circ \lambda^G_x=\lambda_{\alpha(x)}^H\circ \alpha|_D.
\end{equation}
Since $\alpha$ is continuous linear, applying the Chain Rule to (\ref{thereonch}),
we obtain
\[
\alpha\circ d\lambda^G_x =d\lambda^H_{\alpha(x)}\circ (\alpha|_D\times \alpha).
\]
Let $\zeta:=\phi\circ \eta$.
Since $E$ is integral complete,
we can define a $C^1$-curve $\theta\colon [0,1]\to E$ via
\[
\theta(t):=\int_0^t d\lambda_{\phi(\gamma(s))}^G(\gamma(s))\,ds.
\]
Then
\begin{eqnarray*}
\alpha(\theta(t)) &=& \int_0^t \alpha(d\lambda_{\phi(\gamma(s))}^G(\gamma(s))\,ds
= \int_0^t d\lambda_{\alpha(\phi(\gamma(s)))}^H(\alpha(\gamma(s)))\,ds\\
&=& \int_0^t d\lambda_{\psi(f(\gamma(s)))}^H(\alpha(\gamma(s)))\,ds
= \psi(\Evol^H(\alpha\circ \gamma)(t))\\
&=& (\psi\circ f\circ \eta)(t)=\alpha((\phi\circ \eta)(t))
\end{eqnarray*}
and hence $\phi\circ \eta=\theta$, which is $C^1$.
Moreover,
\[
(\phi\circ \eta)'(t)=\theta'(t)=
d\lambda_{\phi(\gamma(t))}^G(\gamma(t))
\]
for all $t\in [0,1]$. Hence
\[
\delta^\ell(\eta)=L(\phi)\delta^\ell(\eta)=\delta^\ell(\phi\circ\eta)
=\gamma
\]
and thus $\eta=\Evol^G(\gamma)$.
\end{proof}
\begin{example}
Recall that a Lie group is called a \emph{BCH-Lie group}
if it is an analytic Lie group whose exponential function
induces an analytic diffeomorphism from an
open $0$-neighbourhood onto an open identity neighbourhood
(see \cite{QUO}, \cite{GaN}, and \cite{SUR}).
If $f\colon G\to H$ is an analytic homomorphism
between BCH-Lie groups, then $f\circ \exp_G=\exp_H\circ L(f)$,
showing that $f$ corresponds to the restriction of the continuous
linear map $L(f)$ to an open $0$-neighbourhood in $L(G)$
on which $\exp_G$ is a diffeomorphism.
The hypothesis (\ref{loclinbuin})
in Lemma~\ref{newfreq} is therefore satisfied.
\end{example}
%
%
%
%
%
%
%
%
%
%
%
%
%
\section{Proof of Theorem G}\label{secG}
We first formulate and prove an analogous result for \emph{local}
Lie groups. Part of Theorem~G then is a special case. 
\begin{numba}
Recall that a \emph{local Lie group} is
a quintuple $(U,D,\mu,\eta,e)$
where $U$ is a smooth manifold modelled
on a locally convex space $E$,
$D\sub U\times U$ is an open subset,
$\mu\colon D\to U$, $(x,y)\mto \mu(x,y)=:xy$ and $\eta\colon U\to U$, $x\mto \eta(x)=:x^{-1}$
are smooth maps and $e\in U$, such that natural axioms are satisfied (see \cite{SUR} or \cite{GaN}
for details
and further basic theory and notation).
In particular, $(x,e), (e,x), (x,x^{-1})\in D$ for all $x\in U$.
If $E$ is a locally convex space over $\K\in \{\R,\C\}$
and $U$, $\mu$ and $\eta$ are $\K$-analytic,
then we speak of a \emph{$\K$-analytic local Lie group}.
We write $L(U):=T_eU$ for the Lie algebra of a local Lie group
and recall that $g^{-1}.y\in L(U)$ can be defined for $g\in U$, $y\in T_g(U)$
as in the case of a (global) Lie group,
using left translation with $g^{-1}$ on some open neighbourhood of~$g$.  
Similarly, we can define $g.y\in T_gU$ for $g\in U$, $y\in L(U)$.
\end{numba}
\begin{numba}\label{deflocreg}
A local Lie group (as before) is called \emph{$C^k$-semiregular}
if there exists an open $0$-neighbourhood
$\Omega\sub C^k([0,1],L(U))$
such that each $\gamma\in \Omega$
has a left evolution $\Evol(\gamma):=\eta\in C^{k+1}([0,1],U)$,
determined by $\eta(0)=e$, $\delta^\ell(\eta)(t):=\eta(t)^{-1}.\eta'(t)=\gamma(t)$ for all $t\in [0,1]$.
As in the case of (global) Lie groups, left evolutions
are always unique.
If $\Omega$ can be chosen such that $\evol\colon\Omega\to U$,
$\gamma\mto \Evol(\gamma)(1)$ is smooth,
then the local group is called \emph{$C^k$-regular}.
\end{numba}
\begin{numba}\label{sublocal}
The validity of $C^k$-regularity or $C^k$-semiregularity is unchanged
if we replace $D$ by any open subset $D'$
for which $(U,D',\mu|_{D'},\eta,e)$ is a local Lie group.
\end{numba}
\begin{numba}\label{opnsub}
Let $(U,D,\mu,\eta,e)$ be a local Lie group.
If $V=V^{-1}$ is an open identity neighbourhood
and $D_V\sub D\cap V\times V$
an open subset with $\mu(D_V)\sub V$
such that $(V,D_V,\mu|_{D_V}^V,\eta|_V^V,e)$ is a local group,
then the latter is called an \emph{open local Lie subgroup}
of the given local Lie group.
For example, we can always take $D_V:=(V\times V)\cap\mu^{-1}(V)$.
It is clear that $U$ is $C^k$-regular if and only if $V$
is $C^k$-regular. If $V$ is $C^k$-semiregular, then also $U$ is $C^k$-semiregular.
\end{numba}
\begin{numba}\label{insidenearcircular}
The mapping $h\colon [0,1]\times \Omega\to C^k([0,1],L(U))$
determined by $h(s,\gamma)(t)=s\gamma(st)$ is continuous
%
%
 and $h([0,1]\times \{0\})=\{0\}$.
Hence, by the Wallace Lemma,
there exists
an open $0$-neighbourhood $W\sub\Omega$ such that $h([0,1]\times W)\sub\Omega$.
For each $\gamma\in W$ and $t\in [0,1]$, we then have
\[
\Evol(\gamma)(s)=\evol(s\gamma(s.))=\evol(h(s,\gamma)).
\]
\end{numba}
\begin{numba}\label{globalegal}
It is well known that a (global) Lie group $G$ is $C^k$-regular as a Lie group
if and only if it is $C^k$-regular as a local Lie group
(see, e.g., \cite{Dah}).
The same argument shows that $G$ is $C^k$-semiregular as a Lie group
if and only if $G$ is $C^k$-semiregular as a local Lie group.
\end{numba}
\begin{numba}
Let $(U_j,D_j,\mu_j,\eta_j,e_j)$ be local Lie groups for $j\in \{1,2\}$.
A map $\alpha\colon U_1\to U_2$ is called a \emph{homomorphism
of local Lie groups} if $\alpha$ is smooth, $\alpha(e_1)=e_2$, $(\alpha\times \alpha)(D_1)\sub D_2$,
$\alpha(x,y)=\alpha(x)\alpha(y)$ for all $(x,y)\in D_1$
and $\alpha(x^{-1})=\alpha(x)^{-1}$ for all $x\in U_1$.
\end{numba}
\begin{prop}\label{thmGloc}
Let $U$ be a local Lie group and
\[
W:=\{x\in U\colon (\forall j\in J)\;\alpha_j(x)=\beta_j(x)\}
\]
for families $(\alpha_j)_{j\in J}$
and $(\beta_j)_{j\in J}$ of smooth homomorphisms
$\alpha_j,\beta_j\colon U\to U_j$ to some local Lie groups~$U_j$.
Let $k\in \N_0\cup\{\infty\}$.Then the following holds:
\begin{itemize}
\item[\rm(a)]
If~$U$ is $C^k$-semiregular and $W$ is $C^m$-initial in~$U$
for some $m\in \N\cup\{\infty\}$
such that $m\leq k+1$, then also~$W$ is $C^k$-semiregular.
\item[\rm(b)]
If $U$ is $C^k$-regular and
$W$ is $C^m$-initial and $C^\infty$-initial in~$U$
for some $m\in \N\cup\{\infty\}$
such that $m\leq k+1$,
then also~$W$ is $C^k$-regular.
\end{itemize}
\end{prop}
Here, a subset $W=W^{-1}$of $U$ containing $e$
is called \emph{$C^m$-initial}
if it admits a local Lie group structure
$(W,D,\mu,\eta,e)$ with properties
like those in the case of $C^m$-initial Lie subgroups.
While the Lie group structure on a $C^m$-initial
Lie subgroup is unique, $D$ is not unique in the current situation.
The proposition holds for each fixed choice of~$D$
(cf.\ also \ref{sublocal}).\\[2mm]
\begin{proof}
(a) By hypothesis, there exists an open $0$-neighbourhood $\Omega\sub C^k([0,1],L(U))$
such that each $\gamma\in \Omega$ admits a $C^{k+1}$-evolution $\Evol(\gamma)\colon [0,1]\to U$.
Let $\iota\colon W\to U$ be the smooth inclusion map.
Then $L(\iota)\colon L(W)\to L(U)$ is an injective, continuous linear (and hence smooth) map
and thus also $L(\iota)_*\colon C^k([0,1],L(W))\to C^k([0,1],L(U))$, $\gamma\mto L(\iota)\circ \gamma$
is continuous linear (cf.\ \cite{QUO}). We identify $L(W)$ with the image $\im\,L(\iota)\sub L(U)$
as a vector space (although the topology may be finer).
By the preceding,
\[
\Omega_W:=\Omega\cap C^k([0,1],L(W))=(L(\iota)_*)^{-1}(\Omega)
\]
is an open $0$-neighbourhood in $C^k([0,1],L(W))$.
Note that
\begin{equation}\label{infteseq}
L(\alpha_j).y=L(\alpha_j)L(\iota).y=L(\alpha_j\circ\iota).y=L(\beta_j\circ\iota).y=L(\beta_j).y
\end{equation}
for all $j\in J$ and $y\in L(W)$.
Given $\gamma\in \Omega_W$, let $\eta:=\Evol(\gamma)\colon [0,1]\to U$.
For each $j\in J$, we have that
\[
\delta^\ell(\alpha_j\circ\eta)=L(\alpha_j)\circ \delta^\ell\eta=L(\alpha_j)\circ \gamma
=L(\beta_j)\circ\gamma=\delta^\ell(\beta_j\circ \eta),
\]
using (\ref{infteseq}) to get the third equality.
Hence $\alpha_j\circ \eta=\beta_j\circ\eta$ for all $j\in J$ and thus
$\eta([0,1])\sub W$. Since $\eta$ is $C^{k+1}$ and thus $C^m$ as a map to $U$
and $W$ is $C^m$-initial, we deduce that $\eta$ is $C^m$ as a map to $W$,
hence $C^1$ as a map to $W$ and hence $C^{k+1}$ as a map to $W$
(using that $\eta'=\eta.\delta^\ell\eta=\eta.\gamma$ is $C^k$ as a map to $TW$).
Thus $W$ is $C^k$-semiregular.

(b) If $U$ is $C^k$-regular, choose $\Omega$ in the proof of~(a) such that
$\evol\colon \Omega\to U$ is smooth. Then also $\evol\circ L(\iota)_*|_{\Omega_W}$
is smooth as a map to $U$ and takes values in $W$.
Thus, by $C^\infty$-initiality, $\evol_W:=\evol\circ L(\iota)_*|_{\Omega_W}^W\colon
\Omega_W\to W$ is $C^\infty$.
\end{proof}
\noindent
{\bf Proof of Theorem G.}
(a) follows from Proposition~\ref{thmGloc}\,(a) and \ref{globalegal}.

(b) Let $\iota\colon H\to G$ be the inclusion map.
We know that $\evol C^k([0,1],\cg)\to G$ is smooth and hence $C^m$.
As in the preceding proof, $\evol_H(\gamma):=\evol(\gamma)\in H$
for all $\gamma\in C^k([0,1],\ch)$.
Thus $\evol_H\colon C^k([0,1],\ch)\to H$ is $C^m$,
as we assume that $H$ is $C^m$-initial.
As a consequence, $\evol_H$ is $C^1$ and thus $\evol_H$ is $C^\infty$,
by Theorem~E.\vspace{2.8mm}\Punkt
\begin{rem}
It should be possible to omit the $C^\infty$-initiality in Proposition~\ref{thmGloc}\,(b),
if we first prove an analogue of Theorem~E for local Lie groups.
However, the author did not want to make the presentation more technical.
\end{rem}
\begin{example}\label{mapstar}
Let $n\in \N$, $k^,\ell\in \N_0\cup\{\infty\}$ and $H$ be a $C^k$-regular Lie group,
modelled on a locally convex space~$E$.
Then also the Lie group
\[
C^\ell_*([{-n},n],H):=\{\gamma\in C^\ell([{-n},n],H)\colon \gamma(0)=e\}
\]
is $C^k$-regular.\\[2.3mm]
In fact, if $U=U^{-1}$ is an open identity neighbourhood in~$H$ and $\phi\colon U\to V\sub E$
a chart for~$H$ such that $\phi(e)=0$, then $\phi_*\colon C^\ell([{-n},n],U)\to C^\ell([{-n},n],V)$, $\gamma\mto\phi\circ\gamma$
is a chart for $C^\ell([{-n},n],H)$.
We shall verify in Proposition~\ref{cpcureg} that $C^\ell([{-n},n],H)$ is $C^k$-regular.
Since
\[
\phi_*(C^\ell([{-n},n],U)\cap C^\ell_*([{-n},n],H))=C^\ell([{-n},n],V)\cap C^\ell_*([{-n}.n],E),
\]
$C^\ell_*([{-n},n],H)$ is a Lie subgroup in $C^\ell([{-n},n],H)$ and hence $C^m$-initial
for each $m\in \N_0\cup\{\infty\}$.
Moreover,
\[
C^\ell_*([{-n},n],H)=\{\gamma\in C^\ell([{-n},n],H)\colon \alpha(\gamma)=\beta(\gamma)\}
\]
with the smooth homomorphisms $\alpha,\beta\colon C^\ell([{-n},n],H)\to H$,
\[
\alpha(\gamma):=\gamma(0),\qquad \beta(\gamma):=e.
\]
Thus Theorem~G applies and shows that $C^\ell_*([{-n},n],H)$
inherits $C^k$-regularity from $C^\ell([{-n},n],H)$.
\end{example}
A more substantial application will be given in Proposition~\ref{gaugereg}.
\section{Application to complexifications and\\
projective limits}\label{cxpl}
\begin{prop}\label{regpl}
Let $(I,\leq)$ be a directed set and $\cS:=((G_i)_{i \in I},(\phi_{i,j})_{i\leq j})$
be a projective system of
Lie groups $G_i$ and smooth homomorphisms $\phi_{i,j}\colon G_j\to G_i$.
Let $G$ be a Lie group, $k\in \N_0\cup\{\infty\}$,
$m\in \N\cup\{\infty\}$ such that $m\leq k+1$ $($e.g., $m=k+1)$
and $\phi_i\colon G\to G_i$
be smooth homomorphisms such that $(G,(\phi_i)_{i\in I})$
is the projective limit of $\cS$ in the category of $C^m$-manifolds
$($possibly with boundary$)$ and $C^m$-maps between them,
and the maps $L(\phi_i)$ separate points on $L(G)$
for $i\in I$.
\begin{itemize}
\item[\rm(a)]
If each Lie group $G_i$ is $C^k$-semiregular, then also $G$ is $C^k$-semiregular.
\item[\rm(b)]
If each Lie group $G_i$ is $C^k$-regular and $(G,(\phi_i)_{i\in I})$ is the projective limit of $\cS$
also in the category of sets and maps,
then $G$ is $C^k$-regular.
\end{itemize}
\end{prop}
\begin{proof}
(a) Let $\gamma\colon [0,1]\to L(G)$ be $C^k$.
Then $L(\phi_i)\circ \gamma\colon [0,1]\to L(G_i)$ is $C^k$ and thus
$\eta_i:=\Evol_{G_i}(L(\phi_i)\circ\gamma)\in C^{k+1}([0,1],G_i)$
can be formed. The maps $\eta_i$ are $C^m$ in particular.
From $\delta^\ell(\phi_{i,j}\circ \eta_j)=L(\phi_{i,j})\circ \delta^\ell(\eta_j)=L(\phi_{i,j})\circ L(\phi_j)\circ \gamma
=L(\phi_{i,j}\circ \phi_j)\circ\gamma=L(\phi_i)\circ\gamma=\delta^\ell(\eta_i)$
we deduce that $\phi_{i,j}\circ \eta_j=\eta_i$.
The universal property of the projective limit therefore provides a unique
$C^m$-map $\eta\colon [0,1]\to G$ such that $\phi_i\circ\eta=\eta_i$
for all $i\in I$. For each $i$, we have $L(\phi_i)\circ\delta^\ell(\eta)=\delta^\ell(\phi_i\circ \eta)=\delta^\ell(\eta_i)=L(\phi_i)
\circ \gamma$. Hence $\delta^\ell(\eta)=\gamma$ (as the $L(\phi_i)$ separate points)
and thus $\eta=\Evol_G(\gamma)$. Since $\eta$ is $C^1$ and $\delta^\ell\eta=\gamma$ is $C^k$,
the map $\eta$ is $C^{k+1}$.

(b) Because $G$ is the projective limit of $\cS$ both as a set and as a $C^m$-manifold,
a map $f\colon N\to G$ on a $C^m$-manifold~$N$ is $C^m$ if and only if $\phi_i\circ f$ is $C^m$
for each $i\in I$.
We have just seen that $\phi_i\circ \Evol_G(\gamma)=\Evol_{G_i}(L(\phi_i)\circ\gamma)$,
whence $\phi_i(\evol_G(\gamma))=\evol_{G_i}(L(\phi_i)\circ\gamma)$.
Using the linear map
\[
L(\phi_i)_*\colon C^k([0,1],L(G))\to C^k([0,1],L(G_i)),\quad \gamma\mto L(\phi_i)\circ\gamma
\]
which is continuous (cf.\ \cite{QUO}) and hence smooth, we can
thus write
\[
\phi_i\circ \evol_G=\evol_{G_i}\circ L(\phi_i)_*.
\]
Thus $\phi_i\circ \evol_G$ is smooth and therefore $C^m$,
and hence $\evol_G$ is $C^m$ (by our preparatory consideration).
Theorem~E now shows that $\evol_G$ is smooth,
completing the proof that $G$ is $C^k$-regular.
\end{proof}
Two special cases with more tangible hypotheses are of particular interest.
\begin{cor}\label{easierplim}
Let $G_1\supseteq G_2\subseteq\cdots$ be a descending sequence of Lie groups
such that the inclusion maps $\phi_{i,j}\colon G_j\to G_i$ are smooth homomorphisms for $i\leq j$.
Assume that $G:=\bigcap_{i\in \N}G_i$ is endowed with a Lie group structure
modelled on a locally convex space $E$, such that the inclusion maps $\phi_i\colon G\to G_i$
are smooth homomorphisms.
Let $E_i$ be the modelling locally convex space of $G_i$.
Assume that $E_j$ is a vector subspace of $E_i$ for $i\leq j$ $($possibly with a finer topology$)$
and
\[
E=\bigcap_{i\in \N}E_i=\pl\, E_i\vspace{-.8mm}
\]
as a locally convex space.
Finally, assume that there exist charts $\psi_i\colon U_i\to V_i\sub E_i$ for $G_i$
around~$e$ and a chart $\psi\colon U\to V\sub E$ for $G$ around $e$,
with $\psi(e)=0$,
such that $U=U_i\cap G$ and $\psi=\psi_i|_U$ for all $i\in I$,
and $U_j=U_i\cap G_j$ and
$\psi_j=\psi_i|_{U_j}$ whenever $i\leq j$.
Then
\[
G=\pl\, G_i\vspace{-.8mm}
\]
in the category of $C^m$-manifolds $($with or without boundary$)$,
for each $m\in \N_0\cup\{\infty\}$.
If $k\in \N_0\cup\{\infty\}$ and each $G_i$ is $C^k$-regular $($resp., $C^k$-semiregular$)$,
then also $G$ is $C^k$-regular $($resp., $C^k$-semiregular$)$.
\end{cor}
\begin{proof}
Let $N$ be a $C^m$-manifold (with or without boundary) and $f_i\colon N\to G_i$ be $C^m$-maps
such that $\phi_{i,j}\circ f_j=f_i$ whenever $i\leq j$, i.e., $f_j(x)=f_i(x)$.
Then $f(x):=f_i(x)$ is independent of $i\in \N$ and takes values in $\bigcap_{i\in \N}G_i=G$.
We claim that each $p\in N$ has an open neighbourhood $P\sub N$
such that $f|_P$ is $C^m$. If this is true, then $f$ is $C^m$
and we deduce that $G=\pl\, G_i$\vspace{-.8mm} in the category of $C^m$-manifolds (with or without boundary).
From $\psi\circ \phi_i=\psi_i|_U$ we deduce that $L(\phi_i)$ is injective for each $i$,
and hence the $L(\phi_i)$ separate points on $L(G)$.
Therefore all hypotheses of Proposition~\ref{regpl}
are satisfied and we get the desired conclusions.
It remains to prove the claim. We first consider the special case that $f(p)=e$.
Since $f_1$ is continuous, there exists and open neighbourhood $P\sub N$
of $p$ such that $f_1(P)\sub U_1$.
Then $f_j(P)=f_1(P)\sub G_j\cap U_1= U_j$
for each $j\in \N$. Thus $\psi_j\circ f_j|_P$ is a $C^m$-map to $E_j$,
for each $j\in \N$. Likewise, $\psi\circ f|_P$ is defined.
Composes with the inclusion map $E\to E_j$, the map $\psi\circ f|_P$
yields the $C^m$-map $\psi_j\circ f_j|_P$. If $N$ is a manifold
without boundary, we may assume that $P$ is an open subset of a locally convex space.
Since $E=\pl\, E_j$\vspace{-.8mm} as a locally convex space,
we deduce from \cite[Lemma~10.3]{BGN} that $\psi\circ f|_P$ is $C^m$.
If $N$ has a boundary, an analogous lemma from \cite{GaN}
for $C^m$-maps on non-open subsets applies.
Now consider the general case, that $f(p)$ is arbitrary.
We then define $g_i(x):=f_i(x)f_i(p)^{-1}$
and $g(x):=f(x)f(p)^{-1}$. Since $g(p)=e$,
$g|_P$ is $C^m$ by the special case already treated,
for an open neighbourhood $P\sub N$ of~$p$.
Thus also $f|_P=g|_Pf(p)$~is~$C^m$.
\end{proof}
\begin{cor}\label{bothgptvs}
Let $(I,\leq)$ be a directed set and $((E_i)_{i\in I},(\phi_{i,j})_{i\leq j})$
be a projective system of locally convex spaces $E_i$ and continuous
linear mappings\linebreak
$\phi_{i,j}\colon E_j\to E_i$. Let
\[
E=\pl\, E_i
\]
as a locally convex space, together with the continuous linear maps
$\phi_i\colon E\to E_i$. Assume that $\mu_i\colon E_i\times E_i\to E_i$
is a smooth map which is the group multiplication of a Lie group
structure on~$E_i$, such that $\phi_{i,j}$ is a homomorphism
of groups for all $i\leq j$.
Then there is a group multiplication $\mu$ on $E$ making it a Lie group.
For $k\in \N_0\cup\{\infty\}$, $(E,\mu)$ is
$C^k$-semiregular if $(E_i,\mu_i)$ is $C^k$-semiregular for each $i\in I$.
Likewise, $(E,\mu)$ is $C^k$-regular if $(E_i,\mu_i)$
is $C^k$-regular for each $i\in I$.
\end{cor}
\begin{proof}
A map $f$ from a $C^m$-manifold $N$ to the projective limit locally convex space $E$
is $C^m$ if and only if $\phi_i\circ f$ is $C^m$ for each
$i\in I$ (we may assume that $N$ is an open subset
of a locally convex space, and apply \cite[Lemma~10.3]{BGN}).
The same holds if $N$ is a manifold with boundary (cf.\ \cite{GaN}).
Hence $E=\pl\, E_i$\vspace{-.8mm} in the category of $C^m$-manifolds (with or without boundary).
Since the $\phi_i$ separate points on the projective limit,
all hypotheses of Proposition~\ref{regpl} are satisfied. 
\end{proof}
\begin{example}\label{Wage}
Let $k\in \N_0$ and $H$ be a $C^k$-regular Lie group.
Then $H$ is regular and hence there is a regular Lie group structure
on $C^\infty(\R,H)$ (as described in \cite{NaW}).
It is know from \cite[Corollary~144]{Alz} that $C^\infty(\R,H)$
is $C^k$-regular. This result was obtained in \emph{op.\,cit.}
as part of a long-winded, general theory. Using Corollary~\ref{bothgptvs},
we now give a very short, alternative proof:
\emph{$C^\infty(\R,H)$ is $C^k$-regular.}\\[2.8mm]
By \cite{NaW},
\[
C^\infty(\R,H)\cong C^\infty_*(\R,H)\rtimes H,
\]
where $C^\infty_*(\R,H):=\{\gamma\in C^\infty(\R,H)\colon \gamma(0)=e\}$
is a regular Lie group isomorphic to $(C^\infty(\R,\ch),*)=:G$ via the isomorphism
\[
\delta^\ell\colon C^\infty_*(\R,H)\to C^\infty(\R,\ch),\quad\gamma\mto\delta^\ell(\gamma).
\]
If we can show that $G$ is $C^k$-regular, then also $C^\infty(\R,H)$ will be $C^k$-regular,
being an extension of $C^k$-regular Lie groups (see \cite{NaS}).
Now $C^\infty_*([{-n},n],H)$ is $C^k$-regular (see Example~\ref{mapstar})
and
\[
\delta^\ell\colon C^\infty_*([{-n},n],H)\to (C^\infty([{-n},n],\ch),*)=:G_n,\quad \gamma\mto \delta^\ell(\gamma)
\]
is an isomorphism of Lie groups (with inverse $\Evol$), by the regularity of $H$.
Hence $H_n$ is $C^k$-regular for each $n\in\N$.
%
%
Thus $G=\pl\,G_n$\vspace{-.8mm} both as a locally convex space and
as a group. Hence also $G$ is $C^k$-regular (as required), by Corollary~\ref{bothgptvs}.
\end{example}
\begin{rem}
Projective limit arguments have a long history in infinite-dimensional
Lie theory (cf.\ \cite{OMO}, \cite{OMBOOK}).
See \cite{Wal}
for another example of their use
in connection with regularity.
\end{rem}
\begin{numba}
Let $E$ be a real locally convex space.
A complex analytic local Lie group $(U^*,D^*,\mu^*,\eta^*,e)$ modelled on $E_\C$
is called a \emph{complexification} of a real analytic local Lie group $(U,D,\mu,\eta,e)$
if $U$ is a real analytic submanifold of $U^*$,
the inclusion map $U\to U^*$ is a homomorphism of local groups
and for each $x\in U$, there exists an open neighbourhood $V\sub U^*$ and a complex analytic diffeomorphism
$\phi\colon V\to W\sub E_\C$ such that $\phi(V\cap U)=W\cap E$.
\end{numba}
\begin{numba}
\emph{Every real analytic local Lie group $(U,D,\mu,\eta,e)$
has an open local Lie subgroup which admits a complexification.}\\[2.5mm]
In fact, there is an open identity neighbourhood $P=P^{-1}$ in $U$
on which a chart $\phi\colon P\to Q$ with $\phi(e)=0$ is defined,
with $Q$ an open $0$-neighbourhood in the modelling space~$E$ of~$U$.
Then $P$ is an open local Lie subgroup
with domain $D_P=(P\times P)\cap\mu^{-1}(P)$
for the multiplication. We may replace $U$ with $P$.
Since we can make $Q$ a local Lie group isomorphic to $P$,
eventually we may assume that $U=Q$ is an open $0$-neighbourhood in~$E$.
Now the real analytic maps $\mu$ and $\eta$ extend to complex analytic maps
$\tilde{\mu}\colon \tilde{D}\to E_\C$
and $\tilde{\eta}\colon \tilde{U}\to E_\C$, respectively,
with $\tilde{D}$ an open neighbourhood of $D$ in $E_\C\times E_\C$ and
$\tilde{U}$ an open neighbourhood of~$U$ in~$E_\C$.
Let $A\sub E_\C$ be an open $0$-neighbourhood such that $A\times A\sub \tilde{D}$.
Let $B\sub A$ be an open connected $0$-neighbourhood such that $B\times B\sub \tilde{\mu}^{-1}(A)$
and $B\sub \tilde{U}\cap\tilde{\eta}^{-1}(\tilde{U})$.
Then $\tilde{\eta}(\tilde{\eta}(x))$ is defined for all $x\in B$.
Since $B$ is connected and meets $U$, and $\eta\circ\eta=\id_U$,
the identity theorem shows that
\[
\tilde{\eta}(\tilde{\eta}(x))=x\quad
\mbox{for all $x\in B$.}
\]
Hence $\tilde{\eta}|_B$ is injective.
Finally, let $W$ be the connected component of $0$ of the open set $B\cap \tilde{\eta}^{-1}(B)=\{x\in B\colon
\tilde{\eta}(x)\in B\}$.
For $x\in B$, we have $\tilde{\eta}(x)\in B$ and $\tilde{\eta}(\tilde{\eta}(x))=x\in B$,
i.e., $\tilde{\eta}(x)\in B\cap\tilde{\eta}^{-1}(B)$.
Thus $\tilde{\eta}(W)\sub B\cap\tilde{\eta}^{-1}(B)$
and hence $\tilde{\eta}(W)\sub W$, by connectedness.
From $W=\tilde{\eta}(\tilde{\eta}(W))\sub \tilde{\eta}(W)\sub W$, we deduce
that $\tilde{\eta}(W)=W$. Thus $\tilde{\eta}|_W\colon W\to W$ is an involutive
complex analytic diffeomorphism. Write $xy:=\tilde{\mu}(x,y)$.
Since all triple products of elements of $W$ are defined and $W$ is connected,
using the identity theorem we see that all axioms for a complex analytic
local Lie group $(W,D_W,\tilde{\mu}|_{D_W},\tilde{\eta}|_W,0)$
are satisfied if we set $D_W:=(W\times W)\cap\tilde{\eta}^{-1}(W)$.
Let $V\sub U$ be an open $0$-neighbourhood such that $V=\eta(V)$,
$V\times V\sub D$ and $V\times V\sub D_W$.
Then $V$ is an open local Lie subgroup of $U$ if we take
$D_V:=(V\times V)\cap\mu^{-1}(V)$ as domain of definition
for the multiplication $\mu|_{D_V}$,
and $W$ is a complexification of $V$.
%
%
\end{numba}
\begin{la}
Let $U$ be a complex analytic local Lie group
which is $C^k$-regular for some $k\in \N_0\cup\{\infty\}$.
Then $\evol\colon C^k([0,1],L(U))\supseteq \Omega\to U$
is complex analytic on some open $0$-neighbourhood in $C^k([0,1],L(U))$.
\end{la}
\begin{proof}
Choose $\Omega\sub C^k([0,1],L(U))$ such that
$\evol\colon \Omega\to U$ is $C^\infty$.
In particular, $\Evol(\eta)\in C^{k+1}([0,1],U)$
exists for each $\eta\in\Omega$,
enabling us to define
\[
C^k([0,1],L(U))\times\Omega\to C^k([0,1],L(U)),\quad (\gamma,\eta)\mto \gamma*\eta
\]
%
%
%
with $\gamma*\eta:=\Ad(\Evol(\eta))^{-1}\gamma-\eta$.
There is an open zero-neighbourhood $\Omega_1\sub\Omega$
such that $\evol(\gamma*\eta)=\evol(\gamma)\evol(\eta)$
for all $\gamma,\eta\in\Omega_1$. For $\eta\in\Omega_1$, the map
$\rho_\eta:=(.)*\eta$ is an invertible is continuous complex affine-linear
with continuous inverse, and hence is a complex analytic
diffeomorphism. Also the right translation $r_{\evol(\eta)}\colon V\to U$
is complex analytic on some identity neighbourhood $V\sub U$. By the preceding,
$
T_\eta\evol \circ T_0\rho_\eta=T_0r_{\evol(\eta)}\circ T_0\evol
$
and thus
\begin{equation}\label{compolin}
T_\eta\evol=T_0r_{\evol(\eta)}\circ T_0\evol\circ (T_0\rho_\eta)^{-1}.
\end{equation}
Here $T_0\evol$ corresponds to the integration map
and therefore is complex linear.
As also the remaining maps on the right-hand side of
(\ref{compolin}) are complex linear, we deduce that so is $T_\eta\evol$.
Being smooth with complex linear tangent maps, $\evol$ is complex
analytic (cf.\ \cite{RES}).
\end{proof}
\begin{prop}\label{regcx}
Let $U$ be a real analytic local Lie group and
$U^*$ be a complex analytic local Lie group which is a complexification for $U$.
Let $k\in \N_0\cup\{\infty\}$.
Then the following conditions are equivalent:
\begin{itemize}
\item[\rm(a)]
$U^*$ is $C^k$-regular.
\item[\rm(b)]
$U$ is $C^k$-regular and the evolution
$\evol\colon \Omega\to U$ $($in~{\rm\ref{deflocreg}}$)$
can be chosen as a real analytic map.
\end{itemize}
\end{prop}
\begin{proof}
(a)$\impl$(b):
After shrinking $U^*$ and $U$, we may assume that $U$ is a real analytic submanifold
of $U^*$ and $U=\{z\in U^*\colon \sigma(z)=z\}$
for an antiholomorphic automorphism $\sigma\colon U^*\to U^*$.
By Proposition~\ref{thmGloc},
$U$ is $C^k$-regular,
with an evolution map which is a restriction
of the complex analytic evolution map for $U^*$
and hence real analytic.

(b)$\impl$(a):
After shrinking $U$ and $U^*$,
we may assume (up to isomorphism)
that $U$ is an open subset of some real locally convex space $E$
and $U^*$ an open subset of $E_\C$. Identify $L(U)=T_0(U)$ with $E$,
and likewise for $U^*$.
Since $\evol$ is real analytic, it admits a complex analytic
extension $h\colon \Omega^*\to E_\C$
to an open superset $\Omega^*$ of $\Omega$ in $C^k([0,1],E_\C)$.
After shrinking $\Omega^*$, we may assume that $h(\Omega^*)\sub U^*$.
After replacing $\Omega^*$ by the union of such, we may assume that each component of
$\Omega^*$ meets $\Omega$.
The map $\delta^\ell\colon C^{k+1}([0,1],U^*)\to C^k([0,1],E_\C)$
is complex analytic, and $\delta^\ell(h(\gamma))=\gamma$ for all $\gamma\in \Omega$.
Hence
$\delta^\ell\circ h=\id_{\Omega^*}$, by the Identity Theorem.
\end{proof}
\begin{cor}
Let $G$ be a real analytic Lie group
and $U^*$ be a complexification of an open local Lie subgroup
$U\sub G$. Let $k\in \N_0\cup\{\infty\}$.
Then the following conditions are equivalent:
\begin{itemize}
\item[\rm(a)]
$U^*$ is $C^k$-regular.
\item[\rm(b)]
$G$ is $C^k$-regular and $\evol\colon C^k([0,1],\cg)\to G$
is real analytic.
\end{itemize}
\end{cor}
\begin{proof}
Combine Proposition \ref{regcx} and the following lemma.
\end{proof}
\begin{la}
Let $k\in \N_0\cup\{\infty\}$ and $G$ be a $C^k$-regular real analytic Lie group.
If $\evol\colon C^k([0,1],\cg)\to G$ is real analytic on some open $0$-neighbourhood,
then $\evol$ is real analytic.
\end{la}
\begin{proof}
By hypothesis, there is $\ell\in \N_0$ such that $\ell\leq k$
and a continuous seminorm $q$ on $\cg$ such that
$\evol$ is real analytic on
\[
Q_0:=\{\gamma\in C^k([0,1],\cg)\colon \|\gamma\|_{C^\ell,q}<1\}.
\]
We show that $\evol$ is real analytic on
$Q_n:=\{\gamma\in C^k([0,1],\cg)\colon \|\gamma\|_{C^\ell,q}<2^n\}$
for each $n\in\N_0$; since $C^k([0,1],\cg)=\bigcup_{n\in\N_0}Q_n$,
the desired real analyticity of $\evol$ then follows.
We proceed by induction on~$n$;
the case $n=0$ is valid by choice of $q$ and $\ell$.
To perform the induction step, note that the maps
$C^k([0,1],\cg)\to C^k([0,1],\cg)$ taking
$\gamma$ to $\gamma_0$ and $\gamma_1$, respectively, are
continuous linear and hence real analytic,
where
\[
\gamma_j(t):=\frac{1}{2}\gamma\left(\frac{j+t}{2}\right)\quad\mbox{for all $t\in[0,1]$
and $j\in\{0,1\}$.}
\]
Moreover, $\|\gamma_0\|_{C^\ell,q},\|\gamma_1\|_{C^\ell,q}\leq\frac{1}{2}\|\gamma\|_{C^\ell,q}$.
If $n\geq 1$ and $\gamma\in Q_n$,
then $\gamma_0,\gamma_1\in Q_{n-1}$
by the preceding estimate.
Note that
\[\Evol(\gamma)(t)=\left\{
\begin{array}{cl}
\Evol(\gamma_0)(2t)&\mbox{if $t\in[0,\frac{1}{2}]$,}\\
\evol(\gamma_0)\Evol(\gamma_1)(2t-1)&\mbox{if $t\in[\frac{1}{2},1]$}
\end{array}\right.
\]
and thus $\evol(\gamma)=\evol(\gamma_0)\evol(\gamma_1)$,
which is real analytic in $\gamma$.
\end{proof}
\section{Regularity if there is a candidate for {\boldmath$\evol$}}
Frequently, we already have a good guess for a mapping $f\colon C^k([0,1],\cg)\to G$
which should be the evolution map $\evol$.
In this section, we provide simple criteria ensuring that indeed $f=\evol$.
\begin{la}\label{lala1}
Let $G$ be a Lie group, $\gamma\in C^0([0,1],\cg)$
and $\eta\in C^1([0,1],G)$.
If there is a family $(\alpha_j)_{j\in J}$
of smooth homomorphisms $\alpha_j\colon G\to G_j$ to
Lie groups $G_j$ such that
\[
(\forall j\in J)\quad
\delta^\ell(\alpha_j\circ \eta)=L(\alpha_j)\circ \gamma
\]
and the maps $L(\alpha_j)$ separate points on $L(G)$ for $j\in J$,
then $\delta^\ell(\eta)=\gamma$.
\end{la}
\begin{proof}
Since also $\delta^\ell(\alpha_j\circ \eta)=L(\alpha_j)\circ \delta^\ell(\eta)$
and the $L(\alpha_j)$ separate ponts,
we obtain $\delta^\ell\eta=\gamma$.
\end{proof}
\begin{la}\label{lala2}
Let $G$ be a Lie group, $k\in \N_0\cup\{\infty\}$ and $f\colon C^k([0,1],\cg)\to G$
be a $C^1$-map such that
\[
[0,1]\to G,\quad s\mto f(s\gamma(s.))
\]
is $C^1$ for each $\gamma\in C^k([0,1],\cg)$.
If there exists a family $(\alpha_j)_{j\in J}$ of smooth
homomorphisms $\alpha_j\colon G\to G_j$ to Lie groups $G_j$
such that
\begin{equation}\label{taeusch}
\alpha_j\circ f=\evol_{G_j}\circ L(\alpha_j)
\end{equation}
and the maps $L(\alpha_j)$ separate points on $\cg$,
then $G$ is $C^k$-regular and $f=\evol$.
\end{la}
The meaning of (\ref{taeusch}) is clear if we assume that each $G_j$ is $C^k$-semiregular.
In the general case, let $\Omega_j$ be the set of all $\gamma\in C^k([0,1],\cg_j)$
such that $\Evol(L(\alpha_j)\circ\gamma)$
exists, and define $\evol_{G_j}\colon \Omega_j\to G_j$,
$\gamma\mapsto \Evol(\gamma)(1)$.
We require that $L(\alpha_j)\circ\gamma\in \Omega_j$ for all
$\gamma\in C^k([0,1],\cg)$ and (\ref{taeusch}) holds.\\[2mm]
\begin{proof}
For each $\gamma\in C^k([0,1],\cg)$ and $j\in J$, we have
\begin{eqnarray*}
\alpha_j(f(s\gamma(s.))&=&\evol_{G_j}(L(\alpha_j)(s\gamma(s,.))=\evol_{G_j}(s(L(\alpha_j\circ\gamma)(s.))\\
&=&
\Evol(L(\alpha_j)\circ\gamma)(s)
\end{eqnarray*}
for each $s\in [0,1]$.
Hence, by Lemma~\ref{lala1}, $(s\mto f(s\gamma(s.)))=\Evol(\gamma)$.
Thus $G$ is $C^k$-semiregular with $\Evol=f$ a $C^1$-map, and thus $G$
is $C^k$-regular (by Theorem~E).
\end{proof}
\begin{la}\label{lala3}
Let $G$ be a Lie group and $f\colon C^0([0,1],\cg)\to G$
be a continuous map such that
\[
[0,1]\to G,\quad s\mto f(s\gamma(s.))
\]
is $C^1$ for each $\gamma\in C([0,1],\cg)$.
If $\cg$ is integral complete and there exists a family $(\alpha_j)_{j\in J}$ of smooth
homomorphisms $\alpha_j\colon G\to G_j$ to $C^0$-regular Lie groups $G_j$
such that
\begin{equation}\label{taeusch2}
\alpha_j\circ f=\evol_{G_j}\circ L(\alpha_j)
\end{equation}
and the maps $L(\alpha_j)$ separate points on $\cg$,
then $G$ is $C^0$-regular and $f=\evol$.
\end{la}
\begin{proof}
As in the preceding proof, we see that $G$ is $C^0$-regular and $f=\evol$.
Now Theorem F applies.
\end{proof}
\section{Regularity of weak direct products}\label{sum}
Let $(G_j)_{j\in J}$
be a family of Lie grpoups $G_j$, with modelling space $E_j$.
Let
\[
E:=\bigoplus_{j\in J}E_j:=\{(x_j)_{j \in J}\in\prod_{j\in J}E_j\colon\mbox{$x_j=0$
for all but finitely many $j$}\}
\]
be the direct sum of the given locally convex spaces,
endowed with the locally convex direct sum topology.
Then
\[
G:=\bigoplus_{j\in J}G_j:=\{(x_j)_{j \in J}\in\prod_{j\in J}G_j\colon\mbox{$x_j=e$
for all but finitely many $j$}\}
\]
is a group under pointwise multiplication.
If $\phi_j\colon U_j\to V_j$ is a chart for $G_j$ defined on
an open identity neighbourhood $U_j=U_j^{-1}$ with $\phi(e)=0$,
then $G$ can be given a Lie group structure modelled on~$E$
such that
\[
\phi:=\oplus_{j\in J}\phi_j\colon \bigoplus_{j\in J}U_j\to \bigoplus_{j\in J}V_j,\quad (x_j)_{j\in J}\mto
(\phi_j(x_j))_{j\in J}
\]
is a chart around the identity element (cf.\ \cite{MEA}).\\[2.3mm]
Let $\pi_i\colon E\to E_i$, $(x_j)_{j\in J}\mto x_i$
be the projection onto the $i$-th component, for $i\in J$.
We shall use the following fact (see \cite{DIF} or \cite{Alx}):
\begin{numba}\label{numbacompmap}
The map
\begin{equation}\label{compmap}
\Phi\colon C^k([0,1],E)\to\bigoplus_{j\in J}C^k([0,1],E_j),\quad f\mto (\pi_j\circ f)_{j\in J}
\end{equation}
sending $f$ to its family of components $f_j:=\pi_j\circ f$ is an isomorphism
of vector spaces.
If $J$ is countable and $k\in \N_0$ is finite, then $\Phi$ is an isomorphism
of topological vector spaces.
\end{numba}
\begin{prop}\label{weakreg}
Let $k\in \N_0\cup\{\infty\}$ and $G:=\bigoplus_{j\in J}G_j$ be
the weak direct product of a family $(G_j)_{j\in J}$ of Lie groups,
with Lie algebra $\cg:=L(G)$. 
\begin{itemize}
\item[\rm(a)]
If $G_j$ is $C^k$-semiregular for each $j\in J$,
then also the weak direct product $G:=\bigoplus_{j\in J}G_j$
is $C^k$-semiregular.
\item[\rm(b)]
If $J$ is countable, $k<\infty$ and $G_j$ is $C^k$-regular for each $j\in J$,
then also $\bigoplus_{j\in J} G_j$
is $C^k$-regular.
\item[\rm(c)]
If $G_j$ is $C^k$-regular for each $j\in J$
and $\evol\colon C^k([0,1],\cg)\to G$ is continuous,
then $G$ is $C^k$-regular.
\end{itemize}
\end{prop}
\begin{proof}
Let $\pi_j\colon G\to G_j$ be the projection onto the $j$th component.
Let $I$ be the set of all finite subsets $F\sub J$
and set $G_F:=\prod_{j\in F}G_j$ for $F\in J$,
which is a Lie subgroup of~$G$.
Let $\iota_F\colon G_F\to G$ be the inclusion map.
Then $G=\bigcup_{F\in I}G_F$
is a directed union. We can identify $L(G)$ with $\bigoplus_{j\in J}L(G_j)$
by means of the Lie algebra homomorphism $(L(\pi_j))_{j\in J}$,
which is an isomorphism of topological vector spaces
by the above discussion.
If $\gamma\in C^k([0,1],\cg)$, then $\gamma([0,1])\sub G_F$
for some $F\in I$, the co-restriction
$\gamma|^{L(G_F)}\colon [0,1]\to L(G_F)=\prod_{j\in F} L(G_j)$
is $C^k$
and $\gamma=\iota_F\circ \gamma|^{L(G_F)}$
(cf.\ \ref{compmap})).
Thus (a) and (c) follow from Corollary~\ref{moredl}.

(b) Step 1. We may assume that $J=\N$. Let $E_n$ be the modelling space of $G_n$ and
$\phi=\oplus_{n\in \N}\phi_n$ be the chart from above (with $J:=\N$).
Identifying $\cg$ with $E=\bigoplus_{n\in\N}E_n$ by means of the isomorphism $d\phi|_\cg$
and $\cg_n$ with $E_n$ by means of $d\phi_n|_{\cg_n}$,
the map $L(\pi_n)$ corresponds to the projection
$\pr_n\colon E\to E_n$ onto the $n$-th component.
If $\gamma=(\gamma_n)_{n\in\N}\in C^k([0,1],E)$,
then there exists a finite set $F\sub \N$ such that $\gamma_n=0$ for all $n\in \N\setminus F$.
Then
\[
\eta:=(\Evol(\gamma_n))_{n\in F}\colon [0,1]\to\prod_{n\in F}G_n=:P
\]
is the evolution of $\gamma$, considered as a $C^k$-map to $L(P)$.
Since $P$ is a Lie subgroup of $G$, we see that $\Evol(\gamma)=\eta$.

Step 2. As each of the maps $\evol_{G_n}\colon C^k([0,1],E_n)\to G_n$
is smooth, also the map
\[
\oplus_{n\in\N}\evol_{G_n}\colon \bigoplus_{n\in\N} C^k([0,1],E_n)\to\bigoplus_{n\in\N}G_n=G,\quad
(\gamma_n)_{n\in\N}\mto(\evol_{G_n}(\gamma_n))_{n\in\N}
\]
is smooth~\cite{MEA}.
%
%
Hence also the map $f:=\big(\oplus_{n\in \N}\evol_{G_n}\big)\circ\Phi\colon
C^k([0,1],E)\to G$ is smooth, and its satisfies
$\pi_n\circ f=\evol_{G_n}\circ\pr_n$ by contruction (where the $\pr_n$
separate points on $E$). Finally, for $\gamma$ and $\eta$ as in Step~1,
\[
f(s\gamma(s.))=(\evol_{G_n}(s\gamma_n(s.)))_{n\in\N}=(\Evol_{G_n}(\gamma_n)(s))_{n\in\N}=\eta(s)
\]
is a $C^1$-map in $s$. Thus all hypotheses of Lemma~\ref{lala2}
are satisfied and thus $G$ is $C^k$-regular with $f=\evol_G$.
\end{proof}
\begin{rem}
Assume that $E_j\not=\{0\}$
for all $j\in J$. If $J$ is uncountable or $k=\infty$
then the map (\ref{compmap}) is not continuous
(see, e.g., \cite{Alx} and \cite{DIF}, respectively).
Therefore the method of the preceding proof
breaks down.
\begin{itemize}
\item[(a)]
The author therefore does not know
if
$\bigoplus_{j\in J}G_j$ is regular
if the $G_j$ are merely regular
(but not $C^k$-regular for a fixed $k\in \N_0$).
\item[(b)]
The author also does not know whether $\bigoplus_{j\in J}G_j$
is $C^k$-regular if $J$ is uncountable and
each $G_j$ a $C^k$-regular Lie group
(where $k\in \N_0$).
A special case where this is true will be discussed in Section~\ref{secergae}.
It is also known that $\bigoplus_{j\in J}G_j$ will be
$C^{k+1}$-regular in the preceding situation
(see Corollary~\ref{alsodoch}).
\end{itemize}
\end{rem}
\section{Regularity of mapping groups and\\
gauge groups}\label{gauge}
In this section, we establish $C^k$-regularity for some important
classes of Lie groups in larger generality than recorded previously.
The proofs illustrate the typical use of many of the tools
and techniques from the previous sections.\\[2.3mm]
The following result was announced in~\cite{OWR};
the arguments were then also spelled out in~\cite{Alz}.
\begin{prop}\label{cpcureg}
Let $M$ be a compact smooth manifold $($with or without smooth boundary$)$, $k,\ell\in \N_0\cup\{\infty\}$
and $H$ be a $C^k$-regular Lie group.
Then also $G:=C^\ell(M,H)$ is $C^k$-regular.
\end{prop}
\begin{proof}
Let $E$ be the modelling space of $H$ and $\phi\colon U=U^{-1}\to V\sub E$
be a chart around $e$ such that $\phi(e)=0$.
Then $\phi_*\colon C^\ell(M,U)\to C^\ell(M,V)$, $\gamma\mto\phi\circ\gamma$ is a chart
for~$G$. For $x\in M$, let $\ve_x\colon C^\ell(M,H)\to H$
and $e_x\colon C^\ell(M,E)\to E$ be the point evaluation at $x$ sending
$\gamma$ to $\gamma(x)$.
Identifying $\ch$ with $E$ by means of $d\phi|_\ch$ and
$C^\ell(M,E)$ with $\cg$ by means of $d(\phi_*)|_\cg$,
we have $L(\ve_x)=e_x$.
If $\gamma\in C^k([0,1], C^\ell(M,\ch)$,
then $\Phi(\gamma):=\gamma^\wedge\colon [0,1]\times M\to \ch$
is a $C^{k,\ell}$-map and the map
\[
\Phi\colon C^k([0,1], C^\ell(M,E))\to C^{k,\ell}([0,1]\times M,E)
\]
so obtained an isomorphism of locally convex spaces (see \cite{Alz} or \cite{AS}).
Also the mappings
\[
\Psi\colon C^{k,\ell}([0,1],M,E)\to C^{\ell,k}(M\times [0,1],E),\quad \Psi(\gamma)(x,s):=\gamma(s,x)
\]
and
\[
\Theta\colon C^{\ell,k}(M\times [0,1],E)\to C^\ell(M,C^k([0,1],E)),\quad \gamma\mto \gamma^\vee
\]
are isomorphisms of locally convex spaces (see \emph{loc.\,cit.}),
and hence so is $\Theta\circ \Psi\circ\Phi\colon
C^k([0,1],C^\ell(M,E))\to C^\ell(M,C^k([0,1],E))$.
Similarly, the map
\[
\Gamma\colon C^{k+1}(M,C^\ell([0,1],H)\to C^\ell([0,1],C^{k+1}(M,H)), \;
\Gamma(\gamma)(s,x)=\gamma(x)(s)
\]
is an isomorphism of Lie groups~\cite{Alz}.
Finally, because $\Evol_H\colon C^k([0,1],\ch)\to C^{k+1}([0,1],H)$ is smooth, also the map
\[
(\Evol_H)_*\colon C^\ell(M,C^k([0,1],E))\to C^\ell(M,C^{k+1}([0,1],H))
\]
is smooth (cf.\ \cite{QUO}). Hence $h:=\Gamma\circ (\Evol_H)_*\circ \Theta\circ\Psi\circ\Phi$
is a smooth map $C^k([0,1],C^\ell(M,E))\to C^{k+1}([0,1],C^\ell(M,H))$.
As the smooth homomorphisms $\ve_x$ separate points on $C^\ell(M,H)$ and
\[
C^{k+1}([0,1],\ve_x)\circ h=\Evol_H\circ C^k([0,1],e_x)
\]
for all $x$,
we deduce that $h=\Evol_G$ for $G:=C^\ell(M,H)$.
\end{proof}
\begin{numba}\label{preparetestfu}
If $M$ is a paracompact, finite-dimensional smooth manifold,
$H$ a Lie group and $\ell\in\N_0\cup\{\infty\}$,
then the group $C^\ell_c(M,H)$ of compactly supported $H$-valued $C^\ell$-maps
on~$M$ is a Lie group which can be identified with the Lie subgroup
\[
S:=\{(\gamma_j)_{j\in J}\in\prod_{j\in J}C^\ell(M_j,H)\colon (\forall i,j\in J)(\forall x\in M_i\cap M_j)\, \gamma_i(x)=\gamma_j(x)\}
\]
of the weak direct product $\bigoplus_{j\in J}C^\ell(M_j,H)$,
for $(M_j)_{j\in J}$ a locally finite family of compact submanifolds
with boundary of $M$ whose interiors cover~$M$.
Hence, by Theorem~G, the Lie group $C^\ell_c(M,H)$ will be $C^k$-regular
whenever we can show that $\bigoplus_{j\in J}C^\ell(M_j,H)$ is $C^k$-regular.
\end{numba}
\begin{prop}\label{testfreg}
If $M$ is a paracompact, finite-dimensional smooth
manifold and $H$ a Lie group which is $C^k$-regular for some
$k\in\N_0$, then $C^\ell_c(M,H)$ is $C^{k+1}$-regular
for each $\ell\in \N_0\cup\{\infty\}$.
If $M$ is, moreover, $\sigma$-compact, then $C^\ell_c(M,H)$ is $C^k$-regular.
\end{prop}
\begin{proof}
Assume that $M$ is $\sigma$-compact.
By Propositions~\ref{weakreg}\,(b) and \ref{cpcureg},
the weak direct product $\bigoplus_{j\in J}C^\ell(M_j,H)$
is $C^k$-regular and hence also $C^\ell_c(M,H)$, by \ref{preparetestfu}.
If $M$ is merely paracompact, we can argue in the same
way, using Corollary~\ref{alsodoch} to get $C^{k+1}$-regularity
of the weak direct product.
\end{proof}
\begin{prop}\label{gaugereg}
Let $P\to M$ be a smooth principal bundle
over a $\sigma$-compact finite-dimensional smooth
manifold $M$, whose structure group is a Lie group $G$
modelled on a locally convex space.
Assume that the condition $SUB_\oplus$
from {\rm\cite{Sch}} is satisfied
and $H$ is $C^k$-regular
for some $k\in \N_0\cup\{\infty\}$.
Let
$\Aut_c(P)$ be the Lie group of all compactly supported symmetries of $P$
{\rm(as constructed in \cite{Sch}).}
If $M$ is compact or $k<\infty$, then
$\Gau_c(P)$ and $\Aut_c(P)$ are $C^k$-regular Lie group.
\end{prop}
\begin{proof}
By \cite[Lemma~4.12]{Sch}, the gauge group $\Gau_c(P)$
is isomorphic to a Lie subgroup $S$ of the weak direct product
$\bigoplus_{n\in \N}C^\infty(\wb{V_i},G)$
for a suitable locally finite cover of $M$ by compact submanifolds
$\wb{V_i}$ with corners. The subgroup is the equalizer
of smooth homomorphisms:
It consists of all $(\gamma_n)_{n\in\N}$ in the weak direct product
such that $k_{ij}(x)\gamma_i(x)k_{ji}(x)=\gamma_j(x)$
for all
$i,j\in\N$ and $x\in \wb{V_i}\cap\wb{V_j}$
(with suitable elements $k_{i,j}\in G$).
Hence $S$ (and thus also $\Gau_c(P)$)
is $C^k$-regular, by Proposition~\ref{weakreg}
and Theorem~G.
Since $\Aut_c(P)$ is an extension
of an open subgroup of $\Diff_c(M)$ (which is $C^0$-regular,
see \cite{Alx}, cf.\ \cite{DIF})
by the $C^k$-regular Lie group $\Gau_c(P)$,
it follows that $\Aut_c(P)$ is $C^k$-regular
(as $C^k$-regularity is an extension property~\cite{NaS}).
\end{proof}
\section{Regularity properties of uncountable weak direct
products}
It is useful to strengthen the concept of $C^\ell$-regularity by weakening
the topology on the space of curves in the Lie algebra.
\begin{defn}\label{def2parreg}
Let $\ell\in \N_0\cup\{\infty\}$
and $\cO$ be a functional class which assigns to each
locally convex space $E$ a locally convex vector topology
$\cO(E)$ on $C^\ell([0,1],E)$ which is coarser
than the compact-open $C^\ell$-topology.
We say that a Lie group $G$ is $(C^\ell,\cO)$-\emph{regular}
if $G$ is $C^\ell$-semiregular and
\[
\evol_G\colon \big(C^\ell([0,1],\cg),\cO(\cg)\big)\to G
\]
is smooth, where $\cg:=L(G)$.
To make the notation less heavy, we frequently write $\cO$ instead of $\cO(E)$.
\end{defn}
\begin{example}
Let $k,\ell\in \N_0\cup\{\infty\}$ such that $k\leq \ell$
and $G$ be a Lie group with Lie algebra~$\cg$.
Then $G$ is $(C^\ell,\cO_{C^k})$-regular
if and only if $G$ is $C^\ell$-semiregular and
$\evol_G\colon \big(C^\ell([0,1],\cg),\cO_{C^k}\big)\to G$
is smooth.
\end{example}
Further instances of Definition~\ref{def2parreg} are useful.
The current section hinges on a notion of
$(C^\ell,\cO_{C^k_+})$-regular Lie groups,
and Section~\ref{secergae}
will involve a notion of $(C^0,\cO_{L^1})$-regularity.
These involve topologies defined as follows:
\begin{defn}
If $E$ is a locally convex space
and $q$ a continuous seminorm on $E$,
we define
\[
\|\gamma\|_{L^1,q}:=\int_0^1q(\gamma(t))\, dt
\]
for $\gamma\in C([0,1],\cg)$.
Then $\|.\|_{L^1,q}$
is a seminorm on $C([0,1],\cg)$
which is continuous with respect to the $C^0$-topology
(as $\|\gamma\|_{L^1,q}\leq \sup\{q(\gamma(t))\colon t\in [0,1]\}$).
We let $\cO_{L^1}$ be the locally convex vector topology
on $C([0,1],E)$ defined by the seminorms $\|.\|_{L^1,q}$,
for $q$ ranging through the set of all continuous seminorms on~$E$.
If $\ell\in \N_0\cup\{\infty\}$ and $k\in \N_0$ such that $k<n$,
we write $\cO_{C^k_+}$ for the topology on $C^\ell([0,1],E)$
which is initial with respect to the
inclusion map
$C^\ell([0,1],E)\to (C^k([0,1],E),\cO_{C^k})$
and the mapping
\[
C^\ell([0,1],E)\to (C([0,1],E),\cO_{L^1}),\quad\gamma\mto\gamma^{(k+1)}.
\]
Thus $\cO_{C^k}\sub \cO_{C^k_+}\sub \cO_{C^\ell}$ on $C^\ell([0,1],E)$.
\end{defn}
\begin{rem}
Let $G$ be a Lie group.
\begin{itemize}
\item[(a)]
Let $\ell$ and $\cO$ be as in Definition~\ref{def2parreg}.
If $G$ is $(C^\ell,\cO)$-regular, then $G$
is $C^\ell$-regular.
\item[(b)]
If $G$ is $C^k$-regular with $k\in \N_0$, then $G$ is $(C^\ell,\cO_{C^k_+})$-regular
for each $\ell\in \N\cup\{\infty\}$ such that $\ell>k$.
\item[(c)]
The notion of $C^k$-regularity used in this article
was called "strong $C^k$-regularity'' in some
of the author's older works,
while the old notion of "$C^k$-regularity''
is $(C^\infty,\cO_{C^k})$-regularity in the current
terminology.
\end{itemize}
\end{rem}
We prove the following result:
\begin{thm}\label{great}
Let $\ell \in \N\cup\{\infty\}$, $k \in \N_0$ such that $k<\ell$
and $(G_j)_{j\in J}$ be a family of $(C^\ell,\cO_{C^k_+})$-regular Lie groups,
with Lie algebra $\cg_j$.
Then also $G:=\bigoplus_{j\in J}G_j$
is a $(C^\ell,\cO_{C^k_+})$-regular Lie group
and hence $G$ is $C^\ell$-regular.
\end{thm}
Taking $\ell:=k+1$, we deduce:
\begin{cor}\label{alsodoch}
Let $k\in \N_0$ and $(G_j)_{j\in J}$ be a family of $C^k$-regular Lie groups.
Then $\bigoplus_{j\in J}G_j$ is $C^{k+1}$-regular.\,\Punkt
\end{cor}
If $M$ is a paracompact finite-dimensional smooth manifold, consider the Lie algebra
$\cV_c(M)$ of compactly supported smooth vector fields on $M$
as the locally convex direct limit of the spaces $\cV_K(M)$
of smooth vector fields supported in a compact set $K\sub M$
(where $\cV_K(M)$ is endowed with the compact-open $C^\infty$-topology).
\begin{cor}\label{cordffeo}
Let $M$ be a paracompact finite-dimensional smooth manifold
and $\Diff(M)$ be the Lie group of all $C^\infty$-diffeomorphisms of~$M$,
modelled on the locally convex space $\cV_c(M)$.
\begin{itemize}
\item[\rm(a)]
If $M$ is $\sigma$-compact,
then $\Diff(M)$ is $C^0$-regular.
\item[\rm(b)]
If $M$ is not $\sigma$-compact, then
$\Diff(M)$ is $C^1$-regular.
\end{itemize}
\end{cor}
\begin{proof}
(a)
This is a special case of the corresponding result for
diffeomorphism groups of $\sigma$-compact orbifolds in \cite{Alx} (cf.\ also the author's
unpublished preparatory work \cite{DIF}
and later improvements thereof).

(b) Let $(M_j)_{j\in J}$ be the family of pairwise distinct connected components of $M$.
Then $M_j$ is a $\sigma$-compact, open submanifold of $M$ for each $j\in J$ and
thus $\Diff(M_j)$ is $C^0$-regular, by (a).
Therefore the weak direct product
$U:=\bigoplus_{j\in J}\Diff(M_j)$ is $C^1$-regular, by Corollary~\ref{alsodoch}.
Since $U$ is an open subgroup of $\Diff(M)$ by construction of the Lie group
structure thereon, we deduce that also $\Diff(M)$ is $C^1$-regular.
\end{proof}
The same argument shows that diffeomorphism groups of paracompact
orbifolds (as studied in \cite{Alx}) are $C^1$-regular.\\[2.3mm]
The following three lemmas are the key for our proof of Theorem~\ref{great}.
\begin{la}\label{topviaL}
Let $k\in \N_0$ and $E$ be an integral complete locally convex space.
Then
\[
(C^{k+1}([0,1],E),\cO_{C^k_+})\cong E^{k+1}\times (C([0,1],E),\cO_{L^1})
\]
as a locally convex space via $\gamma\mto (\gamma(0),\gamma'(0),\ldots,\gamma^{(k)}(0),\gamma^{(k+1)})$.
\end{la}
\begin{proof}
Let $h_k$ be the map described in the lemma.
%
%
The topology on $(C^{k+1}([0,1],E),\cO_{C^k_+})$ is initial with respect
to the inclusion map $C^{k+1}([0,1],E)\to (C([0,1],E),\cO_{C^0})$
and the map $C^{k+1}([0,1],E)\to (C^k([0,1],E),\cO_{C^{k-1}_+})$, $\gamma\mto\gamma'$
(where $\cO_{C^{-1}_+}:=\cO_{L^1}$).
As a consequence, the linear map
\[
f_k\colon (C^{k+1}([0,1],E),\cO_{C^k_+})\to E\times (C^k([0,1],E),\cO_{C^{k-1}_+}),\;\;
\gamma\mto (\gamma(0),\gamma')
\]
is continuous, for each $k\in\N$.
Now consider the linear map
\[
g_k\colon E\times (C^k([0,1],E),\cO_{C^{k-1}_+})\to (C^{k+1}([0,1],E),\cO_{C^k_+})
\]
determined by
$g_k(y,\gamma)(t)=y+\int_0^t\gamma(s)\,ds$.
Note that $g_k(y,\gamma)'=\gamma$ is continuous in $(y,\gamma)$ as a map
to $(C^k([0,1],E),\cO_{C^{k-1}_+})$ and
\[
\sup_{t\in[0,1]}q(g_k(y,\gamma)(t))\leq q(y)+\|\gamma\|_{L^1,q}
\]
for each continuous seminorm $q$ on $E$,
entailing that $g_k$ is continuous as a map to $(C([0,1],E),\cO_{C^0})$.
We now deduce with the initiality property just mentioned that
$g_k$ is continuous. By the Fundamental Theorem of Calculus, both, $f_k\circ g_k$ and $g_k\circ f_k$
is the respective identity map. Hence $f_k$ is invertible and $f_k^{-1}=g_k$ is continuous.
In particular, $h_0=f_0$ is an isomorphism of topological vector spaces.
If $k\geq 1$, then
$h_k(\gamma)=(\gamma(0),h_{k-1}(\gamma'))$
and thus
\begin{equation}\label{hkfk}
h_k=(\id_E\times h_{k-1})\circ f_k.
\end{equation}
Since $h_{k-1}$
is an isomorphism of topological vector spaces
by induction, and $f_k$ is an isomorphism of topological vector spaces,
we deduce with (\ref{hkfk}) that also $h_k$ is an isomorphism of topological vector spaces. 
\end{proof}
\begin{la}\label{keylemsum}
Let $(E_j)_{j\in J}$ be a family of locally convex spaces.
Give $\bigoplus_{j\in J}E_j$ and
\[
\bigoplus_{j\in J}\left(C([0,1],E_j),\cO_{L^1}\right)
\]
the locally convex direct sum topology.
Let
\[
\Psi\colon (C([0,1],{\textstyle\bigoplus_{j\in J}}E_j),\cO_{L^1})\to\bigoplus_{j\in J}\big(C([0,1],E_j),\cO_{L^1}\big)
\]
be the map taking a function to its family of components.
Then $\Psi$ is an isomorphism of topological vector spaces.
\end{la}
\begin{proof}
We know from \ref{numbacompmap} that $\Psi$ is an isomorphism of vector spaces.
A typical continuous seminorm on the range of $\Psi$ is of the form
\[
q\colon \bigoplus_{j\in J}C([0,1],E_j)\to[0,\infty[,\quad
q((\gamma_j)_{j\in J}:=
\sum_{j\in J}\|\gamma_j\|_{L^1,q_j}
\]
with continuous seminorms $q_j\colon E_j\to[0,\infty[$.
Now 
\[
p\colon \bigoplus_{j\in J}E_j \to[0,\infty[, \quad p((x_j)_{j\in J}):=\sum_{j\in J}q_j(x_j)
\]
is a continuous seminorm on $E:=\bigoplus_{j\in J}E_j$. If $\gamma=(\gamma_j)_{j\in J}\in
C([0,1],E)$, then there is a finite set $F\sub J$ such that $\gamma_j=0$ for all
$j\in J\setminus F$. Then
\begin{eqnarray*}
q\Big(\Psi(\gamma)\Big)
&=&\sum_{j\in J}\|\gamma_j\|_{L^1,q_j}
=\sum_{j\in F}\int_0^1q_j(\gamma_j(t))\,dt\\
&=&\int_0^1\sum_{j\in F}q_j(\gamma_j(t))\,dt
=\int_0^1 p(\gamma(t))\,dt=\|\gamma\|_{L^1,p}.
\end{eqnarray*}
Hence $\Psi$ is continuous
and actually a topological embedding (since $\|.\|_{L^1,p}$
with $p$ as just encountered is a typical
seminorm on the domain of $\Psi$.
\end{proof}
\begin{la}\label{keyla3x}
Let $\ell\in\N\cup\{\infty\}$, $k\in \N_0$ such that $k< \ell$
and $(E_j)_{j\in J}$ be a family of locally convex
spaces. Endow $E:=\bigoplus_{j\in J} E_j$
and
\[
\bigoplus_{j\in J}\big(C^\ell([0,1],E_j),\cO_{C^k_+}\big)
\]
with the locally convex direct sum topology. Let
\[
\Phi\colon \left( C^\ell\big([0,1],{\textstyle\bigoplus_{j\in J}E_j}\big),\cO_{C^k_+}\right)
\to\bigoplus_{j\in J}\big(C^\ell([0,1],E_j),\cO_{C^k_+}\big)
\]
be the map taking a function to its family of components.
Then $\Phi$ is an isomorphism of topological vector spaces.
\end{la}
\begin{proof}
We know that $\Phi$ is an isomorphism of vector spaces
(see \ref{numbacompmap}).
The domain of $\Phi$ is a vector subspace
of $\left( C^{k+1}\big([0,1],\bigoplus_{j\in J}E_j\big),\cO_{C^k_+}\right)$,
endowed with the induced topology.
Moreover, the range of $\Phi$ is a vector subspace of
$\bigoplus_{j\in J}\big(C^{k+1}([0,1],E_j),\cO_{C^k_+}\big)$, endowed with the induced topology.
To see that $\Phi$ is a homeomorphism,
it therefore suffices to consider the case that $\ell=k+1$.
Let $\wt{E}_j$ be a completion of $E_j$ such that $E_j\sub\wt{E}_j$.
The domain of $\Phi$ is a topological vector subspace
of $\left( C^{k+1}\big([0,1],\bigoplus_{j\in J}\wt{E}_j\big),\cO_{C^k_+}\right)$
and
the range of $\Phi$ is a topological vector subspace of
$\bigoplus_{j\in J}\big(C^{k+1}([0,1],\wt{E}_j),\cO_{C^k_+}\big)$.
We may therefore assume that each $E_j$ (and hence also
$\bigoplus_{j\in J}E_j$) is complete.
Lemma~\ref{topviaL} provides isomorphisms
\[
(C^{k+1}([0,1],E_j),\cO_{C^k_+})\to E_j^{k+1}\times (C([0,1],E_j),\cO_{L^1})
\]
of topological vector spaces.
Since locally convex direct sums are compatible with finite direct products,
the preceding isomorphisms combine to an isomorphism
\[
\Gamma\colon
\bigoplus_{j\in J}\big(C^\ell([0,1],E_j),\cO_{C^k_+}\big)\to
E^{k+1}\times \bigoplus_{j\in J}\big(C([0,1],E_j),\cO_{L^1}\big)
\]
of topological vector spaces.
Lemma~\ref{topviaL} also provides an isomorphism
\[
\Theta\colon C^{k+1}([0,1],E)\to E^{k+1}\times (C([0,1],E),\cO_{L^1})
\]
of topological vector spaces.
Finally, Lemma~\ref{keylemsum}
provides an isomorphism
\[
\Psi\colon  (C([0,1],{\textstyle\bigoplus_{j\in J}}E_j),\cO_{L^1})\to\bigoplus_{j\in J}\big(C([0,1],E_j),\cO_{L^1}\big)
\]
of topological vector spaces.
Then also $\id\times\Psi\colon E^{k+1}\times (C([0,1],E),\cO_{L^1})\to E^{k+1}\times \bigoplus_{j\in J}\big(C([0,1],E_j),\cO_{L^1}\big)$ is an isomorphism of topological vector spaces, where $\id\colon E^{k+1}\to E^{k+1}$
is the idenitity map. It is clear from the construction of the preceding mappings that
\[
\Gamma\circ\Phi=(\id\times \Psi)\circ \Theta.
\]
Since $(\id\times \Psi)\circ \Theta$ is an isomorphism of topological vector spaces,
also $\Phi=\Gamma^{-1}\circ (\id\times \Psi)\circ\Theta$
is an isomorphism of topological vector spaces.
\end{proof}
{\bf Proof of Theorem~\ref{great}.}
By Proposition~\ref{weakreg}\,(a),
$G$ is $C^\ell$-semiregular.
By Lemma~\ref{keyla3x}, the map
\[
\Phi\colon (C^\ell([0,1],\cg),\cO_{C^k})\to\bigoplus_{j\in J}(C^\ell([0,1],\cg_j),\cO_{C^k})
\]
taking a $\cg$-valued $C^\ell$-curve to its family of components is an isomorphism
of topological vector spaces.
It remains to observe that $\evol_G=\big(\oplus_{j\in J}\evol_{G_j}
\big)\circ\Phi$, where $\oplus_{j\in J}\evol_{G_j}$ is smooth
(like each $\evol_{G_j}$), by \cite{MEA}.\,\Punkt
\section{Weak direct products of locally {\boldmath$\mu$}-convex groups}\label{secergae}
In this section, we establish stronger regularity properties
for weak direct products $\bigoplus_{j\in J}G_j$ over arbitrary (not necessarily countable) index sets~$J$
if all $G_j$ are sufficiently nice (e.g., Banach-Lie groups).
\begin{defn}
If $G$ is a $C^k$-semiregular Lie group for some $k\in \N_0\cup\{\infty\}$
and $q$ is a continuous seminorm on $\cg:=L(G)$,
we define
\[
\|\gamma\|_{L^1,q}:=\int_0^1q(\gamma(t))\, dt
\]
for $\gamma\in C^k([0,1],\cg)$.
Then $\|.\|_{L^1,q}$
is a seminorm on $C^k([0,1],\cg)$
which is continuous with respect to the $C^0$-topology
(as $\|\gamma\|_{L^1,q}\leq \sup\{q(\gamma(t))\colon t\in [0,1]\}$).
We let $\cO_{L^1}$ be the locally convex vector topology
on $C^k([0,1],\cg)$ defined by the seminorms $\|.\|_{L^1,q}$,
for $q$ ranging through the set of all continuous seminorms on $\cg$.
\end{defn}
In the next lemma, $\gamma*\eta=\Ad(\Evol(\eta))^{-1}.\gamma+\eta$
(as before).
\begin{la}\label{L1ctszero}
If $k\in \N_0\cup\{\infty\}$ and $G$ is a $C^k$-semiregular Lie group with Lie algebra $\cg:=L(G)$,
then right translation
\[
\rho_\eta\colon (C^k([0,1],\cg),\cO_{L^1})\to (C^k([0,1],\cg),\cO_{L^1}),\quad
\gamma\mto\gamma*\eta
\]
is a continuous affine-linear map and hence smooth,
for each $\eta\in C^k([0,1],\cg)$.
As a consequence, $\rho_\eta$ is a homeomorphism and thus
\[
\evol\colon (C^k([0,1],\cg),\cO_{L^1})\to G
\]
is continuous if and only if it is continuous at~$0$.
\end{la}
\begin{proof}
It suffices to show that, for each $\eta\in C^k([0,1],\cg)$,
the linear self-map $\alpha\colon \gamma\mto\Ad(\Evol(\eta))^{-1}\gamma$
is continuous. Then $\rho_\eta$ is continuous, hence a homeomorphism
(as $\rho_{\eta^{-1}}$ is its inverse) and the final conclusion follows
as in the proof of Theorem~D.

Let $q$ be a continuous seminorm on~$\cg$.
The map $h\colon [0,1]\times \cg\to\cg$, $(t,v)\mto \Ad(\Evol(\eta)(t)^{-1})(v)$
is continuous and vanishes on $[0,1]\times\{0\}$.
Then $h([0,1]\times Q)\sub B^q_1(0)$ for some
zero-neighbourhood $Q\sub \cg$ (by the Wallace Lemma).
We have $B^p_1(0)\sub Q$ for some continuous
seminorm $p$ on $\cg$.
If $v\in \cg$ and $r>0$ such that $rp(v)<1$, then
$rq(h(t,v))=q(h(t,rv))<1$ and thus $q(h(t,v))\leq \frac{1}{r}$.
Letting $\frac{1}{r}\to p(v)$, we deduce that $q(h(t,v))\leq p(v)$.
As a consequence, $\|\alpha(\gamma)\|_{L^1,q}\leq \|\gamma\|_{L^1,p}$
for all $\gamma$ and hence $\alpha$ is continuous.
\end{proof}
\begin{defn}\label{defmuc}
If $G$ is a Lie group with modelling space~$E$, $\phi\colon U\to V\sub E$
a chart for $G$ around~$e$ with $\phi(e)=0$ and $q$ a seminorm on $E$, we write
\[
q_\phi(g):=q(\phi(g))
\]
for all $g\in U$.
The Lie group $G$ is called \emph{locally $\mu$-convex}
if there exists a chart $\phi\colon U\to V$ of $G$ with $\phi(e)=0$
such that, for each continuous seminorm $q$ on $E$,
there exists a continuous seminorm $p$ on $E$
and $R>0$ such that
\begin{eqnarray*}
\lefteqn{(\forall r\in \,]0,R])\,(\forall n\in\N)\, (\forall g_1,\ldots, g_n\in U)}\qquad\quad\\
& & \sum_{j=1}^np_\phi(g_j)<r\impl g_1g_2\cdots g_n\in U\;\mbox{and}\;
q_\phi(g_1g_2\cdots g_n)< r.
\end{eqnarray*}
\end{defn}
\begin{rem}
If $G$ is locally $\mu$-convex,
then the statement at the end of Definition~\ref{defmuc}
is valid for every chart $\phi$ with$\phi(e)=0$
(with different choices of $R$ and $p$).
To see this, note that the change of charts is a smooth map
and hence Lipschitz in a suitable sense (which we presently recall).
\end{rem}
We recall from \cite{HOL}:
%
%
\begin{numba}
If $E$ and $F$ are locally convex spaces, $U\sub E$is open and $f\colon U\to F$ a $C^1$-map,
then $f$ is Lipschitz in the sense that, for each $x\in U$ and continuous seminorm
$p$ on $F$, there exists a neighbourhood $W\sub U$ of $x$ and continuous seminorm $q$ on~$E$ such that
\begin{equation}\label{deflips}
(\forall y,z\in U)\;\; p(f(z)-f(y))\leq q(z-y).
\end{equation}
\end{numba}
\begin{prop}\label{Banachismu}
Every Banach-Lie group is locally $\mu$-convex.
\end{prop}
%
\begin{proof}
If $G$ is a Lie group modelled on a Banach-space $(E,\|.\| )$
and $\phi\colon U\to V$ a chart around $1$ such that $\phi(1)=0$,
we can choose an open identity neighbourhood $P\sub U$ such that $PP\sub U$.
Then $Q:=\phi(P)$ is an open zero-neighbourhood in~$V$ and
\[
m \colon Q\times Q\to V,\quad m(x,y):=\phi(\phi^{-1}(x)\phi^{-1}(y))
\]
expresses the group multiplication in the local chart. The first
order Taylor expansion around $(0,0)$ reads
\begin{equation}\label{tay1}
m(x,y)=x+y+\rho(x,y)
\end{equation}
and, after shrinking $Q$ if necessary, there is $C\in \,]0,\infty[$ such that
\begin{equation}\label{tay2}
\|\rho(x,y)\|\leq C\,\|x\|\,\|y\|\quad\mbox{for all $x,y\in Q$}
\end{equation}
(see, e.g., \cite{PAD}). 
After replacing $\|.\|$ with $C\|.\|$, we may assume that $C=1$.
Now $B_\ve(0)\sub Q$ for some $\ve\in \,]0,1]$.
If $q$ is a continuous seminorm on~$E$, then $q\leq a\|.\|$ for some $a\in\,]0,\infty[$.
Let $p:=\max\{1,2a\}\|.\|$,
$R:=\ve\max\{1,2a\}/2$ and $r\in\,]0,R]$.
If $n\in\N$ and $g_1,\ldots, g_n\in U$
such that $\sum_{j=1}^n p_\phi(g_j)<r$,
set $x_j:=\phi(g_j)$ for $j\in\{1,\ldots, n\}$.
Thus $\sum_{j=1}^n p(x_j)<r$ and hence $\sum_{j=1}^n\|x_j\|<\varepsilon/2$.
Then $z_1:=x_1$, $z_{j+1}:=m(z_j,x_{j+1})$
can be formed for all $j\in\{1,\ldots, n-1\}$ and
\[
\|z_j\|\leq 2\sum_{i=1}^j\|x_i\|,
\]
as we verify by induction: If $j=1$,
then $\|z_j\|=\|x_1\|\leq 2\|x_1\|$ indeed.
If the assertion holds for some $j\in\{1,\ldots,n-1\}$,
then $\|z_j\|\leq 2\sum_{i=1}^j\|x_i\|<\ve$ shows that $z_j\in Q$.
Since $\|x_{j+1}\|<\varepsilon/2$, we also have $x_{j+1}\in Q$
and thus $z_{j+1}:=m(z_j,x_{j+1})$ makes sense.
Moreover, by (\ref{tay1}) and (\ref{tay2}),
\begin{eqnarray*}
\|z_{j+1}\|&=&\|m(z_j,x_{j+1})\|=\|z_j+x_{j+1}+\rho(z_j,x_{j+1})\\
&\leq& \|z_j\|+\|x_j\|+\|\rho(z_j,x_{j+1})\|
\leq 2\sum_{i=1}^j\|x_i\|+\|x_{j+1}\|+\|z_j\|\,\|x_{j+1}\|.
\end{eqnarray*}
Since $\|z_j\|<\ve\leq 1$, we deduce that $\|z_{j+1}\|\leq 2\sum_{i=1}^{j+1}\|x_i\|$,
which completes the inductive proof.
It only remains to note that $q_\phi(g_1\cdots g_n)=q(z_n)\leq a\|z_n\|
\leq 2a\sum_{j=1}^n\|x_j\|\leq\sum_{j=1}^np(x_j)<r$.
\end{proof}
\begin{prop}\label{plismuconv}
Let $G=\bigcap_{i\in\N}G_i$ be as in Corollary~{\rm\ref{easierplim}}
and assume that, moreover, each $G_i$ is a Banach-Lie group.
Then $G$ is locally $\mu$-convex.
\end{prop}
\begin{proof}
Let $E$, $E_i$, $\psi\colon U\to V$ and $\psi_i\colon U_i\to V_i$ be as in Corollary~\ref{easierplim},
and $\|.\|_i$ be a norm on $E_i$ defining its topology.
If $q$ is a continuous seminorm on $E$,
then there is $i\in\N$ and $c>0$ such that $q(x)\leq c\|x\|_i$
for all $i\in E$. Set $Q:=c\|.\|_i$. Since $G_i$ is locally $\mu$-convex,
there exists $R>0$ and a continuous seminorm $p$ on $E_i$
such that, for all $r\in\,]0,R]$, $n\in\N$ and $g_1,\ldots, g_n\in U_i$
with $\sum_{j=1}^n p_{\psi_i}(g_j)<r$,
we have $g_1\cdots g_n\in U_i$ and $Q_{\psi_i}(g_1\cdots g_n)<r$.
As a consequence,
for all $r\in\,]0,R]$, $n\in\N$ and $g_1,\ldots, g_n\in U$
with $\sum_{j=1}^n p_{\psi}(g_j)<r$,
we have $g_1\cdots g_n\in U_i\cap G=U$ and $q_\psi(g_1\cdots g_n)\leq Q_{\psi_i}(g_1\cdots g_n)<r$.
\end{proof}
\begin{example}\label{preprex}
If $M$ is a compact smooth manifold with or without boundary
and $H$ is a Banach-Lie group,
then $C^k(M,H)$ is locally $\mu$-convex for all $k\in \N_0\cup\{\infty\}$.\\[2.3mm]
[If $k$ is finite, then $C^k(M,H)$ is a Banach-Lie group and hence locally $\mu$-convex by
Proposition~\ref{Banachismu}.
To $C^\infty(M,H)=\bigcap_{i\in \N_0}C^i(M,H)$,
Proposition~\ref{plismuconv} applies.\,]
\end{example}
\begin{la}\label{getevoct}
Let $k\in \N_0\cup\{\infty\}$ and $G$ be a $C^k$-regular Lie group
which is locally $\mu$-convex.
Then $\evol\colon (C^k([0,1],\cg),\cO_{L^1})\to G$ is continuous.
\end{la}
\begin{proof}
Let $Q$ be an identity neighbourhood in~$G$.
Let $\phi\colon U\to V\sub E$, $p$, $q$ and $R$ be as in the definition
of local $\mu$-convexity; we may assume that $E=\cg$ and $d\phi|_\cg=\id_{\cg}$.
Assume that $q$ has been chosen such that $B^q_s(0)\sub V$
and $\phi^{-1}(B^q_s(0))\sub Q$ for some $s>0$.
After shrinking $R$, we may assume that $R\leq s$.
Because $\evol\colon C^k([0,1],\cg)\to G$ is smooth
with respect to the $C^k$-topology on its domain,
there is a $0$-neighbourhood
$W\sub C^k([0,1],\cg)$ such that $\evol(W)\sub U$.
We may assume that $W=\{\gamma\in C^k([0,1],\cg)\colon \|\gamma\|_{C^\ell,P}<1\}$
for a continuous seminorm~$P$ on~$\cg$
and $\ell\in \N_0$ with $\ell\leq k$.
Since $\phi\circ\evol$ is smooth
and
\[
(\phi\circ\evol)'(0)(\gamma)=\int_0^1\gamma(t)\,dt \quad\mbox{for all
$\gamma\in C^k([0,1],\cg)$,}
\]
we have
\begin{equation}\label{gammone}
(\phi\circ\evol)(\gamma)=\int_0^1\gamma(t)\, dt\,+\, \rho(\gamma)\quad
\mbox{for all $\gamma\in C^k([0,1],\cg)$,}
\end{equation}
with the first order Taylor remainder $\rho(\gamma)\in \cg$.
After increasing $\ell$ and $P$ if necessary, we may assume that
\[
p(\rho(\gamma))\leq (\|\gamma\|_{C^\ell,P})^2
\]
for all $\gamma\in W$.
%
%
 We may also assume that $P\geq p$ pointwise. Now
\[
W_1:=\left\{\gamma\in C^k([0,1],\cg)\colon \|\gamma\|_{L^1,P}<\frac{R}{2}\right\}
\]
is a zero-neighbourhood in $(C^k([0,1],\cg),\cO_{L^1})$.
Given $\gamma\in W_1$ and $n\in \N$, define
$\gamma_{n,j}\in C^k([0,1],\cg)$ for $j\in\{1,\ldots, n\}$ via
\[
\gamma_{n,j}(t):=\frac{1}{n}\, \gamma\left(\frac{j-1+t}{n}\right).
\]
Then
\begin{equation}\label{recogamm}
\evol(\gamma)=\evol(\gamma_{n,1})\evol(\gamma_{n,2})\cdots\evol(\gamma_{n,n}).
\end{equation}
Choose $N\in \N$ so large that $\frac{1}{N}\|\gamma\|_{C^\ell,P}<1$
and  $\frac{1}{N}(\|\gamma\|_{C^\ell,P})^2<\frac{R}{2}$.
Then $\|\gamma_{n,j}\|_{C^\ell,P}\leq\frac{1}{n}\|\gamma\|_{C^\ell,P}<1$
for all $n\geq N$ and $j\in\{1,\ldots, n\}$, whence $\gamma_{n,j}\in W$;
moreover,
\begin{equation}\label{gammtwo}
p(\rho(\gamma_{n,j}))\leq \frac{1}{n}\frac{(\|\gamma\|_{C^\ell,P})^2}{N}<\frac{R}{2n}.
\end{equation}
Combining (\ref{gammone}) and (\ref{gammtwo}),
we obtain
\begin{eqnarray*}
p_\phi(\evol(\gamma_{n,j}))& =& p(\phi(\evol(\gamma_{n,j})))
\leq \int_0^1p(\gamma_{n,j}(t))\,dt\,+\, p(\rho(\gamma))\\
&<& \int_{(j-1)/n}^{j/n}p(\gamma(u))\,du\,+\,\frac{R}{2n},
\end{eqnarray*}
entailing that
\begin{eqnarray*}
\sum_{j=1}^n p_\phi(\evol(\gamma_{n,j}))&<&\int_0^1 p(\gamma(t)\,dt+\frac{1}{2R}\\
&=& \|\gamma\|_{L^1,p}+\frac{R}{2}\leq \|\gamma\|_{L^1,P}+\frac{R}{2}<R.
\end{eqnarray*}
Using (\ref{recogamm}) and the definition of local $\mu$-convexity,
we deduce that
\[
\evol(\gamma)=\evol(\gamma_{n,1})\evol(\gamma_{n,2})\cdots\evol(\gamma_{n,n})
\in U
\]
and $q_\phi(\evol(\gamma))<R$.
Hence $\evol(\gamma)\in \phi^{-1}(B^q_R(0))\sub
\phi^{-1}(B^q_s(0))\sub Q$.
Thus $\evol(W_1)\sub Q$,
showing that $\evol\colon (C^k([0,1],\cg),\cO_{L^1})\to G$
is continuous at~$0$ and hence continuous,
by Lemma~\ref{L1ctszero}.
\end{proof}
\begin{prop}\label{propweakpro2}
Let $k\in \N_0\cup\{\infty\}$ and
$(G_j)_{j\in J}$ be a family
of Lie groups $G_j$, with Lie algebra $\cg_j:=L(G_j)$.
Let $G:=\bigoplus_{j\in J}G_j$ be the weak direct product
and $\cg:=L(G)$.
If $G_j$ is $C^k$-regular for each $j\in J$ and $\evol_{G_j}\colon (C^k([0,1],\cg_j),\cO_{L^1})\to G_j$
is continuous, then $G$ is $C^k$-regular and $\evol\colon (C^k([0,1],\cg),\cO_{L^1})\to G$
is continuous.
\end{prop}
\begin{proof}
As in the proof of Lemma~\ref{keylemsum}, we see that the map
\[
\Phi\colon (C^k([0,1],\cg),\cO_{L^1})\to\bigoplus_{j\in J}(C^k([0,1],\cg_j),\cO_{L^1})
\]
taking a $\cg$-valued $C^k$-curve to its family of components is an isomorphism
of topological vector spaces.\footnote{Replace $E_j$ with $\cg_j$ and continuous
functions with $C^k$-functions.}
We already know from Proposition~\ref{weakreg}
that $G$ is $C^k$-semiregular.
It remains to observe that $\evol_G=\big(\oplus_{j\in J}\evol_{G_j}
\big)\circ\Phi$, where $\oplus_{j\in J}\evol_{G_j}$ is smooth
(like each $\evol_{G_j}$), by \cite{MEA}.
\end{proof}
\begin{cor}
The test function group $C^\ell_c(M,H)$ is $C^0$-regular
for each $\ell\in\N_0\cup\{\infty\}$, Banach-Lie group $H$ and paracompact finite-dimensional
smooth manifold~$M$.
\end{cor}
\begin{proof}
Let $(M_j)_{j\in J}$ be a locally finite family of
compact submanifolds with boundary of $M$ such that
the interiors $M_j^0$ cover~$M$.
By Example~\ref{preprex}, Lemma~\ref{getevoct}
and Proposition~\ref{propweakpro2},
the weak direct product $\bigoplus_{j\in J}C^\ell(M_j,H)$
is $C^0$-regular. Hence also $C^\ell_c(M,H)$ is
$C^0$-regular, by
\ref{preparetestfu}.
\end{proof}
\section{Regularity of Silva-Lie groups}
We describe a regularity criterion for
Lie groups modelled on Silva spaces
(i.e., (DFS)-spaces).
It implies that the Lie
group $\Diff^\omega(M)$ of real analytic diffeomorphisms
of a compact real analytic manifold $M$
(as studied in \cite{DaS}, cf.\ \cite{Les})
is $C^1$-regular.\footnote{For the proof
of regularity of $\Diff^\omega(M)$
in the sense of convenient differential
calculus, see already \cite{KaM}.}\\[2.3mm]
We shall use a version of Ascoli's Theorem
for $C^k$-maps.
\begin{la}\label{asco}
Let $E$ be an integral complete locally convex space, $k\in \N_0$
and $\Gamma \sub C^k([0,1],E)$ be a subset with the following
properties:
\begin{itemize}
\item[\rm(a)]
$\{\gamma^{(j)}(0)\colon \gamma\in \Gamma\}$ is relatively
compact in $E$ for each $j\in\{0,1,\ldots, k-1\}$;
\item[\rm(b)]
$\{\gamma^{(k)}(t)\colon \gamma\in \Gamma\}$
is relatively compact in $E$
for each $t\in[0,1]$; and
\item[\rm(c)]
$\{\gamma^{(k)}\colon \gamma\in \Gamma\}$ is an equicontinuous
subset of $C([0,1],E)$.
\end{itemize}
Then $\Gamma$ is relatively compact
in $C^k([0,1],E)$.
\end{la}
\begin{proof}
The map
\[
\Phi\colon C^k([0,1],E)\to E^k\times C([0,1],E)
\]
with $\Phi(\gamma):=(\gamma(0),\ldots,\gamma^{(k-1)}(0),\gamma^{(k)})$
is an isomorphism of topological vector spaces.
In view of (b) and (c),
Ascoli's Theorem shows that
the set $\{\gamma^{(k)}\colon \gamma\in \Gamma\}$
has compact closure $K$ in $C([0,1],E)$.
By (a), also the closure $L_j$ of $\{\gamma^{(j)}(0)\colon \gamma\in \Gamma\}$
in~$E$ is compact, for all $j\in \{0,1,\ldots,k-1\}$.
Since
\[
\Phi(\Gamma)\sub L_0\times\cdots\times L_{k-1}\times K,
\]
the set $\Phi(\Gamma)$ is relatively compact and hence
so is $\Gamma$.
\end{proof}
The following lemma can be useful in proofs of regularity
for Lie groups modelled on locally convex direct limits.
\begin{la}\label{evglobalze}
Let $G$ be a Lie group, with Lie algebra $\cg:=L(G)$,
and $F\sub \cg$ be a vector subspace, endowed with
a locally convex vector topology $($which need not be induced by $\cg)$.
Let $k,m\in\N_0\cup\{\infty\}$.
If $\Evol(\gamma)$ exists for $\gamma$ in an open $0$-neighbourhood
$W\sub C^k([0,1],F)$ and $\evol \colon W\to G$
is $C^m$, then $\Evol(\gamma)$ exists for all
$\gamma\in C^k([0,1],F)$ and $\evol\colon C^k([0,1],F)\to G$
is $C^m$.
\end{la}
\begin{proof}
After shrinking $W$,
we may assume that
\[
W=\{\gamma\in C^k([0,1],F)\colon \|\gamma\|_{p,C^\ell}<1\}
\]
for some $\ell\in \N_0$ such that $\ell\leq k$ and some continuous
seminrom $p$ on $F$.
Now the usual inductive argument
(as in Lemma~\ref{??} or \cite{Dah}) shows that $\Evol$ is defined on $2^nW$ for all
$n\in \N_0$ and that $\evol \colon 2^nW\to G$ is $C^m$.
The assertion follows since $C^k([0,1],F)=\bigcup_{n\in \N_0}2^nW$.
\end{proof}
\begin{la}\label{C1curveSilva}
Let $E_1\sub E_2\sub\cdots$ be Banach spaces
such that the inclusion map $E_n\to E_{n+1}$
is a compact operator for each $n\in \N$.
Endow $E:=\bigcup_{n\in \N}E_n=\dl\,E_n$,\vspace{-.3mm}
with the locally convex direct limit topology.
Let~$Y$ be a topological space, $k\in \N_0$
and $P_1\sub P_2\sub\cdots$ be an ascending sequence
of open convex $0$-neighbourhoods $P_n\sub C^k([0,1],E_n)$,
entailing that $P:=\bigcup_{n\in \N}P_n$
is an open $0$-neighbourhood in $C^k([0,1],E)$.
Let
\[
f\colon P\to Y
\]
be a map such that $f|_{P_n}\colon P_n\to Y$
is continuous with respect to the topology induced by
the Banach space $C^k([0,1],E_n)$ on~$P_n$, for each $n\in \N$.
Then $f|_{P\cap C^{k+1}([0,1],E)}\colon P\cap C^{k+1}([0,1],E)\to Y$
is continuous with respect
to the $C^{k+1}$-topology on~$P\cap C^{k+1}([0,1],E)$.
\end{la}
\begin{proof}
Being a Silva space, the locally convex direct
limit $E=\dl\,E_n$\vspace{-.3mm} is complete and
compactly regular,
i.e., every compact subset of $E$ is
contained in $E_n$ for some $n\in \N$
and compact also in $E_n$ (see \cite{Flo}).
Hence, by Mujica's Theorem for $C^k$-curves~\cite[Lemma~1.2.7\,(b)]{Dah}, we have
$C^k([0,1],E)=\dl\,C^k([0,1],E_n)$ and $C^{k+1}([0,1],E)=\dl\,C^{k+1}([0,1],E_n)$
as locally convex spaces. In particular, this implies
that $P$
is indeed open in $C^k([0,1],E)$.\\[2.3mm]
It suffices to show that $f|_{P\cap C^{k+1}([0,1],E)}$ is continuous at~$0$.
In fact, given $x\in P$ we may assume $x\in P_1$ after passung to a cofinal
subsequence. Then $P_n-x$ is a convex open $0$-neighbourhood in $C^k([0,1],E)$
for each $n\in \N$ and if we know that
$P-x\to Y$, $y\mapsto f(x+y)$ is continuous at~$0$, then $f$ is continuous
at~$x$.\\[2.3mm]
Let $U\sub Y$ be an open neighbourhood of $f(0)$.
Since $f|_{P_2}$ is continuous, $P_2$ contains a closed $0$-neighbourhood
$A_1\sub C^k([0,1],E_2)$ with $f(A_1)\sub U$.
As the incusion map $C^k([0,1],E_1)\to C^k([0,1],E_2)$
is continuous linear, we find $r_1>0$
such that
\[
W_1:=\{\gamma\in C^k([0,1],E_1)\colon \|\gamma\|_{C^k([0,1],E_n)}<r_1\}\sub A_1.
\]
The set
\[
V_1:= \{\gamma\in C^{k+1}([0,1],E_1)\colon \|\gamma\|_{C^{k+1}([0,1],E_n)}<r_1\}
\]
satisfies the hypotheses of Lemma~\ref{asco}
(recalling that the closure of $B^{E_1}_{r_1}(0)$
in the Banach space $E_2$ is compact).
Therefore the closure $\wb{V_1}$ of $V_1$ in $C^k([0,1],E_2)$
is compact. Moreover, $\wb{V_1}\sub A_1$ and thus $f(\wb{V_1})\sub U$.
Given $n\in \N$, assume that, for $j\in \{1,\ldots, n\}$,
open convex $0$-neighbourhoods
$V_j$ in $C^{k+1}([0,1],E_n)$ have been found
such that the closure $\wb{V_j}$ of $V_j$
in $C^k([0,1],E_{j+1})$ is compact, $\wb{V_j}\sub P_{j+1}$,
$f(\wb{V_j})\sub U$
and $V_1\sub V_2\sub\cdots\sub V_n$.
Since $f|_{P_{n+2}}$ is continuous and $f(\wb{V_n})\sub U$,
the preimage $(f|_{P_{n+2}})^{-1}(U)$
is an open subset
of $C^k([0,1],E_{n+2})$ which contains the compact set~$\wb{V_n}$.
Using \cite[Theorem 4.10]{HaR}, we find a closed $0$-neighbourhood
$A_{n+1}$ in $C^k([0,1],E_{n+2})$ such that
$\wb{V_n}+A_{n+1}
\sub (f|_{P_{n+2}})^{-1}(U)$ and thus
\[
f(\wb{V_n}+A_{n+1})\sub U.
\]
Note that $\wb{V_n}+A_{n+1}$ is closed in $C^k([0,1],E_{n+2})$.
Let $A_{n+1}^0$ be the interior of $A_{n+1}$ in $C^k([0,1],E_{n+2})$.
The map
\[
h\colon \wb{V_n}\times C^k([0,1], E_{n+1})\to C^k([0,1],E_{n+2}),\;\;
h(\gamma,\eta):=\gamma+\eta
\]
is continuous, entailing that there exists $r_{n+1}>0$
such that $\gamma+\eta\in \wb{V_n}+A_{n+1}^0$
for all $\gamma\in \wb{V_n}$ and $\eta$ in the set
\[
W_{n+1}:=\{\eta\in C^k([0,1],E_{n+1})\colon \|\eta\|_{C^k([0,1],E_{n+1})}<r_{n+1}\}.
\]
As above, we see that
\[
S_{n+1}:=\{\eta\in C^{k+1}([0,1],E_{n+1})\colon \|\eta\|_{C^{k+1}([0,1],E_{n+1})}<r_{n+1}\}
\]
has compact closure $\wb{S_{n+1}}$ in $C^k([0,1],E_{n+2})$.
Then
\[
V_{n+1}:=\wb{V_n}+S_{n+1}
\]
is a convex open $0$-neighbourhood
in $C^{k+1}([0,1],E_{n+1})$ with $V_n\sub V_{n+1}$ and we have
\[
\wb{V_{n+1}}=\wb{V_n}+\wb{S_{n+1}}\sub \wb{V_n}+\wb{W_{n+1}}
\sub \wb{\wb{V_n}+W_{n+1}}\sub \wb{V_n}+A_{n+1}
\]
for the closures in $C^k([0,1],E_{n+2})$.
Hence $\wb{V_{n+1}}\sub P_{n+2}$ in particular
and $f(\wb{V_{n+1}})\sub U$.
Then $V:=\bigcup_{n\in \N}V_n$
is an open $0$-neighbourhood in the locally convex direct limit
$C^{k+1}([0,1],E)=\dl\,C^{k+1}([0,1],E_n)$
such that $f(V)\sub U$ and thus $f|_{P\cap C^{k+1}([0,1],E)}$
is continuous at~$0$.
\end{proof}
\begin{la}\label{C1curveSilva2}
Let $E_1\sub E_2\sub\cdots$ be Banach spaces
such that the inclusion map $E_n\to E_{n+1}$
is a compact operator for each $n\in \N$.
Endow $E:=\bigcup_{n\in \N}E_n=\dl\,E_n$,\vspace{-.3mm}
with the locally convex direct limit topology.
Let $k\in \N_0$, $m\in \N_0\cup\{\infty\}$,
$Y$ be a $C^m$-manifold
and $P_1\sub P_2\sub\cdots$ be an ascending sequence
of open convex $0$-neighbourhoods $P_n\sub C^k([0,1],E_n)$,
entailing that $P:=\bigcup_{n\in \N}P_n$
is an open $0$-neighbourhood in $C^k([0,1],E)$.
Let
\[
f\colon P\to Y
\]
be a map such that $f|_{P_n}\colon P_n\to Y$
is $C^m$ with respect to the topology induced by
the Banach space $C^k([0,1],E_n)$ on~$P_n$, for each $n\in \N$.
Then $f|_{P\cap C^{k+1}([0,1],E)}\colon P\cap C^{k+1}([0,1],E)\to Y$
is $C^m$ with respect
to the $C^{k+1}$-topology on~$P\cap C^{k+1}([0,1],E)$.
\end{la}
\begin{proof}
We may assume that $m\in \N_0$
and proceed by induction on~$m$.
The case $m=0$ is covered by Lemma~\ref{C1curveSilva}.
If $m\geq 1$, then $h:=f|_{P\cap C^{k+1}([0,1],E)}$
is continuous in particular.
Moreover, the map
\[
\omega\colon (P\cap C^{k+1}([0,1],E))\times C^{k+1}([0,1],E)\to TY
\]
which coincides with
$Tf|_{P_n}$ on $(P_n\cap C^{k+1}([0,1],E))\times C^{k+1}([0,1],E)$
is $C^{m-1}$ by the inductive hypothesis.
If $\theta\colon [{-\ve},\ve]\,\to P\cap C^{k+1}([0,1],E)$
is a $C^1$-curve, then
there exists $n\in \N$ such that $\gamma$ is a $C^1$-curve
to $C^{k+1}([0,1],E_n)$ (as follows from compact regularity of
$E=\dl\,E_n$).
%
%
After increasing $n$ if necessary, we may assume that
$\gamma([{-\ve},\ve])\sub P_n$.
As a consequence,
$h\circ \gamma$ is a $C^1$-curve and
$(h\circ\gamma)'(t)=\omega(\gamma(t),\gamma'(t))$.
Hence, by Lemma~\ref{getC1tool}, $h=f|_{P\cap C^{k+1}([0,1],E)}$
is $C^1$ with $Th=\omega$ a $C^{m-1}$-map and thus $h=f|_{P\cap C^{k+1}([0,1],E)}$
is~$C^m$.
\end{proof}
\begin{prop}\label{C1regSilva}
Let $G$ be a
Lie group whose Lie algebra $\cg:=L(G)$
is a Silva space.
Assume that $\cg=\dl\,E_n$\vspace{-.3mm}
for an ascending sequence $E_1\sub E_2\sub\cdots$
of Banach spaces such that the inclusion
map $E_n\to E_{n+1}$ is a compact operator
for each $n\in \N$.
Also assume the following:
\begin{itemize}
\item[\rm(a)]
For each $n\in \N$,
there is an open $0$-neighbourhood
$P_n\sub C([0,1],E_n)$
such that an evolution $\Evol(\gamma)$
exists for each $\gamma\in P_n$
and $\evol\colon P_n\to G$ is continuous
with respect to the topology induced by the Banach space $C([0,1],E_n)$
on $P_n$.
\item[\rm(b)]
There is a point-separating family $(\alpha_j)_{j\in J}$
of smooth group homomorphisms $\alpha_j\colon G\to H_j$
to $C^0$-regular Lie groups~$H_j$.
\end{itemize}
Then $G$ is $C^1$-regular and $C^0$-semiregular.
The same conclusion holds if $\evol$ is replaced
with the right evolution $\evol^r$
in {\rm(a)}.
\end{prop}
\begin{proof}
By Lemma~\ref{evglobalze},
we may assume that $P_n=C([0,1],E_n)$ for each $n\in \N$.
Thus each $\gamma\in \bigcup_{n\in \N}C([0,1],E_n)=C([0,1],\cg)$
has an evolution $\Evol(\gamma)$ and hence $G$ is $C^0$-semiregular.
Lemma~\ref{C1curveSilva} applies to $f:=\evol\colon C([0,1],\cg)\to G$,
whence
\[
\evol|_{C^1([0,1],\cg)}\colon C^1([0,1],\cg)\to G
\]
is continuous. Use the same symbol, $\evol$,
for the latter restriction.
Now the variant of Theorem~F indicated
in Remark~\ref{variantThmF} shows that $\evol$ is smooth.
If $\Evol^r(\gamma)$ exists for $\gamma\in P_n$
and $\evol^r\colon P_n\to G$ is continuous,
then ... shows that $\Evol(\gamma)=\Evol(-\gamma)^{-1}$
exists for all $\gamma\in -P_n$
and $\evol|_{P_n}$ is continuous.
Now apply the above case.
\end{proof}
%
%
%
\begin{prop}\label{C1regSilva2}
Let $G$ be a
Lie group whose Lie algebra $\cg:=L(G)$
is a Silva space.
Assume that $\cg=\dl\,E_n$\vspace{-.3mm}
for an ascending sequence $E_1\sub E_2\sub\cdots$
of Banach spaces such that the inclusion
map $E_n\to E_{n+1}$ is a compact operator
for each $n\in \N$.
Also assume that there is $k\in \N_0$
such that, for each $n\in \N$,
there is an open $0$-neighbourhood
$P_n\sub C^k([0,1],E_n)$
such that an evolution $\Evol(\gamma)$
exists for each $\gamma\in P_n$
and $\evol\colon P_n\to G$ is $C^1$
with respect to the topology induced by the Banach space $C^k([0,1],E_n)$
on $P_n$.
Then $G$ is $C^{k+1}$-regular and $C^k$-semiregular.
The same conclusion holds if $\evol$ is replaced
with the right evolution $\evol^r$.
\end{prop}
\begin{proof}
By Lemma~\ref{evglobalze},
we may assume that $P_n=C^k([0,1],E_n)$ for each $n\in \N$.
Thus each $\gamma\in \bigcup_{n\in \N}C^k([0,1],E_n)=C^k([0,1],\cg)$
has an evolution $\Evol(\gamma)$ and hence $G$ is $C^k$-semiregular.
Lemma~\ref{C1curveSilva2} applies to $f:=\evol\colon C^k([0,1],\cg)\to G$,
whence
\[
\evol|_{C^{k+1}([0,1],\cg)}\colon C^{k+1}([0,1],\cg)\to G
\]
is $C^1$
and hence $C^\infty$ (by Theorem~E).
If $\Evol^r(\gamma)$ exists for $\gamma\in P_n$
and $\evol^r\colon P_n\to G$ is~$C^1$,
then ... shows that $\Evol(\gamma)=\Evol(-\gamma)^{-1}$
exists for all $\gamma\in -P_n$
and $\evol|_{P_n}$ is~$C^1$.
Now apply the above case.
\end{proof}
%
%
\begin{cor}\label{anaDiffreg}
For each compact real-analytic manifold $M$,
the Lie group $\Diff^\omega(M)$ is $C^1$-regular.
\end{cor}
\begin{proof}
The inclusion map $\alpha\colon \Diff^\omega(M)\to \Diff(M)$
into the $C^0$-regular Lie group $\Diff(M)$
of all smooth diffeomorphisms
is a smooth group homomorphism,
whence hypothesis (b) of Proposition~\ref{C1regSilva}
is satisfied. One can check
that also all other hypotheses
are satisfied using right evolution map, and hence the proposition
applies.
On the author's request,
further details will be given in a later
version of~\cite{DaS}.
\end{proof}
{\small
Helge  Gl\"{o}ckner, Universit\"at Paderborn,
Institut f\"{u}r Mathematik,\\
Warburger Str.\ 100, 33098 Paderborn, Germany;
\,Email: {\tt glockner@math.upb.de}}

\begin{thebibliography}{99}\itemsep+2pt
%
%
\bibitem{Alb}
Albeverio, S.\,A.,
H\oe{}egh-Krohn, R.\,J., J.\,A. Marion, D.\,H. Testard,
B.\,S. Torr\'{e}sani, ``Noncommutative
Distributions,''
Marcel Dekker, Inc., New York, 1993.
%
%
\bibitem{Alz}
Alzaareer, H.,
\emph{Lie groups of mappings on non-compact spaces and manifolds},
Doctoral Dissertation,
Universit\"{a}t Paderborn, 2013;
see {\tt http://nbn-resolving.de/urn:nbn:de:hbz:466:2-11572}
%
%
\bibitem{AS}
Alzaareer, H. and A. Schmeding,
\emph{Differentiable mappings on products with different degrees of differentiability in the two
factors}, Expo.\ Math.\ {\bf 33} (2015), no.\ 2, 184--222.
%
%
\bibitem{BGN}
Bertram, W., H. Gl\"{o}ckner and K.-H. Neeb,
\emph{Differential calculus over general base fields and rings},
Expo.\ Math.\ {\bf 22} (2004), 213--282.
%
%
\bibitem{Dah}
Dahmen, R.,
\emph{Direct limit constructions in infinite dimensional Lie theory},
Ph.D.-thesis,
Universit\"{a}t Paderborn;
{\tt
http://nbn-resolving.de/urn:nbn:de:hbz:466:2-239}
%
%
\bibitem{Da2}
Dahmen, R.,
\emph{Regularity in Milnor's sense for ascending unions of Banach-Lie groups},
J. Lie Theory, {\bf 24} (2014), 545--560.
%
%
\bibitem{DGS}
Dahmen, R., H. Gl\"{o}ckner, and A. Schmeding,
\emph{Complexifications of infinite-dimensional manifolds and new constructions
of infinite-dimensional Lie groups}, preprint,
{\tt arXiv:1410.6468}.
%
%
\bibitem{DaS}
Dahmen, R. and A. Schmeding,
\emph{The Lie group of real analytic diffeomorphisms is not real analytic},
Studia Math.\ {\bf 229} (2015), no.\,2, 141--172.
%
%
\bibitem{Flo}
Floret, K.,
\emph{Lokalkonvexe Sequenzen mit kompakten Abbildungen},
J. Reine Angew.\ Math. {\bf 247} (1971), 155--195.
%
%
\bibitem{RES} Gl\"{o}ckner, H., \emph{Lie groups
without completeness restrictions},
Banach Center Publ.\ {\bf 55} (2002), 43--59.
%
%
\bibitem{ALG}
Gl\"{o}ckner, H.,
\emph{Algebras whose groups of units are Lie groups},
Studia Math.\ {\bf 153} (2002), 147--177. 
%
%
\bibitem{QUO}
Gl\"{o}ckner, H.,
\emph{Lie group structures on quotient groups and universal
complexifications for infinite-dimensional Lie groups},
J. Funct.\ Anal.\ {\bf 194} (2002), 347--409.
%
\bibitem{DIF}
Gl\"{o}ckner, H., \emph{Patched locally convex spaces, almost local mappings, and the diffeomorphism
groups of non-compact manifolds}, manuscript, 2002.
%
%
\bibitem{GEM}
Gl\"{o}ckner, H.,
\emph{Lie groups of germs of analytic mappings},
pp. 1--16 in: Turaev, V. and T. Wurzbacher (eds.),
``Infinite Dimensional Groups and Manifolds,''
IRMA Lecture Notes in Mathematics and Theoretical Physics,
de Gruyter, 2004.
%
%
\bibitem{MEA}
Gl\"{o}ckner, H., \emph{Lie groups of measurable mappings},
Canadian J. Math.\ {\bf 55} (2003), 969--999. 
%
%
\bibitem{ZOO}
Gl\"{o}ckner, H.,
\emph{Lie groups over non-discrete topological fields},
preprint, {\tt arXiv:math/0408008v1}.
%
%
\bibitem{FUN}
Gl\"{o}ckner, H.,
\emph{Fundamentals of direct limit Lie theory},
Compositio Math.\ {\bf 141} (2005), 1551--1577. 
%
%
\bibitem{HOL}
Gl\"{o}ckner, H., \emph{H\"{o}lder continuous homomorphisms between infinite-dimensional
Lie groups are smooth}, J. Funct.\ Anal.\ {\bf 228} (2005), no.\,2, 419--444. 
%
%
\bibitem{PAD}
Gl\"{o}ckner, H., \emph{Every smooth $p$-adic Lie group admits a compatible
analytic structure},
Forum Math.\ {\bf 18} (2006), 45--84. 
%
%
\bibitem{COMP}
Gl\"{o}ckner, H.,
\emph{Direct limits of infinite-dimensional Lie groups compared with direct limits in related categories},
J. Funct.\ Anal.\ {\bf 245} (2007), 19--61.
%
%
\bibitem{GDL}
Gl\"{o}ckner, H.,
\emph{Direct limits of infinite-dimensional Lie groups},
pp.\,243--280 in:
K.-H. Neeb and A. Pianzola (eds).,
``Developments and Trends in
Infinite-Dimensional Lie Theory,''
Progr.\ Math.\ {\bf 288},
Birkh\"{a}user, Boston, 2011.
%
%
\bibitem{OWR}
Gl\"{o}ckner, H., \emph{Regularity properties
of infinite-dimensional Lie groups},
Oberwolfach Rep.\ {\bf 13} (2013), 791--794. 
%
%
\bibitem{MER}
Gl\"{o}ckner, H., \emph{Measurable
regularity properties of infinite-dimensional Lie groups},
preprint, {\tt arXiv:1601.02568}.
%
%
\bibitem{GaN}
Gl\"{o}ckner, H. and K.-H. Neeb,
``Infinite-Dimensional Lie Groups,''
book in preparation.
%
%
\bibitem{HaR}
Hewitt, E. and K.\,A. Ross,
``Abstract Harmonic Analysis I,''
Springer, New York, ${}^2$1979.
%
%
\bibitem{Kel}
Kelley, L., ``General Topology,'' Springer, New York, 1975.
%
%
\bibitem{KaM} Kriegl, A. and P.\,W. Michor,
``The Convenient Setting of Global Analysis,''
AMS, Providence, 1997.
%
%
\bibitem{KMR}
Kriegl, A. and P.\,W. Michor, \emph{Regular infinite-dimensional Lie groups},
J. Lie Theory {\bf 7} (1997), 61--99.
%
%
\bibitem{Les}
Leslie, J.,
\emph{On the group of real-analytic diffeomorphisms of a compact
real analytic manifold}, Trans.\ Amer.\ Math.\ Soc.\
{\bf 274} (1982),
651--669.
%
%
\bibitem{Mael}
Michael, E.\,A.,
"Locally Multiplicatively-Convex Topological Algebras,''
Memoirs of the AMS {\bf 11}, 1952.
%
%
\bibitem{Mic}
Michor, P.\,W., ``Manifolds of Differentiable Mappings,''
Shiva Publishing, Orpington, 1980.
%
%
\bibitem{MaT} Michor, P.W. and J. Teichmann,
\emph{Description of infinite-dimensional abelian regular Lie groups},
J. Lie Theory {\bf 9} (1999), 487--489. 
%
%
\bibitem{Mil} Milnor, J., \emph{Remarks on infinite-dimensional Lie groups},
pp.\,1007--1057 in: B.\,S. DeWitt and R. Stora (eds.),
``Relativit\'{e}, groupes et topologie II,'' North-Holland,
Amsterdam, 1984.
%
%
\bibitem{NAB}
Neeb, K.-H.,
\emph{Central extensions of infinite-dimensional Lie groups},
Ann.\ Inst.\ Fourier (Grenoble) {\bf 52}
(2002), 1365--1442. 
%
%
\bibitem{SUR}
Neeb, K.-H.,
\emph{Towards a Lie theory of locally convex groups},
Jpn.\ J.\linebreak
Math.\ {\bf 1} (2006), 291--468. 
%
%
\bibitem{NaS}
Neeb, K.-H. and H. Salmasian,
\emph{Differentiable vectors and unitary representations of
Fr\'{e}chet-Lie supergroups},
Math.\ Z. {\bf 275} (2013), 419--451. 
%
%
\bibitem{NaW} Neeb, K.-H. and F. Wagemann,
\emph{Lie group structures on groups of smooth and holomorphic maps on non-compact manifolds},
Geom.\ Dedicata {\bf 134} (2008), 17--60. 
%
%
\bibitem{OMO}
Omori, H., Y. Maeda, A. Yoshioka, and
O. Kobayashi,
\emph{On regular Fr\'{e}chet-Lie groups}, IV,
Tokyo J. Math.\ {\bf 5} (1982), 365--398. 
%
%
\bibitem{OMBOOK}
Omori, H., "Infinite-Dimensional Lie Groups,'' AMS, 1997.
%
%
\bibitem{Pie}
Pieper, T., \emph{Lie groups of unbounded weighted mappings},
Master's thesis, Universit\"{a}t Paderborn,
June 2014 (advisor: H. Gl\"{o}ckner).
%
%
\bibitem{Alx} Schmeding, A.,
\emph{The diffeomorphism group of a non-compact orbifold},
Diss.\ Math.\ {\bf 507} (2015), 179 pp.
%
%
\bibitem{Sch} Sch\"{u}tt, J.,
\emph{Symmetry groups of principal bundles over non-compact bases},
preprint, {\tt arXiv:1310.8538}.
%
%
%
%
\bibitem{Voi}
Voigt, J., \emph{On the convex compactness property for the strong operator topology},
Note Mat.,
{\bf 12} (1992), 259--269.
%
%
\bibitem{Wal}
Walter, B.,
\emph{Weighted diffeomorphism groups of Banach spaces and weighted mapping groups},
Diss.\ Math.\ {\bf 484} (2012), 126 pp.
%
%
\bibitem{Wei}
von Weizs\"acker, H., \emph{In which spaces is every curve Lebesgue-Pettis-integrable}?,
preprint, {\tt arXiv:1207.6034}.
%
%
\bibitem{Wil}
Wilansky, A.,
``Modern Methods in Topological Vector Spaces,''
McGraw--Hill, 1978.
%
%
\bibitem{Woc}
Wockel, C.,
\emph{Lie group structures on symmetry groups of principal bundles},
J. Funct.\ Anal.\ {\bf 251} (2007), 254--288. 
%
%
\end{thebibliography}
\end{document}